%% file: ms.tex
\documentclass[11pt]{article} 
\usepackage{amssymb, amsmath, amsthm, mathtools, mathrsfs} 
\usepackage[x11names]{xcolor}
\usepackage[utf8]{inputenc}
\usepackage{natbib} 
\usepackage{geometry} 
\usepackage{todonotes} 
\usepackage{enumerate} 
\usepackage{empheq} 
\usepackage{tikz-cd}
\usepackage{tcolorbox}
\usepackage{comment} 
\usepackage{appendix} 
\usepackage{import} 
\usepackage{url}
\usepackage{caption}
\usepackage{subcaption}

\usepackage{booktabs}
\usepackage{longtable}
\usepackage{array}
\usepackage{multirow}
\usepackage{wrapfig}
\usepackage{float}
\usepackage{colortbl}
\usepackage{pdflscape}
\usepackage{tabu}
\usepackage{threeparttable}
\usepackage{threeparttablex}
\usepackage[normalem]{ulem}
\usepackage{makecell}
\usepackage{xcolor}

\specialcomment{noobs}%
{\begin{tcolorbox}[before upper={\parindent15pt}, %
		colback=blue!5!white, %
		colframe=blue!50!white 
	]%
	\color{blue}\small}%
{\end{tcolorbox}} 

\excludecomment{noobs} 

\input{./math_symbols.tex}

\newcommand{\C}{\mathbb{C}} 

\graphicspath{{./plots/}}

\theoremstyle{plain}
\newtheorem{thm}{Theorem}[section]
\newtheorem{lem}[thm]{Lemma}
\newtheorem{prop}[thm]{Proposition}

\theoremstyle{definition}
\newtheorem{defn}{Definition}[section]

\theoremstyle{remark}
\newtheorem{rem}{Remark}[section]

\newtheoremstyle{named}{}{}{\itshape}{}{\bfseries}{.}{.5em}{\thmnote{#3 }#1}
\theoremstyle{named}

\theoremstyle{plain}
\newtheorem*{empiricalimp*}{Empirical Hypothesis}
\newtheorem*{thm*}{Theorem}
\newtheorem*{exercise*}{Exercise} 
\newtheorem*{example*}{Example} 
\newtheorem*{discussion*}{Discussion} 
\newtheorem*{claim*}{Claim}
\newtheorem*{lem*}{Lemma}
\newtheorem*{prop*}{Proposition}
\newtheorem*{cor*}{Corollary}
\newtheorem*{defn*}{Definition}

\theoremstyle{remark}
\newtheorem*{rem*}{Remark}
\newtheorem*{note*}{Note}
\newtheorem{case*}{Case}

\theoremstyle{plain}
\newtheorem*{algorithm*}{Algorithm}

\theoremstyle{plain}
\newtheorem{assumption}{Assumption}

\begin{document}

\title{A Small-Uniform Statistic for the Inference of Functional Linear Regressions}
\author{Raymond C. W. Leung and Yu-Man Tam}
\date{February 21, 2021}
\maketitle

\begin{abstract}
	\input{abstract}
\end{abstract} 

{\bf Keywords and phrases:} Empirical process, functional data analysis, functional linear model, functional principal components estimator, Gaussian processes, hypothesis testing, supremum.
\newpage

\input{introduction}

\input{FLM}

\input{smalluniformstat}

\input{hypothesistesting}

\input{simulations}

\input{conclusion}

\newpage
\appendix 
\appendixpage 
\addappheadtotoc

\input{proofs}

\bibliographystyle{./ecta} 
\bibliography{ms}

\end{document}

%% file: math_symbols.tex
\newcommand{\R}{\mathbb{R}} 
\newcommand{\N}{\mathbb{N}} 
\newcommand{\Z}{\mathbb{Z}} 
\newcommand{\Field}{\mathbb{F}} 
\newcommand{\im}{\iota} 

\newcommand{\Prob}{\mathbb{P}} 
\newcommand{\E}{\mathbb{E}} 
\newcommand{\Var}{\ensuremath{\mathrm{Var}}} 
\newcommand{\inner}[2]{\ensuremath{\left\langle #1,#2 \right\rangle}} 
\newcommand{\innersmall}[2]{\ensuremath{\langle #1,#2 \rangle}} 
\newcommand{\norm}[1]{\vert\vert#1\vert\vert} 
\newcommand{\abs}[1]{|#1|} 
\newcommand{\normBIG}[1]{\left\vert\left\vert#1\right\vert\right\vert} 
\newcommand{\absBIG}[1]{\left|#1\right|} 
\newcommand{\ind}{\ensuremath\boldsymbol{1}} 
\newcommand{\weakcvgto}{\ensuremath\rightsquigarrow} 
\newcommand{\BigOhPee}{\ensuremath\mathcal{O}_\Prob} 
\newcommand{\BigOh}{\ensuremath\mathcal{O}} 

\newcommand{\Efr}{\ensuremath\mathsf{erf}} 
\newcommand{\Efrc}{\ensuremath\mathsf{efrc}} 

\newcommand{\Ker}{\ensuremath\mathrm{ker}\,} 
\newcommand{\ident}{\ensuremath\mathrm{id}} 

\newcommand{\diff}[1]{\,\mathrm{d}#1}

\renewcommand{\widehat}{\hat}
\newcommand{\HSpace}{\ensuremath\mathscr{H}} 
\newcommand{\BddOp}{\ensuremath\mathscr{B}(\HSpace)} 
\newcommand{\CompactOp}{\ensuremath\mathscr{B}_0(\HSpace)} 
\newcommand{\ballH}{\ensuremath\mathrm{ball}\,\HSpace} 
\newcommand{\EmpOptDomain}{\ensuremath \widehat{\mathcal{J}}_n} 
\newcommand{\OptDomain}{\ensuremath \mathcal{J}_n} 
\newcommand{\FnOptDomain}{\ensuremath \mathcal{G}_n} 
\newcommand{\Gn}{\ensuremath \mathbb{G}_n} 

\newcommand{\KeyCvgRateStepI}{\ensuremath k_n^{11/2} (\log k_n)^{3/2}}
\newcommand{\KeyCvgRateStepII}{\ensuremath \frac{k_n^{13/2} (\log k_n)^{9/2} (\log n)}{\gamma^{1/2} n^{1/2}} + \frac{k_n^{15/4} (\log k_n)^{9/4} (\log n)^{3/4}}{\gamma^{1/2} n^{1/4}} + \frac{k_n^{17/6} (\log k_n)^{3/2} (\log n)^{2/3}}{\gamma^{1/3} n^{1/6}}  } 
\newcommand{\KeyCvgRateStepIIsimplified}{\ensuremath \frac{k_n^{13/2} (\log k_n)^{9/2} (\log n)}{n^{1/6}} }

\newcommand{\KeyCvgRateLTStatistic}{\ensuremath \frac{\KeyCvgRateStepI}{\beta_n}  + \KeyCvgRateStepIIsimplified } 
\newcommand{\SuffCondCvgRateLTStatistic}{\ensuremath \frac{k_n^{13/2} (\log k_n)^{9/2} (\log n) }{\min\{ \beta_n, n^{1/6} \}  } } 
\newcommand{\SuffCondEmpCvgRateLTStatistic}{\ensuremath \frac{\hat{k}_n^{13/2}
(\log \hat{k}_n)^{9/2} (\log n) }{\min\{ \beta_n, n^{1/6} \}  } }

%% file: abstract.tex
We propose a ``small-uniform'' statistic for the inference of the functional PCA estimator in a functional linear regression model. The literature has shown two extreme behaviors: on the one hand, the FPCA estimator does not converge in distribution in its norm topology; but on the other hand, the FPCA estimator does have a pointwise asymptotic normal distribution. Our statistic takes a middle ground between these two extremes: after a suitable rate normalization, our small-uniform statistic is constructed as the maximizer of a fractional programming problem of the FPCA estimator over a finite-dimensional subspace, and whose dimensions will grow with sample size. We show the rate for which our scalar statistic converges in probability to the supremum of a Gaussian process. The small-uniform statistic has applications in hypothesis testing. Simulations show our statistic has comparable to slightly better power properties for hypothesis testing than the two statistics of Cardot, Ferraty, Mas and Sarda (2003). 

%% file: introduction.tex
The \emph{functional linear model} (FLM) and its associated \emph{functional principal components estimator} (FPCA estimator) are now staples in the statistics literature. However, while much is known about the FPCA's mean squared error convergence and consistency properties, much less is known about its asymptotic distributional properties. In particular, although there are hypothesis testing procedures on the FLM, the literature has few hypothesis testing procedures of the FLM that are explicitly based on the FPCA slope estimate. This dearth of hypothesis testing procedures based on the estimator of the model is in stark contrast to its finite-dimensional counterpart; for instance, ordinary least squares is both an estimator of the slope and also the input of the $t$-tests, $F$-tests and many others in tests of the finite-dimensional linear model. 

This paper has two main objectives. Firstly, we introduce a \emph{small-uniform statistic} that is constructed out of a normalized fractional programming problem of the FPCA estimator. Theorem~\ref{thm:NormalizedLeungTamStatisticConvergence} is the main result of this paper and shows our small-uniform statistic converges in probability to a supremum of a Gaussian process. This result is the basis for a hypothesis testing procedure that explicitly depends on the FPCA estimator. Secondly, we show in numerical simulations the hypothesis testing procedure based off of our small-uniform statistic has comparable to slightly better power properties than the two statistics proposed in \cite{cardot2003testing}.

The key references of our paper are \cite{cardot2007clt, cardot2003testing} and \cite{chernozhukov2014gaussian}. In particular, \cite{cardot2003testing} and \cite{hilgert2013minimax} are one of the first few studies for conducting hypothesis testing on the FLM. However, as far as we understand, none of these studies base their hypothesis testing procedure on the FPCA estimator. Recently, \cite{cuesta2019goodness} has proposed an interesting goodness-of-fit test of the FLM based on random projections, and a step in its testing procedure does indeed depend on the FPCA estimator. Roughly speaking, the testing procedure of \cite{cuesta2019goodness} is dependent on a single randomly drawn vector (i.e.\ a ``direction'') of the functional regressors' underlying Hilbert space. To smooth out the uncertainty in just drawing a single direction, the authors recommend drawing multiple directions to thus conduct several hypothesis tests, and the final inference step is concluded by a multiple hypothesis testing correction (see their Algorithms 4.1 and 4.2). In contrast and intuitively, our small-uniform statistic considers finitely many (but that number increases with the sample size) of these directions, and then look for the ``largest'' direction. Thus our small-uniform statistic is a single scalar and does not require multiple hypothesis testing corrections. \cite{ramsay2005functional} is the well-known seminal survey of the functional data analysis (FDA) literature. \cite{cardot2011functional}, \cite{horvath2012inference}, \cite{hsing2015theoretical}, \cite{goia2016introduction} and \cite{wang2016functional} are some recent surveys on the advancements of the FDA literature. 

Section~\ref{sec:FLM} fixes notations for the FLM and reviews the two extreme asymptotic behavior of the FPCA estimator as documented by \cite{cardot2007clt}. Section~\ref{sec:SmallUniformStatistic} introduces our small-uniform statistic. Section~\ref{sec:HypothesisTesting} outlines the hypothesis testing procedure based off of our small-uniform statistic, and Section~\ref{sec:NumericalSimulations} shows some simulated numerical results. We conclude in Section~\ref{sec:Conclusion}. The proofs are technical in nature and thus we gather them in the Supplementary Materials \cite{leung2021suppsmalluniform}.

%% file: FLM.tex
\section{Functional linear model} 
\label{sec:FLM}
Let's begin with the standard \emph{functional linear model}. Throughout this paper, we will fix a sufficiently rich probability space $(\Omega, \mathcal{F}, \Prob)$ that accommodates all the random quantities in this paper. Let $\HSpace$ be an arbitrary real separable infinite dimensional Hilbert space equipped with an inner product $\inner{\cdot}{\cdot}$ and denote its norm as $\norm{\cdot}$. Let, 
\begin{equation} 
	Y = \inner{\rho}{X} + \varepsilon,
	\label{eq:TheModel} 
\end{equation}
where $Y$ is a real valued scalar dependent variable, $X$ is $\HSpace$-valued random element, and $\rho$ is an $\HSpace$-valued coefficient vector. Moreover, $\varepsilon$ is a scalar error term 
such that $\E[ \varepsilon | X ] = 0$ and $\E[ \varepsilon^2 | X ] = \sigma_\varepsilon^2$. We are interested in the estimation and subsequent inference of the coefficient vector $\rho$. 

Let's define the usual covariance and cross-covariance operators. For any $x_1, x_2 \in \HSpace$, we denote their \emph{tensor product} as $x_1 \otimes x_2(h) := \inner{x_1}{h} x_2$ for all $h \in \HSpace$. We denote the \emph{covariance operator} of $X$ as $\Gamma : \HSpace \to \HSpace$, 
\begin{equation}
	\Gamma h := \E[ X \otimes X(h) ], \quad h \in \HSpace  
	\label{eq:CovarianceOperator} 
\end{equation}
and define the \emph{cross-covariance operator} of $X$ and $Y$ as $\Delta : \HSpace \to \R$,
\begin{equation}
	\Delta h := \E[ X \otimes Y(h) ], \quad h \in \HSpace 
\end{equation}
We denote $\{ \lambda_j \}_{j \in \Z^+}$ as the sequence of sorted non-null distinct eigenvalues of $\Gamma$, $\lambda_1 > \lambda_2 > \ldots > 0$, and $\{ e_j \}_{j \in \Z^+}$ a sequence of orthonormal eigenvectors associated with those eigenvalues. We assume the multiplicity of each $\lambda_j$ is one. From \eqref{eq:TheModel} we have normal equation, 
\begin{equation} 
	\Delta = \Gamma \rho.
	\label{eq:MomentCond} 
\end{equation}

For the $\HSpace$-valued random element $X$, there is the well-known \emph{Karhunen-Lo\`{e}ve} expansion of $X$ and is given by, 
\begin{equation}
	X = \sum_{l=1}^\infty \sqrt{\lambda_l} \xi_l e_l, 
	\label{eq:KLExpansion} 
\end{equation}
where $\xi_l$'s are centered real random variables such that $\E[ \xi_l \xi_{l'}] = 1$ if $l = l'$ and $0$ otherwise.

\subsection{Estimation and Assumptions} 
\label{sec:Estimation}
This section will revisit some of the key definitions and setup from \cite{cardot2007clt}. Suppose we have have $n$ independent and identically distributed observations $\{(Y_i, X_i)\}_{i = 1}^n$ of \eqref{eq:TheModel}. We construct the empirical counterparts of $\Gamma$ and $\Delta$ as, 
\begin{subequations}
	\begin{align}
		\Gamma_n &:= \frac{1}{n} \sum_{i = 1}^n X_i \otimes X_i, \\ 
		\Delta_n &:= \frac{1}{n} \sum_{i = 1}^n X_i \otimes Y_i, \\ 
		U_n &:= \frac{1}{n} \sum_{i = 1}^n X_i \otimes \varepsilon_i.
	\end{align}
	\label{eq:EmpiricalOperators} 
\end{subequations}
Then from \eqref{eq:TheModel}, we get the empirical normal equation
\begin{equation} 
	\Delta_n = \Gamma_n \rho + U_n.
	\label{eq:EmpiricalMomentCond} 
\end{equation} 
We denote the $j$th empirical eigenelement of $\Gamma_n$ as $(\hat{\lambda}_j, \hat{e}_j)$.

As is well known in the FLM literature, we will need some sort of regularization method to define an ``approximate inverse'' to $\Gamma_n$. We will again follow the setup of \cite{cardot2007clt} and \cite{bosq2012linear} and define the sequence $\delta_j$'s, $j = 1, 2, \ldots$ of the smallest difference between distinct eigenvalues of $\Gamma$ as, 
\begin{subequations}
	\begin{align}
		\delta_1 &:= \lambda_1 - \lambda_2, \\ 
		\delta_j &:= \min\{ \lambda_j - \lambda_{j+1}, \lambda_{j - 1} - \lambda_j \}.
	\end{align} 
	\label{eq:Radii} 
\end{subequations}
Now take $\{ c_n \}_{n \in \N}$ a sequence of strictly positive numbers tending to zero such that $c_n < \lambda_1$ and set, 
\begin{equation}
	k_n := \sup\{ p : \lambda_p + \delta_p / 2 \ge c_n \}. 
	\label{eq:Truncation} 
\end{equation}
This $k_n$ will be our \emph{truncation parameter}; note when $n \to \infty$ we have $c_n \to 0$, which then implies $k_n \uparrow \infty$.

\begin{noobs} 
	Let's make clear on the domain and ranges of these operators: 
	\begin{align*}
		\Gamma_n &: \HSpace \to \HSpace \\ 
		\Delta_n &: \HSpace \to \R \\ 
		U_n &: \HSpace \to \R 
	\end{align*}
	and recalling $\rho \in \HSpace$, we read the moment condition \eqref{eq:EmpiricalMomentCond} as for any $h \in \HSpace$, 
	\begin{equation*}
		\Delta_n h 
		\equiv \inner{\Delta_n}{h}  
		= \inner{ \Gamma_n \rho + U_n}{h} 
		= \inner{\Gamma_n \rho}{h} + \inner{U_n}{h} 
		= \inner{\Gamma_n \rho}{h} + U_n(h) 
	\end{equation*}
	In particular, we take on the conventional notation in functional analysis where for a linear operator $L$, we can freely denote $L(h)$ or $\inner{L}{h}$ equivalently. 
\end{noobs} 

Let's gather the assumptions of our paper here. Unless noted otherwise, we will enforce these assumptions throughout the paper's results and proofs.
\begin{assumption}[Identifiability]
	\label{as:ConditionH}
	\hfill
	\begin{enumerate}[(i)]
		\item $\sum_{j = 1}^\infty \frac{\inner{\E[XY]}{e_j}^2 }{\lambda_j^2} < \infty$; and

		\item $\Ker\Gamma = \{ 0 \}$.
	\end{enumerate} 
\end{assumption}

\begin{assumption}[Tail behavior]
	\label{as:ConditionA}
	\hfill
	\begin{enumerate}[(i)]
		\item $\sum_{l = 1}^\infty |\inner{\rho}{e_l}| < \infty$;

		\item There exists some finite $M$ such that $\sup_l \E[\xi_l^6] \le M < \infty$; and

		\item There exists a convex positive function $\lambda$ such that for $j$ sufficiently large, $\lambda_j = \lambda(j)$.
	\end{enumerate}
\end{assumption}

\begin{assumption}[Approximate reciprocal]
	\label{as:ConditionF} 
	\hfill
	\begin{enumerate}[(i)]
		\item $f_n$ is decreasing on $[c_n, \lambda_1 + \delta_1]$; 

		\item $\lim_{n \to \infty} \sup_{x \ge c_n} | x f_n(x) - 1 | = 0$;

		\item $f_n'(x)$ exists for $x \in [c_n, \infty)$; and

		\item $\sup_{s \ge c_n} \abs{s f_n(s) - 1} = o\left( \frac{1}{\sqrt{n}} \right)$.
	\end{enumerate}
\end{assumption} 

\begin{assumption}[Roughening the standard deviation]
	\label{as:ConditionR}
	\hfill
	There exists a sequence of positive numbers $\{a_n\}$ such that $a_n \to 0$ and $a_n \sqrt{k_n \log k_n } \to 0$, as $n \to \infty$; and
\end{assumption}

\begin{assumption}[Empirical eigenvector approximations]
	\label{as:ConditionE}
	Assume $k_n$ is such that $\frac{1}{\lambda_{k_n} - \lambda_{k_n + 1}} = \BigOh(n^{1/2})$. 
\end{assumption}
Assumption~\ref{as:ConditionH} is a basic identifiability condition in a functional linear model and these conditions are discussed in detail in \cite{cardot1999functional} and \cite{cardot2003testing}. Assumption~\ref{as:ConditionA} corresponds to Assumption~A of \cite{cardot2003testing} which are basic conditions that ensure the statistical problem is correctly posed. For our purposes, however, we replace \cite{cardot2007clt}'s finite fourth moment assumption on the $\xi_l$'s with a stronger finite sixth moment assumption. Assumption~\ref{as:ConditionF} corresponds to Assumption~F of \cite{cardot2003testing} which effectively says the sequence of functions $\{f_n\}$ should behave like $f_n(x) \approx 1 / x$ when $n$ is sufficiently large. Assumption~\ref{as:ConditionR} is new: it says $\{ a_n \}$ is a regularization that tends to zero, and more importantly, tends to zero faster than the reciprocal of the eigenvalues tending to infinity. Assumption~\ref{as:ConditionE} will be used to ensure the empirical eigenvectors of the empirical covariance operator uniformly converges in probability to the population eigenvectors of population covariance operator. 

At this point, we will need to use the \emph{resolvent formalism} to define an object $\Gamma_n^\dagger$ which will serve as our ``approximate empirical inverse'' to $\Gamma_n$. For the purpose of exposition, we delegate the definition and details of this object to the supplementary materials. To construct $\Gamma_n^\dagger$, we will need a sequence of positive functions $\{ f_n \}_{n \in \N}$ with support on $[c_n, \infty)$ that satisfy Assumption~\ref{as:ConditionF}. Intuitively, the functions $f_n$ have the behavior of $f_n(x) \approx 1/x$ when $n$ is sufficiently large. By \emph{Riesz functional calculus}, we can define the following quantity (see supplementary materials \cite[\eqref{eq:EmpiricalProjections}]{leung2021suppsmalluniform} for details), 
\begin{equation} 
	\Gamma_n^\dagger 
	:= f_n(\Gamma_n). 
	\label{eq:EmpInverseCovarianceOperator} 
\end{equation}
In particular, $\Gamma_n^\dagger$ will serve as the approximate inverse of $\Gamma_n$. We will also let $\widehat{\Pi}_{k_n}$ denote the projection operator from $\HSpace$ onto $\mathrm{span}\{ \hat{e}_1, \ldots, \hat{e}_{k_n} \}$, which is subspace of all possible linear combinations of the first $k_n$ empirical eigenvectors (equation \eqref{eq:EmpiricalProjections} in the supplementary materials \cite{leung2021suppsmalluniform} will define $\widehat{\Pi}_{k_n}$ precisely via Riesz functional calculus). 

Finally, a natural estimator of $\rho$ from $n$ iid observations based on \eqref{eq:MomentCond} and \eqref{eq:EmpiricalMomentCond} is the \emph{functional principal components (FPCA) estimator},
\begin{equation} 
	\hat{\rho} := \Gamma_n^\dagger \Delta_n. 
	\label{eq:RhoHat} 
\end{equation} 
\cite{cardot1999functional} shows this estimator is consistent for the choice of $f_n(x) \equiv 1 / x$.

\begin{noobs}
	Let's make clear on how exactly do those ``abstract'' operator definitions $\hat{\rho}$ exactly compute out explicitly for actual numerical computation purposes. Fix any $h \in \HSpace$. Then 
	\begin{align*}
		\hat{\rho}(h)
		&= \Gamma_n^\dagger \Delta_n(h) \\ 
		&\equiv \Delta_n \circ \Gamma_n^\dagger(h) \\ 
		&= \Delta_n \circ \left(  \sum_{j = 1}^{k_n} f_n(\hat{\lambda}_j) \hat{e}_j \otimes \hat{e}_j (h) \right) \\ 
		&= \sum_{j = 1}^{k_n} f_n(\hat{\lambda}_j) \inner{\hat{e}_j}{h} \Delta_n(\hat{e}_j) \\ 
		&= \sum_{j = 1}^{k_n} f_n(\hat{\lambda}_j) \inner{\hat{e}_j}{h}  \left( \frac{1}{n} \sum_{i = 1}^n X_i \otimes Y_i (\hat{e}_j) \right) \\
		&= \frac{1}{n} \sum_{j = 1}^{k_n} \sum_{i = 1}^n f_n(\hat{\lambda}_j) \inner{\hat{e}_j}{h} \inner{X_i}{\hat{e}_j} Y_i
	\end{align*}
	This is the numerically computable form of the estimator. In particular, this explicitly shows that $\hat{\rho}$ eats an element from $\HSpace$ and returns $\R$, so meaning $\hat{\rho} \in \HSpace^*$. By Riesz representation theorem and some abuse of notations, we can thus denote $\hat{\rho}(h) = \inner{\hat{\rho}}{h}$. 
\end{noobs}

\subsection{Motivation of our paper} 
The motivation of our paper starts from two key insights from \cite{cardot2007clt}. Their first key result (also more recently \cite[Theorem 8]{crambes2013asymptotics}) is that the FPCA estimator \eqref{eq:RhoHat} cannot converge in distribution to a non-degenerate random element in the norm topology of $\HSpace$. 
\begin{thm*}[\cite{cardot2007clt}, Theorem 1]   
	It is impossible for $\hat{\rho} - \rho$ to converge in distribution to a non-degenerate random element in the norm topology of $\HSpace$. 
\end{thm*}
This impossibility result suggests that we may not directly use the FPCA estimator for the purpose of inference in the norm topology of $\HSpace$. In contrast, uniform prediction intervals can still be constructed (see concluding remarks of \cite{cardot2007clt} and \cite[Corollaries 10 and 11]{crambes2013asymptotics}). 

Their second result (see also more recently \cite[Theorem 9]{crambes2013asymptotics}) shows the following pointwise weak convergence result. 
\begin{thm*}[\cite{cardot2007clt}, Theorem 3] 
	Fix any $x \in \HSpace$. Then under the same Assumptions~\ref{as:ConditionA} to \ref{as:ConditionF} of our paper, and under additional regularity conditions (see their paper for details), 
	\begin{equation*} 
		\frac{\sqrt{n}}{\norm{\Gamma^{1/2}\Gamma^\dagger x} \sigma_\varepsilon} ( \innersmall{\hat{\rho}}{x} - \innersmall{\widehat{\Pi}_{k_n} \rho}{x} ) \weakcvgto \mathcal{N}(0,1). 
	\end{equation*} 
	\label{thm:Cardot2007Thm3} 
\end{thm*}
For the sake of exposition, we will defer the precise definition of $\Gamma^\dagger$ to the supplementary materials (see \cite[\eqref{eq:GammaDagger}]{leung2021suppsmalluniform}), but we can intuitively think of this quantity as an ``approximate inverse'' of the population covariance operator $\Gamma$. This result is extremely useful for constructing \emph{prediction intervals} when we evaluate at $x = X_{n+1}$. However, the rather arbitrary choice of $x \in \HSpace$ renders this result impractical when the researcher is concerned with the statistical inference. 

The main contribution of this paper can be thought of as ``something in between'' Theorem 1 and Theorem 3 of \cite{cardot2007clt}. This paper focuses on the study on a scalar ``partial'' supremum statistic $W_n$ to be defined in \eqref{eq:LeungTamStatistic}. For the sake of heuristics in this section, we will slightly blur the distinction between the empirical eigenelements and the population eigenelements (see Remark~\ref{rem:EmpiricalFeasibility} for the validity of this justification). Let's make three observations. 

The first observation is that there is no need to consider points $x$ in all of $\HSpace$ in \cite[Theorem 3]{cardot2007clt}. Provided $x \neq 0$, we can multiply and divide by $\frac{\epsilon}{\norm{x}}$, where $\epsilon \in (0,1]$, and so we can rewrite as, 
\begin{align}
	\frac{\sqrt{n}}{\sigma_\varepsilon} \frac{\innersmall{\hat{\rho} - \widehat{\Pi}_{k_n}\rho}{x} }{ \norm{ \Gamma^{1/2}\Gamma^\dagger x} }  
	= \frac{\sqrt{n}}{\sigma_\varepsilon} \frac{\innersmall{\hat{\rho} - \widehat{\Pi}_{k_n}\rho}{ \frac{\epsilon x}{\norm{x}}} }{ \normBIG{ \Gamma^{1/2}\Gamma^\dagger \frac{\epsilon x}{\norm{x}} } } 
	\label{eq:PreLeungTamStatistic} 
\end{align}
Of course, $\norm{ \frac{\epsilon x}{\norm{x}} } = \epsilon \in (0,1]$. So rather than considering all points in $\HSpace$, we can immediately confine to those points in $\ballH := \{ h \in \HSpace : \norm{h} \le 1\}$. 

Secondly, we can say a lot more about \cite[Theorem 3]{cardot2007clt} by restricting $\ballH$ further. The main idea is to consider not all points in $\ballH$, but consider a ``small but growing'' linear subspace of it. In the numerator of \eqref{eq:PreLeungTamStatistic}, since $\hat{\rho} - \widehat{\Pi}_{k_n}\rho \in \mathrm{span}\{ \hat{e}_1, \ldots, \hat{e}_{k_n} \}$, by the idempotent property of the projection operator, it follows for any $x \in \ballH$ (or in $\HSpace$) we have $\innersmall{ \hat{\rho} - \hat{\Pi}_{k_n}\rho }{x} = \innersmall{ \widehat{\Pi}_{k_n}(\hat{\rho} - \widehat{\Pi}_{k_n} \rho) }{x} = \innersmall{ \hat{\rho} - \widehat{\Pi}_{k_n}\rho }{ \widehat{\Pi}_{k_n} x}$. Thus only points in $\mathrm{span}\{ \hat{e}_1, \ldots, \hat{e}_{k_n} \} \approx \mathrm{span}\{ e_1, \ldots, e_{k_n} \}$ determine the numerator of \eqref{eq:PreLeungTamStatistic}. Next let's consider the denominator of \eqref{eq:PreLeungTamStatistic}. By the spectral decompositions of $\Gamma^{1/2}$ and $\Gamma^\dagger$,  
\begin{equation*}
	\Gamma^{1/2} \Gamma^\dagger 
	= \sum_{j = 1}^\infty \sum_{l = 1}^{k_n} \sqrt{\lambda_j} f_n(\lambda_l) P_j P_l  
	= \sum_{j = 1}^{k_n} \sqrt{\lambda_j} f_n(\lambda_j) P_j 
\end{equation*}
where $P_j$ is the projection of $\HSpace$ onto the $j$th eigenspace $\Ker (\Gamma - \lambda_j)$. More explicitly, since these orthogonal projections partition $\HSpace$, we can write any $x \in \ballH$ as $x = \sum_{j = 1}^\infty P_j x$. And since $\Ker(\Gamma - \lambda_j) \perp \Ker(\Gamma - \lambda_{j'})$ for $j \neq j'$, this implies, 
\begin{equation*}
	\Gamma^{1/2} \Gamma^\dagger x 
	= \sum_{j = 1}^{k_n} \sqrt{\lambda_j} f_n(\lambda_j) \sum_{l = 1}^\infty P_j P_l x  
	= \sum_{j = 1}^{k_n} \sqrt{\lambda_j} f_n(\lambda_j) P_j x 
\end{equation*}
In other words, picking any $x \in \ballH$ with $x = \sum_{j = 1}^\infty P_j x$ versus picking $h \in \ballH$ with $h = \sum_{j = 1}^{k_n} P_j h$ results in the same value, $\Gamma^{1/2} \Gamma^\dagger x = \Gamma^{1/2} \Gamma^\dagger h$. And since $\HSpace$ (and hence $\ballH$) is assumed to be separable, we can simply assume that such $h$ takes the form $h = \sum_{j = 1}^{k_n} b_j e_j$ with $\sum_{j = 1}^{k_n} b_j^2 \le 1$. In all, we argue it suffices to evaluate $\norm{\Gamma^{1/2} \Gamma^\dagger \cdot}$ on the finite-dimensional domain $\ballH \cap \mathrm{span}\left\{ e_1, \ldots, e_{k_n} \right\}$ instead of on the infinite-dimensional domain $\ballH$. 

Thirdly, instead of considering $\sigma_\varepsilon \norm{\Gamma^{1/2} \Gamma^\dagger h}$ as the asymptotic standard deviation of $\innersmall{\hat{\rho} - \widehat{\Pi}_{k_n}\rho}{h}$, let's use a slightly roughened version and define 
\begin{equation}
	t_n(h) 
	:= \norm{ \Gamma^{1/2} \Gamma^\dagger h } + a_n  
	= \sqrt{ \sum_{j = 1}^{k_n} \lambda_j [f_n(\lambda_j)]^2 \inner{h}{e_j}^2 } + a_n 
	\label{eq:StdDev} 
\end{equation}
where we let $\{ a_n \}$ be a sequence of nonnegative numbers tending to zero. Note and recall that $t_n$ depends on $n$ not just through $a_n$ but also through $\Gamma^\dagger$ which depends on $k_n$. Assumption~\ref{as:ConditionR} ensures this roughening sequence tends to zero at a rate slower than the rate for which the sequence of eigenvalues tend to zero.

%% file: smalluniformstat.tex
\section{A small-uniform statistic} 
\label{sec:SmallUniformStatistic}
Finally, let's put our above observations together. In search for a single scalar statistic, it seems reasonable to look for the largest value of \eqref{eq:PreLeungTamStatistic} over the finite-dimensional domain $\ballH \cap \mathrm{span}\{ e_1, \ldots, e_{k_n} \}$. We thus have the following definition.
\begin{defn}[Small-uniform statistic]
	Let $\hat{\rho}$ be the FPCA estimator \eqref{eq:RhoHat} of the functional linear model \eqref{eq:TheModel} and let $\{ \beta_n \}$ be a sequence of positive numbers with $\beta_n \to \infty$ as $n \to \infty$. Define 
	\begin{equation} 
		\begin{aligned} 
			W_n &:=  \frac{\sqrt{n}}{\sigma_\varepsilon \beta_n} \sup_{h \in \OptDomain} \frac{\innersmall{\hat{\rho} - \widehat{\Pi}_{k_n}\rho}{h}}{t_n(h) } \\ 
			\OptDomain &:= \ballH \cap \mathrm{span}\{ e_1, \ldots, e_{k_n} \}
		\end{aligned} 
		\label{eq:LeungTamStatistic} 
	\end{equation} 
	where $t_n$ is defined in \eqref{eq:StdDev}. We call $W_n$ the \emph{small-uniform} statistic of the functional linear model. 
\end{defn} 
The real-valued scalar statistic $W_n$ is ``small'' because we only consider a low and finite-dimensional linear subspace $\OptDomain$ of $\HSpace$, even though as $n$ becomes large this subspace approaches $\ballH$. It is ``uniform'' because we look for the largest value over this linear subspace $\OptDomain$.

Recall again \cite[Theorem 3]{cardot2007clt} already shows the pointwise asymptotic normality result of the FPCA estimator. Thus under some regularity conditions and a proper rate normalization, one can expect $W_n$ to distribute like the supremum of a Gaussian process indexed by $\OptDomain$. Indeed our main result Theorem~\ref{thm:NormalizedLeungTamStatisticConvergence} shows precisely the rate of convergence under which $W_n$ and a certain Gaussian process converge to each other in probability, and hence also in distribution. Note by linearity in $h$ in the numerator of \eqref{eq:LeungTamStatistic} and as the denominator $t_n$ is strictly positive, the statistic $W_n$ is almost surely nonnegative valued. 

\begin{rem}[Rate normalization] 
	\label{rem:LTStatNormalization} 
	The normalization $1 / \beta_n$ in \eqref{eq:LeungTamStatistic} might seem curious. The normalization by $\sqrt{n}$ is standard, and is well expected by the pointwise asymptotic normality result of \cite[Theorem 3]{cardot2007clt}. The normalization by $1 / \beta_n$ is necessary because we need this rate to ensure some ``nuisance terms'' in $W_n$ converge fast enough to zero. See the proof outline of our main result Theorem~\ref{thm:NormalizedLeungTamStatisticConvergence} for further explanations. 

	In addition, our statistic is a fractional programming problem (see \cite{stancu2012fractional} for a survey). So while $\sigma_\varepsilon \norm{\Gamma^{1/2} \Gamma^\dagger h}$ is indeed the pointwise asymptotic standard deviation for $\sqrt{n} \innersmall{ \hat{\rho} - \widehat{\Pi}_{k_n}\rho }{h}$, we clearly see this standard deviation evaluates to zero at $h = 0$. Using a roughened version $t_n(h)$ of the standard deviation ensures the denominator of our statistic is strictly positive. 
\end{rem} 

\begin{rem}[Existence] 
	\label{rem:Existence} 
	The optimization problem in $W_n$ is well-defined. The objective function is clearly continuous in $\HSpace$, especially since by construction $t_n > 0$. Moreover we're optimizing over $\OptDomain$, which is a compact set
	\footnote{
		Clearly $\ballH$ is bounded, and a finite-dimensional subspace in an infinite-dimensional Hilbert space is closed (in the relative topology). Thus the Heine-Borel theorem applies and so $\OptDomain$ is compact.
	}
	, and so the extreme value theorem applies.
\end{rem} 

\begin{rem}[Empirically feasible form of $W_n$] 
	\label{rem:EmpiricalFeasibility}
	As \eqref{eq:LeungTamStatistic} is written, it is an empirically infeasible quantity for several reasons. Let's argue why putting in empirically feasible plug-in estimates will asymptotically do no harm to our results.
	\begin{enumerate}[(a)]
		\item (\emph{Replacing the truncation parameter}) The truncation parameter $k_n$ as defined in \eqref{eq:Truncation} depends on the unobservable population eigenvalues $\lambda_j$'s. The natural substitute is the empirical truncation 
			\begin{equation}
				\hat{k}_n := \max \{ p = 1, \ldots, n : \hat{\lambda}_p + \hat{\delta}_p / 2 \ge c_n \},
				\label{eq:EmpTruncation}
			\end{equation}
			where $\hat{\delta}_j$ is as analogously defined to its population counterpart in \eqref{eq:Radii} but with the empirical eigenvalues. Thanks to Assumption~\ref{as:ConditionA}(ii) and \cite[\S 4.2, Theorem 4.4]{bosq2012linear}, we have $\sup_{j \ge 1} \abs{\hat{\lambda}_j - \lambda_j} \to 0$ almost surely. Hence for sufficiently large sample sizes, using the empirical truncation $\hat{k}_n$ or population truncation $k_n$ are equivalent in probability. 

		\item (\emph{Replacing the optimization domain}) The optimization domain as defined in \eqref{eq:LeungTamStatistic} is over the unobservable population eigenvectors $e_j$'s. The natural empirically feasible approach is to optimize instead over the empirical eigenvectors $\hat{e}_j$'s. By Assumption~\ref{as:ConditionE}, \cite[\S 4.2, Corollary 4.3]{bosq2012linear} ensures $\E\left[ \sup_{1 \le j \le k_n} \norm{\hat{e}_j - e_j' }^2 \right] \to 0$ as $n \to \infty$, and where we have denoted $e_j' := \mathrm{sign}(\inner{\hat{e}_j}{e_j}) e_j$ and where $\mathrm{sign}(t) = 1$ if $t > 0$, $= 0$ if $t = 0$, and $= -1$ if $t < 0$. This implies optimizing over $\EmpOptDomain := \ballH \cap \mathrm{span}\{ \hat{e}_1, \ldots, \hat{e}_{k_n} \}$ and $\ballH \cap \mathrm{span}\{ e_1', \ldots, e_{k_n}' \}$ are asymptotically equivalent in probability. By fixing the ``orientations'' $\mathrm{sign}(\inner{\hat{e}_j}{e_j})$'s, we can identify optimizing $\ballH \cap \mathrm{span}\{ e_1', \ldots, e_{k_n}' \}$  with optimizing over $\OptDomain$.

		\item (\emph{Replacing the asymptotic standard deviation}) The asymptotic standard deviation $t_n$ as defined in \eqref{eq:StdDev} depends on the unobservable population eigenvalues $\lambda_j$'s and eigenvectors $e_j$'s. An empirically feasible version of $t_n$ is its natural plug-in estimator, 
			\begin{equation}
				\hat{t}_n(h) = 
				\sqrt{ \sum_{j = 1}^{\hat{k}_n} \hat{\lambda}_j [f_n(\hat{\lambda}_j)]^2 \inner{h}{\hat{e}_j}^2 } + a_n,
				\quad h \in \EmpOptDomain
				\label{eq:StdDevEstimator} 
			\end{equation}
			By using the arguments in (a) and (b) above, it is not difficult to see that $t_n$ and $\hat{t}_n$ are asymptotically the same in probability. See also \cite[Corollary 2]{cardot2007clt}.

		\item (\emph{Consistent estimate of noise error}) It is clear the standard deviation of the error term $\sigma_\varepsilon$ can be replaced by any consistent estimator $\hat{\sigma}_\varepsilon \overset{\Prob}{\to} \sigma_\varepsilon$. 
	\end{enumerate}

\end{rem}

Except for Sections~\ref{sec:HypothesisTesting} and \ref{sec:NumericalSimulations} where we discuss numerical simulations, the rest of this section and the proofs will use $W_n$ as defined by \eqref{eq:LeungTamStatistic}.

\subsection{Main result} 

This is the paper's main result. The proof outline sketches the two key steps to proving our result. We delegate all the proof details to the supplementary materials \cite{leung2021suppsmalluniform}. For an arbitrary set $T$, we denote $\ell^\infty(T)$ as the space of all bounded functions from $T$ to $\R$ with the uniform norm $\norm{f}_T := \sup_{t \in T} \abs{f(t)}$.
\begin{thm}[Gaussian suprema approximation of the small-uniform statistic] 
	\label{thm:NormalizedLeungTamStatisticConvergence} 
	Assume Assumptions~\ref{as:ConditionH} to \ref{as:ConditionR} hold and assume $k_n / n \to 0$ as $n \to \infty$. Then for sufficiently large $n$, there exists a mean-zero Gaussian process $\{ G_{P,n}(h) \}_{h \in \OptDomain}$ in $\ell^\infty(\mathcal{J}_n)$ with covariance function, 
	\begin{equation}
		\E[G_{P,n}(h_1) G_{P,n}(h_2) ] 
		= \frac{ \inner{ \Gamma^{1/2} \Gamma^\dagger h_1}{ \Gamma^{1/2} \Gamma^\dagger h_2 }}{ (\norm{\Gamma^{1/2}\Gamma^\dagger h_1} + a_n) \, (\norm{\Gamma^{1/2}\Gamma^\dagger h_2} + a_n)   },
		\label{eq:SupZnGaussianApproximationCovFun} 
	\end{equation}
	for all $h_1, h_2 \in \OptDomain$.

	Moreover, if we define the random variables 
	\begin{equation*}
		\widetilde{Z}_n := \sup_{h \in \OptDomain} G_{P,n} h \\
		\quad\text{and}\quad
		\widetilde{W}_n := \frac{\widetilde{Z}_n}{\beta_n},
	\end{equation*}
	then the small-uniform statistic $W_n$ of \eqref{eq:LeungTamStatistic} and the random variable $\widetilde{W}_n$ are close together in probability at the rate 
	\begin{equation}
		\abs{ W_n - \widetilde{W}_n } 
		= \BigOhPee\left( \KeyCvgRateLTStatistic \right). 
		\label{eq:NormalizedLTStatisticCvgRate} 
	\end{equation}
	In particular if $\SuffCondCvgRateLTStatistic \to 0$, then 
	\begin{equation*}
		\abs{W_n - \widetilde{W}_n} \overset{\Prob}{\to} 0.
	\end{equation*}
\end{thm}

\begin{proof}[Proof outline] 
	For each $h \in \OptDomain$ we have the important decomposition, 
	\begin{align}
		\frac{\sqrt{n}}{\sigma_\varepsilon t_n(h)}\innersmall{ \hat{\rho} - \widehat{\Pi}_{k_n}\rho}{h} = \frac{\sqrt{n}}{\sigma_\varepsilon t_n(h)} \innersmall{\mathcal{T}_n + \mathcal{S}_n + \mathcal{Y}_n + \mathcal{R}_n}{h}.   
		\label{eq:RhoHatDecomposition} 
	\end{align}
	where 
	\begin{subequations}
		\begin{align}
			\mathcal{T}_n &:= (\Gamma_n^\dagger \Gamma_n - \Pi_{k_n} ) \rho, \\ 
			\mathcal{S}_n &:= (\Gamma_n^\dagger - \Gamma^\dagger) U_n, \\ 
			\mathcal{Y}_n &:= (\Pi_{k_n} - \widehat{\Pi}_{k_n})\rho, \\
			\mathcal{R}_n &:= \Gamma^\dagger U_n. 
		\end{align}
		\label{eq:RhoHatDecompositionOperators} 
	\end{subequations}
	For the sake of exposition, we defer the precise functional calculus definitions of the bounded operators $\Gamma_n^\dagger, \Gamma_n, \Gamma^\dagger$ and $\Pi_{k_n}$ to the supplementary materials. 

	Then by triangle inequality, we have
	\begin{align*}
		&\absBIG{ \sup_{h \in \OptDomain} \frac{\sqrt{n}}{\sigma_\varepsilon t_n(h)}\innersmall{ \hat{\rho} - \widehat{\Pi}_{k_n}\rho}{h}  - \widetilde{Z}_n } \\
		&\le \sup_{h \in \OptDomain} \absBIG{ \frac{\sqrt{n}}{\sigma_\varepsilon t_n(h)} \innersmall{\mathcal{T}_n + \mathcal{S}_n + \mathcal{Y}_n}{h}  } 
		+ \absBIG{ \sup_{h \in \OptDomain} \frac{\sqrt{n}}{\sigma_\varepsilon t_n(h)} \inner{\mathcal{R}_n }{h} - \widetilde{Z}_n }. 
	\end{align*}

	The two major steps in the proof are showing the following results for sufficiently large $n$: 
	\begin{enumerate}[\itshape Step I:] 
		\item Asymptotic bias terms
			\begin{equation}
				\sup_{h \in \OptDomain} \absBIG{ \frac{\sqrt{n}}{\sigma_\varepsilon t_n(h)} \inner{\mathcal{T}_n + \mathcal{S}_n + \mathcal{Y}_n}{h}  }	
				= \BigOhPee\left( \KeyCvgRateStepI \right). 
				\label{eq:ProofStepI} 
				\tag{I} 
			\end{equation}
			
		\item Asymptotic distribution term
			\begin{equation}
				\absBIG{ \frac{\sqrt{n}}{\sigma_\varepsilon t_n(h)}\sup_{h \in \OptDomain} \inner{\mathcal{R}_n}{h} - \widetilde{Z}_n } 
				= \BigOhPee\left( \KeyCvgRateStepIIsimplified \right).
				\label{eq:ProofStepII} 
				\tag{II} 
			\end{equation}
	\end{enumerate}
	Step~\eqref{eq:ProofStepI} uses many proof arguments from \cite{cardot2007clt} but we take extra care in keeping track of the rates of various bounds. Proposition~\ref{prop:SupTYSConvergesToZero} of the supplementary materials concludes the discussions of Step~\eqref{eq:ProofStepI}. By our underlying real valued Hilbert space structure, we can apply Riesz's representation theorem to uniquely identify $\HSpace$ with its dual $\HSpace^*$. Thus we can view the indexing of the supremum of $\mathcal{R}_n$ by $\OptDomain$ in Step~\eqref{eq:ProofStepII} as equivalent to indexing by its dual $\OptDomain^*$, which allows us to apply the tools from empirical process theory. Our desired result for Step~\eqref{eq:ProofStepII} is the contents of Proposition~\ref{prop:SupZnGaussianApproximation} in the supplementary materials, which is an application of \cite{chernozhukov2014gaussian}.

	Once Steps \eqref{eq:ProofStepI} and \eqref{eq:ProofStepII} hold, the statistic $W_n$ of \eqref{eq:LeungTamStatistic} and the displayed random variable $W_n$ are, respectively, exactly the quantities $\sup_{h \in \OptDomain} \frac{\sqrt{n}}{\sigma_\varepsilon t_n(h)}  \innersmall{ \hat{\rho} - \widehat{\Pi}_{k_n}\rho }{h}$ and $\widetilde{Z}_n$ both normalized by $1 / \beta_n$ to achieve the rate \eqref{eq:NormalizedLTStatisticCvgRate}. 
\end{proof} 

As discussed earlier, our result is a middle ground between the non-convergence (in norm topology) and the pointwise asymptotic normality of the FPCA estimator. The key contribution of our result is to further understand the asymptotic distributional properties of the FPCA estimator. Up to our knowledge, only a few select studies (most notably \cite{cardot2007clt}) have studied this problem from the perspective of inference. There are more, but still few, studies of the asymptotic distributional properties of the FPCA estimator for the purpose of prediction; for instance, see \cite{yao2005functional} and \cite{crambes2013asymptotics}, among others. 

\begin{rem}[Smoothing to eliminate asymptotic bias] 
	Note the right hand side of Step \eqref{eq:ProofStepI} is not normalized by $\frac{1}{\sqrt{n}}$. In other words, with just a scaling of $\sqrt{n}$ on $\sup_{h \in \OptDomain} \frac{\innersmall{\hat{\rho} - \widehat{\Pi}_{k_n}\rho}{h}}{\sigma_\varepsilon t_n(h)}$, its asymptotic bias terms do not converge to zero. In contrast to \cite{cardot2007clt}, here we do not benefit from the extra smoothing in a prediction problem $\innersmall{\hat{\rho} - \widehat{\Pi}_{k_n}\rho}{X_{n+1}}$, where they show normalizing by just $\sqrt{n}$ is sufficient to ensure the asymptotic bias terms will vanish; see also \cite{cai2006prediction}. 

	The sufficient condition $\SuffCondCvgRateLTStatistic \to 0$ for $W_n$ to converge in probability to $\widetilde{W}_n$ effectively depends on the speed for which the truncation parameter $k_n$ of \eqref{eq:Truncation} tends to infinity. The speed of $k_n$ in turn depends on both the speed the eigenvalues $\lambda_j$'s tend to zero, and the speed the regularization $c_n$ tends to zero. 
\end{rem} 

%% file: hypothesistesting.tex
\section{Hypothesis testing} 
\label{sec:HypothesisTesting}
An important application of the small-uniform statistic is hypothesis testing. \cite{cardot2003testing} introduces two statistics (their $D_n$ and $T_n$; see Section~\ref{sec:ComputeWn} later) based on the norm of the cross-covariance operator $\Delta_n$ to test the hypothesis $H_0 : \rho = \rho_0$ versus $H_1 : \rho \neq \rho_0$ (e.g.\ we can say take $\rho_0$ as the zero functional). However, while their statistics test the relationship $\rho = \rho_0$ versus $\rho \neq \rho_0$, it does \emph{not} use an estimate of $\rho$ to form this test. The procedure in \cite{cardot2003testing} is, in some sense, an analysis of variance approach to testing significance. 
\footnote{
	Loosely speaking, the procedure of \cite{cardot2003testing} has the following counterpart in the finite-dimensional linear model. Let $y_i = x_i^\top \beta + \epsilon_i$ be the usual linear model in finite dimensions. Suppose we have the hypothesis $\beta = 0$. Under this null, it necessarily implies $\E[ y_i x_i]  = \E[(x_i^\top 0 + \epsilon_i) x_i] = \E[ \epsilon_i x_i] = 0$. Thus, a test of the hypothesis $\beta = 0$ is to test for zero correlation between $y_i$ and $x_i$ --- such ``correlation test'' does not require an estimate of $\beta$. 
}
In contrast, our hypothesis testing approach here directly uses the FPCA estimator of $\rho$ via the small-uniform statistic $W_n$. 

Let's summarize and outline a practical recipe to applying our main result for the purpose of hypothesis testing. 
\begin{enumerate} 
	\item Fix a statistical significance level $\alpha \in (0, 1)$ and form the hypothesis $H_0 :\ \rho = 0$, $H_1 : \rho \neq 0$.
		\footnote{
			If the hypothesis were instead $H_0 : \rho = \rho_0$, $H_1 : \rho \neq \rho_0$ for $\rho_0$ is non-zero, we consider $Y' := Y - \innersmall{X}{\widehat{\Pi}_{k_n} \rho_0}$. Then the procedure is exactly as follows but we replace the cross-covariance operator $\Delta$ of $(Y, X)$ with the cross-covariance operator $\Delta'$ of $(Y', X)$.
		}

	\item Perform functional PCA on the empirical covariance operator $\Gamma_n$ and collect the empirical eigenelements $(\hat{\lambda}_j, \hat{e}_j)$'s. 

	\item Fix regularization parameters: 
		\begin{enumerate}[(a)]
			\item Based on $\hat{\lambda}_1$ and $\hat{\delta}_1$, pick a sequence $\{c_n\}$ of positive numbers and a sequence of functions $\{f_n\}$ that satisfy Assumption~\ref{as:ConditionF}. 

			\item Compute the empirical truncation parameter $\hat{k}_n$ of \eqref{eq:EmpTruncation}.

			\item Pick a sequence of positive numbers $\{a_n\}$ that satisfy Assumption~\ref{as:ConditionR}.

			\item Pick a sequence of positive numbers $\{\beta_n\}$ that satisfy $\SuffCondEmpCvgRateLTStatistic \to 0$.  
		\end{enumerate} 

	\item Compute the FPCA estimator $\hat{\rho}$ of \eqref{eq:RhoHat}. 

	\item Pick a consistent estimator of the error standard deviation $\hat{\sigma}_\varepsilon$. 

	\item Construct the small-uniform statistic: Numerically solve the fractional programming problem,
		\begin{equation}
			W_n 
			= \frac{\sqrt{n}}{\hat{\sigma}_\varepsilon  \beta_n} \sup_{h \in \hat{\OptDomain}} \frac{\inner{\hat{\rho}}{h}}{\hat{t}_n(h)}  
			= \frac{\sqrt{n}}{\hat{\sigma}_\varepsilon \beta_n} \sup_{ \substack{b \in \R^{\hat{k}_n} \\ \norm{b} \le 1} } \frac{ \sum_{j = 1}^{\hat{k}_n} b_j \innersmall{\hat{\rho} }{\hat{e}_j}}{ \left( \sqrt{ \sum_{j = 1}^{\hat{k}_n} \hat{\lambda}_j [f_n(\hat{\lambda}_j)]^2 b_j^2 } + a_n \right)}. 
			\label{eq:PracticalOptProb}
		\end{equation}

	\item Simulate the asymptotic distribution:
		\begin{enumerate}
			\item Simulate a mean zero Gaussian process $G_{P,n}$ with covariance function \eqref{eq:SupZnGaussianApproximationCovFun}, replacing all population quantities their empirical or estimated counterparts. 

			\item Take the maximum value of this Gaussian process's sample path. 

			\item Repeat (a) and (b) many times to get a simulated distribution of the scalar random variable $\widetilde{W_n}$.

			\item Compute the quantile $q_{1 - \alpha}$; that is, $\Prob(\widetilde{W}_n \le q_{1 - \alpha}) = 1 - \alpha$. 
		\end{enumerate}

	\item Inference: Reject $H_0$ if $W_n > q_{1 - \alpha}$; otherwise, accept it. 

	\end{enumerate}

\begin{rem}[Gradient and Hessian]
	It is evident there is no closed form analytical solution to the optimization problem \eqref{eq:PracticalOptProb}. However, some numerical optimizers can greatly benefit from inputting a known gradient and the Hessian of the objective function. In particular, for $h = \sum_{j = 1}^{k_n} b_j \hat{e}_j$ with $b = (b_1, \ldots, b_{k_n})$, let
	\begin{equation*}
		L(b) 
		:= \frac{\innersmall{\hat{\rho} - \widehat{\Pi}_{k_n}\rho}{h}}{t_n(h)} 
		= \frac{ \sum_{j = 1}^{k_n} b_j \theta_j }{ \sqrt{\sum_{j = 1}^{k_n} b_j^2 \psi_j^2} + a_n} 
		=: \frac{f(b)}{g(b)}
		=: \frac{f(b)}{\sqrt{p(b)} + a_n}
	\end{equation*}
	where we let $\theta_j := \innersmall{\hat{\rho} - \widehat{\Pi}_{k_n}\rho}{\hat{e}_j}$ and $\psi_j := \sqrt{\hat{\lambda}_j} f_n(\hat{\lambda}_j)$. Direct calculations show the $l$-th element of the gradient vector is, 
	\begin{equation*}
		\frac{\partial L}{\partial b_l}
		= \frac{1}{g} \left( \theta_l - b_l \psi_l^2 \frac{f}{g \sqrt{p}} \right)
	\end{equation*}
	and the $(l',l)$-th component of the Hessian is, 
	\begin{equation*}
		\frac{\partial L}{\partial b_{l'} \partial b_l}
		= \frac{1}{g^2} \left[  - b_l \psi_l^2 \frac{g (\theta_{l'} p - b_{l'} \psi_{l'}^2 f) - b_{l'} \psi_{l'}^2 f \sqrt{p} }{ g p^{3/2} } 
		- \frac{b_{l'} \psi_{l'}^2}{\sqrt{p}} \left( \theta_{l} - b_l \psi_l^2 \frac{f}{g \sqrt{p} }  \right) \right] 
	\end{equation*}
\end{rem}

\begin{rem}[Spherical coordinates]
	As stated, \eqref{eq:PracticalOptProb} is an optimization problem with a nonlinear objective function with a norm inequality constraint. The norm inequality constraint is a nonlinear constraint. However, many local and global numerical optimizers are designed to accommodate only box constraints. By using spherical coordinates, we can replace the single norm inequality constraint with just $k_n$ number of box constraints. Specifically, pick $r \in [0,1]$, $\phi_1, \ldots, \phi_{k_n - 2} \in [0, \pi]$ and $\phi_{k_n - 1} \in [0, 2\pi)$. We can change from spherical to Euclidean coordinates via the well-known equations:
	\begin{align*}
		b_1 &= r \cos(\phi_1), \\
		b_2 &= r \sin(\phi_1) \cos(\phi_2), \\
		    & \vdots \\
		b_{k_n - 1} &= r \sin(\phi_1) \cdots \sin(\phi_{k_n - 2})\cos(\phi_{k_n - 1}), \\ 
		b_{k_n} &= r \sin(\phi_1) \cdots \sin(\phi_{k_n - 2}) \sin(\phi_{k_n - 1}).
	\end{align*}
\end{rem}

%% file: simulations.tex
\section{Numerical simulations in hypothesis testing} 
\label{sec:NumericalSimulations}
Let's illustrate the small sample properties of our hypothesis testing procedure from Section~\ref{sec:HypothesisTesting} with numerical simulations. We focus on the Hilbert space $\HSpace = L^2([0,1], \mathcal{B}, \lambda) =: L^2([0,1])$ where $\mathcal{B}$ are the usual Borel sets in $[0, 1]$ and $\lambda$ is the Lebesgue measure in $[0, 1]$. We will focus on the case where the independent variable $X$ is a standard Brownian motion on $[0,1]$. We will use two forms of regulations: (i) ``simple'' regularization where we set $f_n(x) = 1 / x$ when $x \ge c_n$ and $0$ otherwise; and (ii) ``ridge'' regularization where $f_n(x) = 1 / (x + \alpha_n)$ if $x \ge c_n$ and $0$ otherwise. As shown in Example 2 of \cite{cardot2007clt}, we require $\alpha_n \sqrt{n} / c_n \to 0$ to satisfy Assumption~\ref{as:ConditionF}.

\subsection{Parameterization} 
\label{sec:Parameterization} 
It is well known (for instance, see Example 4.6.3 of \cite{hsing2015theoretical}) the eigenelements of the covariance operator of Brownian motion are, 
\begin{equation*}
	\lambda_j = \frac{4}{((2 j - 1) \pi)^2} 
	\quad\text{and}\quad
	e_j(t) = \sqrt{2} \sin\left( \frac{(2 j - 1) \pi}{2} t \right).
\end{equation*}
These eigenvalues satisfy Assumption~\ref{as:ConditionA}, as $\lambda_j = \BigOh(\frac{1}{j^2}) \le \BigOh(\frac{1}{j \log j})$. In particular we have $\delta_j = \lambda_j - \lambda_{j + 1}$ for $j \ge 2$, and so $\lambda_j + \frac{\delta_j}{2} = \frac{4 (4 j^2 + 8 j + 1)}{\pi^2 (2j + 1)^2 (2j - 1)^2} \lesssim \BigOh( \frac{1}{j^2} )$. Recalling \eqref{eq:Truncation} and $\{ c_n \}$ is a sequence tending to zero, the above implies the upper bound $k_n \lesssim \BigOh\left( \frac{1}{\sqrt{c_n}} \right)$.

With respect to the required rate of our Theorem~\ref{thm:NormalizedLeungTamStatisticConvergence} and Assumption~\ref{as:ConditionR}, we choose $\beta_n = (\log n)^2$. Consequently, we have the bounds $\SuffCondCvgRateLTStatistic \lesssim \BigOh\left( \frac{1}{\log n} \frac{1}{c_n^{13/4}} \left( \log \frac{1}{\sqrt{c_n}} \right)^{9/2} \right)$ and $a_n \sqrt{k_n \log k_n} \lesssim \BigOh\left( a_n \frac{1}{\sqrt{c_n}} \log \frac{1}{\sqrt{c_n}} \right)$. For the ridge regularization, we also need a choice of $\{ \alpha_n \}$ such that $\alpha_n \sqrt{n} / c_n \to 0$. So in all, the $c_n$, $a_n$ and $\alpha_n$, all tending to zero, must also satisfy the three requirements: (i) $\frac{1}{\log n} \frac{1}{c_n^{13/4}} \left( \log \frac{1}{\sqrt{c_n}} \right)^{9/2} \to 0$; (ii) $a_n \frac{1}{\sqrt{c_n}} \log \frac{1}{\sqrt{c_n}} \to 0$; and (iii) $\alpha_n \sqrt{n} / c_n \to 0$.

For our illustrations we pick $c_n = \frac{C}{\log \log n}$, $a_n = \frac{1}{n^2}$, and $\alpha_n = \frac{1}{\sqrt{n} \log n}$. It is easy to show these choices of $c_n$'s, $a_n$'s and $\alpha_n$'s will satisfy the aforementioned requirements (i) to (iii). However in finite samples the choice of the constant $C$ in $c_n$ has a material impact on the numerical results. We consider the choices $C = \lambda_1^c$ (deterministic case) and $C = \hat{\lambda}_1^c$ (data based case) for $c = 2, 3, 5, 7$ and $8$. In the deterministic case, we assume we know perfectly the values $\lambda_j$'s as per the above displayed equation, and this correspondingly implies deterministic quantities $k_n$ of \eqref{eq:Truncation} and $f_n$ (through the defining condition $x \ge c_n$). In the data based case, we use the random truncation $\hat{k}_n$ and the corresponding data dependent $f_n$. 
\footnote{
	We emphasize in the deterministic case, it is only in the calculations of $C = \lambda_1^c$, $k_n$ and $f_n$ do we assume we have perfect knowledge of the eigenvalues $\lambda_j$'s. The eigendecomposition of $\Delta_n$ are still based on the random observations $X_i$'s in our simulations. 
}
The exponents $c$ are chosen as such because they generate a good range of truncation parameter $k_n$ and $\hat{k}_n$ values for our numerical illustrations; higher values of $c$ imply larger values of $k_n$ and $\hat{k}_n$.

We will consider three different coefficient vectors in $L^2([0, 1])$: 
\begin{itemize} 
	\item $\rho_0(t) \equiv 0$; 

	\item $\rho_1(t) = \sin\left(\frac{\pi}{2} t \right) + \frac{1}{2} \sin\left(\frac{3 \pi}{2} t \right) + \frac{1}{4} \sin\left( \frac{5 \pi}{2} t \right)$; and 

	\item $\rho_2(t) = \sin(2 \pi t^3)^3$. 
\end{itemize} 
The first choice $\rho_0$ is used to evaluate the size of our small-uniform statistic
$W_n$, while $\rho_1$ and $\rho_2$ are used to evaluate power. The
second choice is a case where the coefficient vector is exactly spanned by the
first three eigenvectors of the Brownian motion covariance operator. The third
choice is an example where the coefficient vector cannot be linearly spanned by
those eigenvectors. We note \cite{cardot2003testing} also numerically
illustrates cases 1 and 3, while \cite{cardot1999functional} illustrates case
2. We fix the noise $\varepsilon$ distribution as a Gaussian $\mathcal{N}(0,
\sigma_\varepsilon^2)$ distribution where we pick variance $\sigma_\varepsilon^2 = \frac{1
- \mathrm{snr}}{\mathrm{snr}} \Var(\innersmall{X}{\rho})$, and where $\mathrm{snr}$
is the  ``signal-to-noise ratio'' and we let it to be $\mathrm{snr} = 5\%$ and
$10\%$. To focus the discussion on the properties of our small-uniform
statistic and not on the estimation performance of an error estimator
$\hat{\sigma}_\varepsilon$, we assume throughout all these numerical
simulations that the noise parameter $\sigma_\varepsilon^2$ is known with certainty. 

For each of the three example coefficient vectors and each of the two noise distributions, we run $n_s = 2500$ simulations of $\{ (Y_i, X_i) \}_{i = 1}^n$ for each of the sample size choices $n = 50, 200, 1000$. The Brownian motion $X_i$ and the function $\rho$ are discretized by 100 equispaced points in $[0, 1]$, and the $L^2([0, 1])$ inner product is approximated by the trapezoid rule. The eigenelements of the empirical covariance operator are computed using the \texttt{fdapace} package
\footnote{
	\url{https://cran.r-project.org/web/packages/fdapace/}, version 0.5.5.
}
of the \texttt{R} language. 

Once the FPCA estimator $\hat{\rho}$ is constructed as per \eqref{eq:RhoHat}, we can
evaluate it pointwise on $[0, 1]$ as $\hat{\rho}(t) = \sum_{j = 1}^{k_n}
\inner{\hat{\rho}}{\hat{e}_j} \hat{e}_j(t)$. At the end of each simulation
round, we will also compute and record the quadratic error measure,
\begin{align}
	\mathrm{error}(\rho)
	&= 
	\begin{dcases}
		\int_0^1 (\rho(t) - \hat{\rho}(t))^2 \diff{t}, & \text{if $\rho \equiv 0$;
		and} \\
		\frac{\int_0^1 (\rho(t) - \hat{\rho}(t))^2 \diff{t}}{\int_0^1 \rho(t)^2 \diff{t} }, & \text{otherwise.}
	\end{dcases} 
	\label{eq:ErrorMeasure}
\end{align}

\subsection{Computing $W_n$} 
\label{sec:ComputeWn}
For succinctness in discussing both the deterministic truncation case and
the data driven truncation case, let's denote $K_n \in \{ k_n, \lceil \hat{k}_n \, \textrm{avg} \rceil \}$. Here $\hat{k}_n \, \textrm{avg}$ denotes the averaged random
truncations over the $n_s$ number of simulations for a given sample size choice
$n$, and $\lceil \cdot \rceil$ is the ceiling function. That is to say, if we
use deterministic truncation we simply set $K_n$ to the known value $k_n$, and if we were to use a
data driven truncation, we set $K_n$ to be the averaged truncation parameter.

The computation of $W_n$ is as written in \eqref{eq:PracticalOptProb}. As
mentioned above, we evaluate $W_n$ under a known standard deviation
$\sigma_\varepsilon$ of the error distribution. The
optimization step in \eqref{eq:PracticalOptProb} is computed using a
combination of a global and local search. A uniformly random point in drawn in
$\R^{K_n}$ and this serves as the initial point in the constrained nonlinear global
optimizer \texttt{ISRES}
\footnote{
	We set a relative error tolerance of $10^{-4}$ and a maximum of $100$ runs. 
}
of \cite{runarsson2005search}. To further refine the solution, we take the
resulting solution point and set that as the initial point in the constrained
nonlinear local optimizer \texttt{COBYLA} 
\footnote{
	Again, we set a relative error tolerance of $10^{-4}$ and a maximum of $100$ runs. 
}
of \cite{powell1994direct}. The end of this procedure results in our small-uniform statistic $W_n$. 
Although not thoroughly experimented in this paper, but we sense the specific
choices of these numerical optimization algorithms are not particularly important. 

As a matter of comparison, we will also compute the $D_n$ and $T_n$ statistics
of \cite{cardot2003testing}. These two statistics are defined as, 
\begin{equation*}
	D_n := \frac{1}{\sigma_\varepsilon^2} \norm{\sqrt{n} \Delta_n \hat{A}_n}^2,
	\quad
	T_n := \frac{D_n - k_n}{\sqrt{k_n}} 
\end{equation*}
where $\hat{A}_n := \sum_{j = 1}^{k_n} \frac{1}{\hat{\lambda}_j} \hat{e}_j
\otimes \hat{e}_j$. \cite{cardot2003testing} show that under the null
hypothesis, $D_n \weakcvgto \chi^2(k_n)$ and $T_n \weakcvgto \mathcal{N}(0,
2)$. Let $q_{\chi^2(k_n), 1 - \alpha}$ denote the quantile
$\Prob(\chi^2(k_n) \le q_{\chi^2(k_n), 1 - \alpha}) = 1 - \alpha$ and
$q_{\mathcal{N}(0,1), 1 - \alpha / 2}$ denote the quantile $\Prob( \mathcal{N}(0,
1) \le q_{\mathcal{N}(0,1), 1 - \alpha / 2}) = 1 - \alpha / 2$. Then we reject the
null hypothesis $H_0 : \rho = 0$ using the $D_n$ statistic if $D_n >
q_{\chi^2(k_n), 1 - \alpha}$ and
reject the null hypothesis using the $T_n$ statistic if $\abs{T_n} > \sqrt{2}
q_{\mathcal{N}(0,1), 1 - \alpha / 2}$. Otherwise, we accept the null hypothesis.
Notice we can only make the comparison of these two statistics against our small-uniform
statistic $W_n$ in the deterministic truncation parameter $k_n$ case. 

\subsection{Simulating $\widetilde{W}_n$} 
\label{sec:SimTildeWn} 
The distribution of the supremum of our Gaussian process $\widetilde{W}_n$ must be numerically simulated. Note in the data driven case, it necessarily implies a mismatch between the truncation $\hat{k}_n$ that was used to compute each small-uniform statistic $W_n$ for a given simulation epoch, and the asymptotic distribution approximation $\widetilde{W}_n$ that depends on $\hat{k}_n \, \textrm{avg}$. 

Let's describe our simulation procedure. We first uniformly draw 25 points on the boundary of a $K_n$-sphere, and then uniformly draw another 25 points in the interior of that $K_n$-sphere; that is, a total of $25^2 = 625$ of $K_n$-vectors are drawn. We evaluate the covariance function \eqref{eq:SupZnGaussianApproximationCovFun} on the Cartesian product of these 625 points, and this results in a $625 \times 625$ dimensional covariance matrix. 
\footnote{
	For $h_l \in \OptDomain$, $l = 1, 2$ we can write $h_l = \sum_{j = 1}^{K_n} b_j e_j$ where $b = (b_1, \ldots, b_{K_n})$ is a real Euclidean vector in the $K_n$ unit sphere. Consequently, the covariance function \eqref{eq:SupZnGaussianApproximationCovFun} can be more explicitly written as a function on the Cartesian product of two $K_n$ unit spheres,
	\begin{equation*}
		c_n(x, y) 
		= \frac{ \sum_{j = 1}^{K_n} \lambda_j f_n(\lambda_j)^2 x_j y_j}{ \left( \sqrt{
				\sum_{j = 1}^{K_n} \lambda_j f_n(\lambda_j)^2 x_j^2 } + a_n \right) \left(
		\sqrt{ \sum_{j = 1}^{K_n} \lambda_j f_n(\lambda_j)^2 y_j^2 } + a_n  \right) } 
	\end{equation*} 
	where $\norm{x}_{\R^{K_n}} \le 1$ and $\norm{y}_{\R^{K_n}} \le 1$.
}
We draw an observation from a $625$-dimensional mean zero multivariate normal distribution with this covariance matrix. This observation represents one sample path of the Gaussian process $G_{P,n}$. We record the maximum value of this sample path. 

We repeat the above procedure $1.6$ million times. The end result is we will have 1.6 million maximum values and these values represent the simulated distribution of the random variable $\widetilde{Z}_n$ of Theorem~\ref{thm:NormalizedLeungTamStatisticConvergence}. Finally, we normalize $\widetilde{Z}_n$ by $\frac{1}{\beta_n} = \frac{1}{(\log n)^2}$ to arrive at the simulated distribution of $\widetilde{W}_n$. Figure~\ref{fig:TildeWnHistogram} plots the results of the simulations. The quantile numbers $q_{1 - \alpha}$ are accordingly numerically computed. 

\begin{landscape}
\begin{figure}
	\centering
	\begin{subfigure}[b]{0.33\textwidth} 
		\centering
		\includegraphics[width=\linewidth]{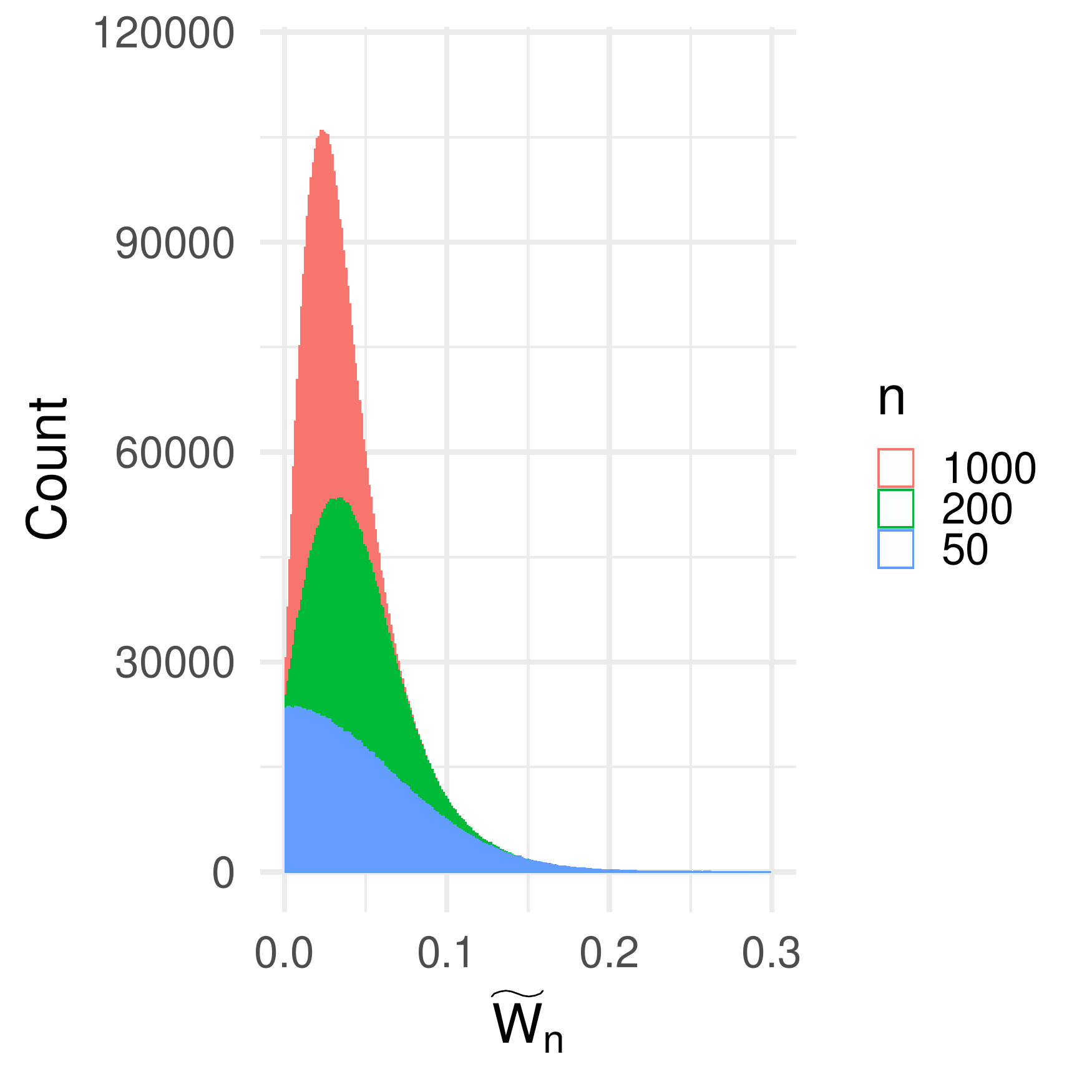} 
		\caption{$c = 3$}
	\end{subfigure} 
	\hfill
	\begin{subfigure}[b]{0.33\textwidth} 
		\centering
		\includegraphics[width=\linewidth]{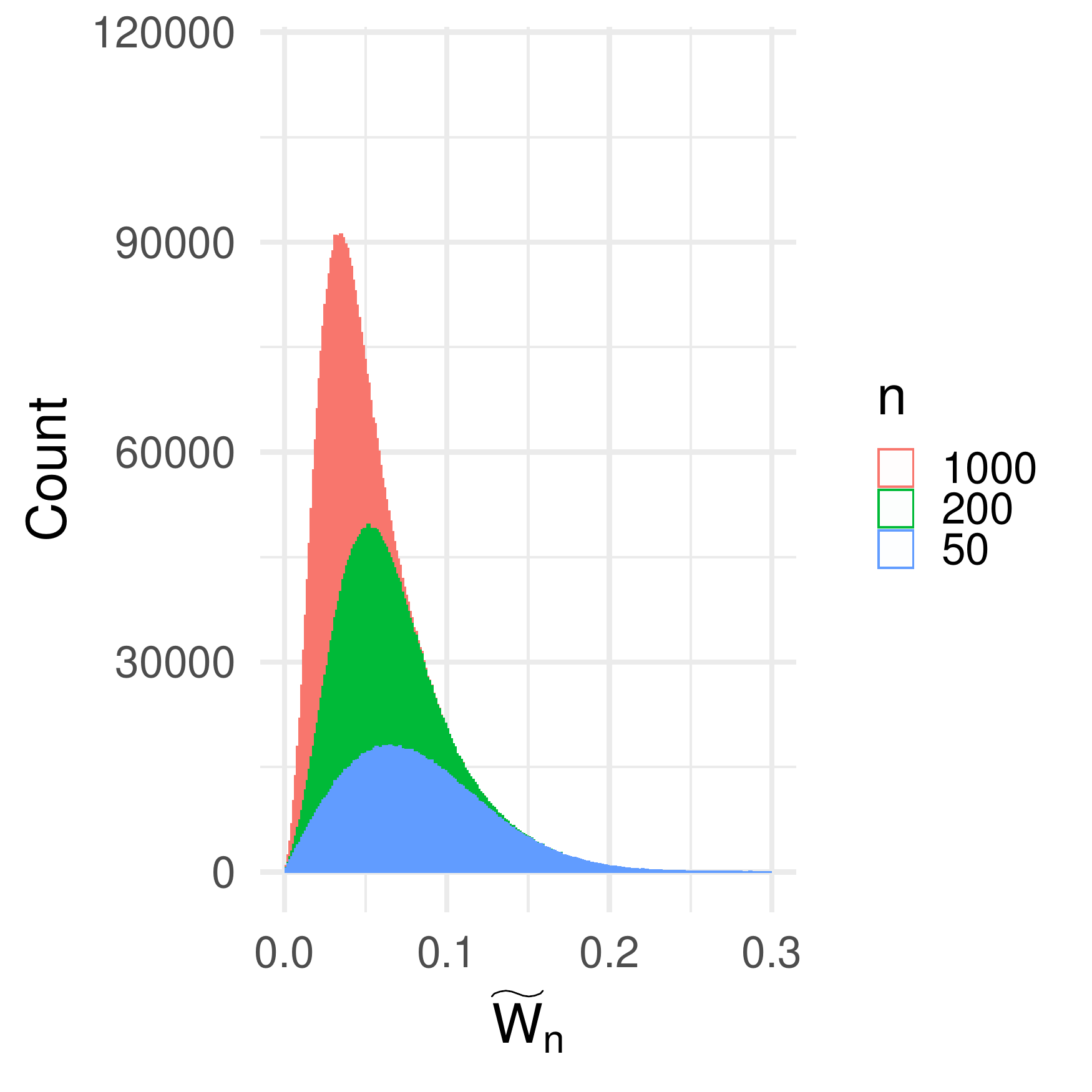} 
		\caption{$c = 4$}
	\end{subfigure} 
	\hfill
	\begin{subfigure}[b]{0.33\textwidth} 
		\centering
		\includegraphics[width=\linewidth]{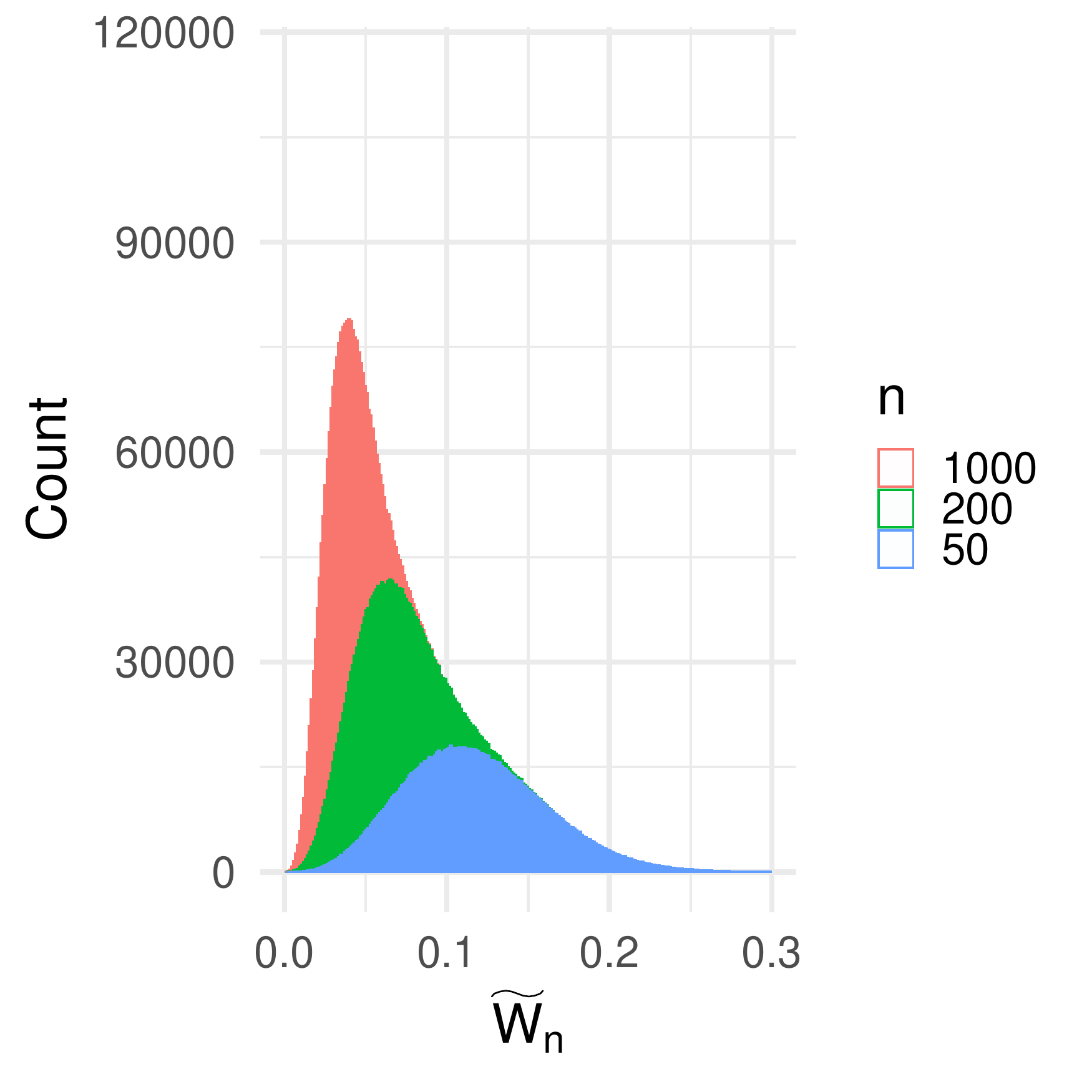} 
		\caption{$c = 5$}
	\end{subfigure} 
	\\
	\begin{subfigure}[b]{0.33\textwidth} 
		\centering
		\includegraphics[width=\linewidth]{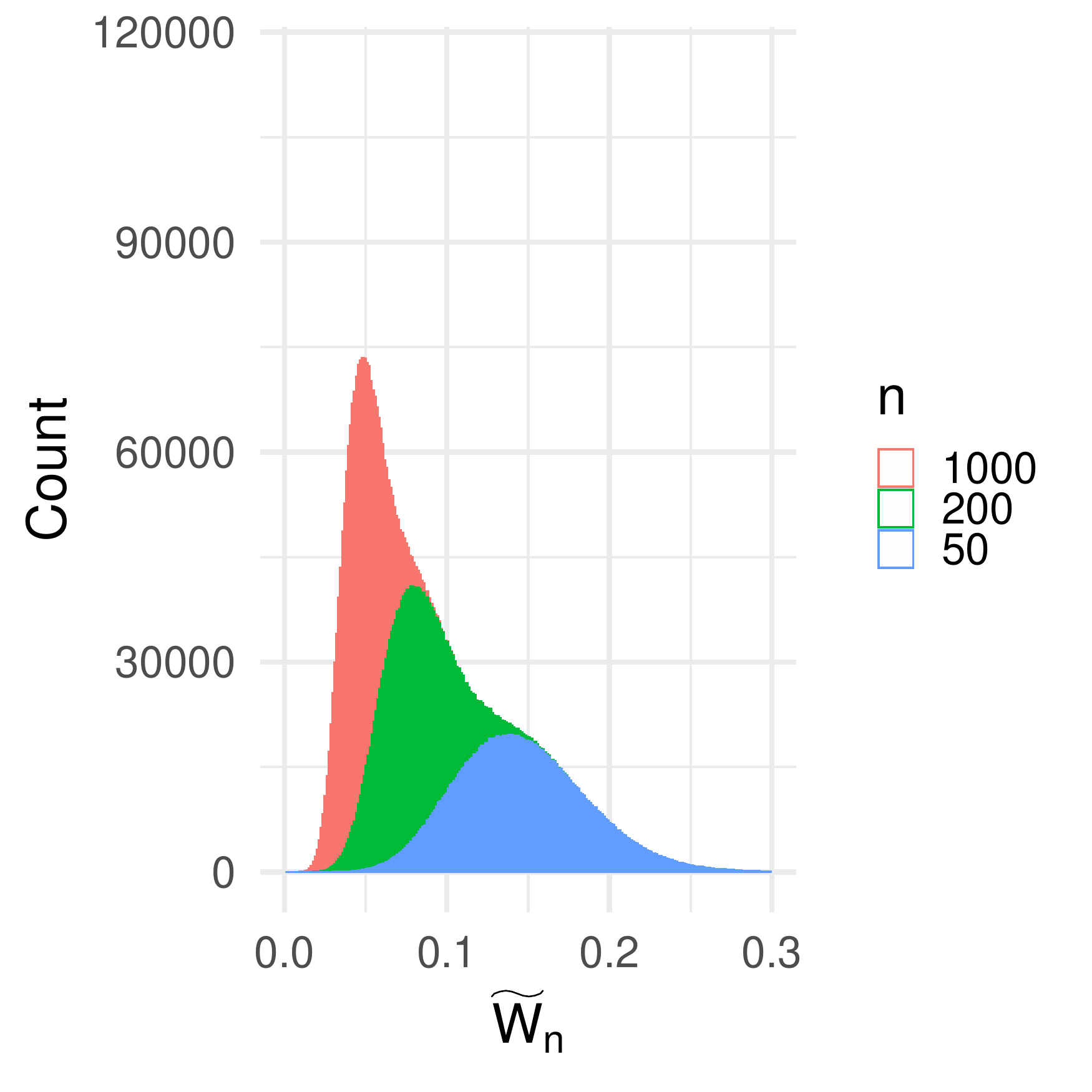} 
		\caption{$c = 7$}
	\end{subfigure} 
	\quad\quad
	\begin{subfigure}[b]{0.33\textwidth} 
		\centering
		\includegraphics[width=\linewidth]{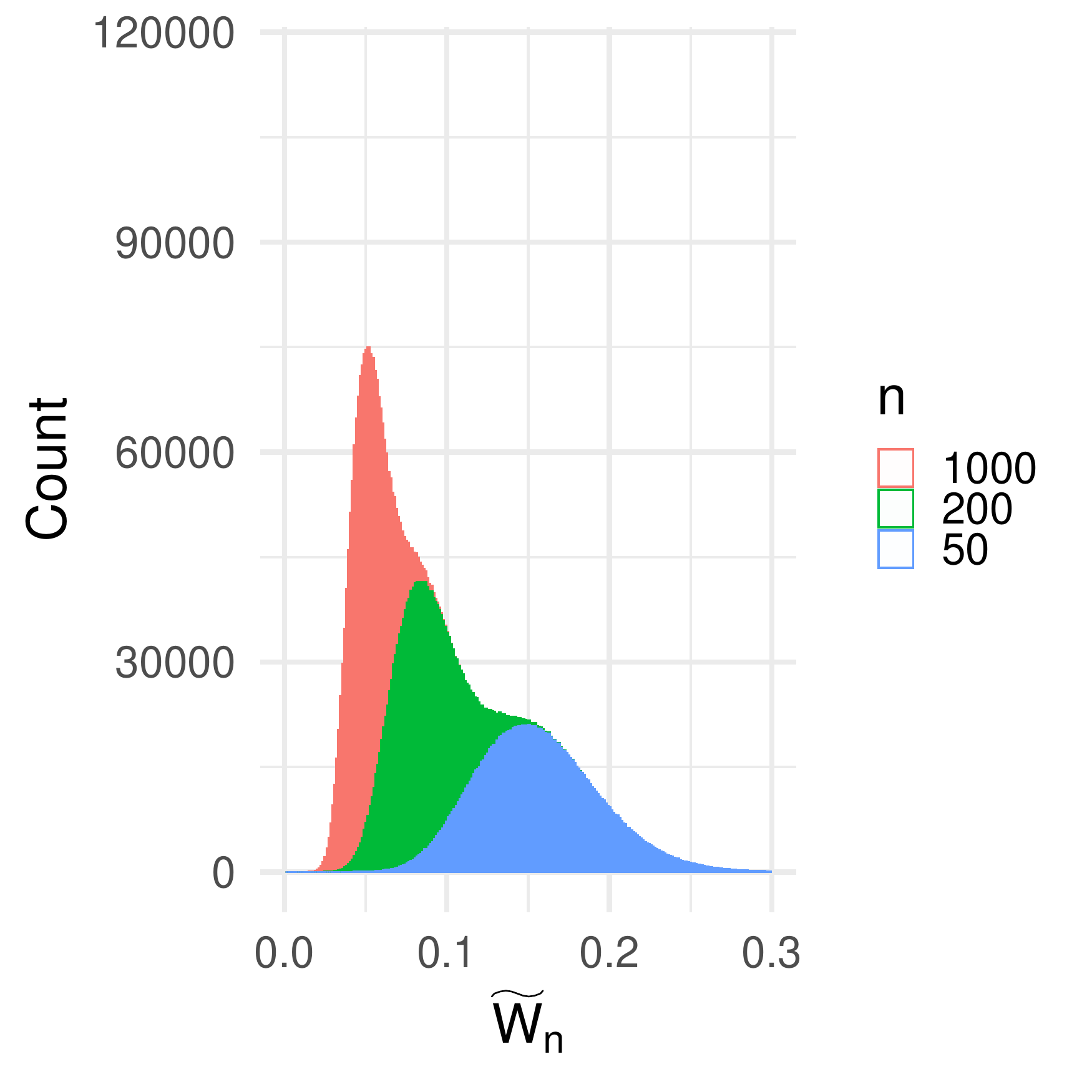} 
		\caption{$c = 8$}
	\end{subfigure} 
	\caption{%
		Histograms of the distribution of $\widetilde{W}_n$ for various sample sizes $n$ and various exponents $c$ when the FPCA uses ridge regularization. These plots are best seen in color. Details of the parameterization are described in Section~\ref{sec:Parameterization}. The procedure for simulating $\widetilde{W}_n$ is described in Section~\ref{sec:SimTildeWn}. The histogram plots for when the FPCA uses simple regularization are similar; they are not shown for brevity.
	}
	\label{fig:TildeWnHistogram} 
\end{figure} 
\end{landscape}

\begin{rem}
	In this numerical simulation exercise, it was far more time consuming to simulate and compute the statistic $W_n$ than simulating its asymptotic approximation $\widetilde{W}_n$. Simply put, both the spectral decomposition of $\Gamma_n$ and the numerical optimization steps in computing $W_n$ are computationally expensive, and made even more so when we have to do this $n_s$ many times across various sample size choices $n$. Of course, in actual practice where $W_n$ is only computed once based on the given data, the computation time of a single $W_n$ is negligible. 
\end{rem} 

\subsection{Discussion of the numerical simulation results} 
Choosing the coefficient vector as $\rho_0(t) \equiv 0$ and a Gaussian noise distribution, Table~\ref{tab:psi1_data_based_kn_gaussian_nu_1900} ($\mathrm{snr} = 5\%$) and Table~\ref{tab:psi1_data_based_kn_gaussian_nu_1900} ($\mathrm{snr} = 5\%$) show the results when we use a data based truncation parameter $\hat{k}_n$; and Table~\ref{tab:psi1_oracle_kn_gaussian_nu_900} ($\mathrm{snr} = 5\%$) and Table~\ref{tab:psi1_oracle_kn_gaussian_nu_1900} ($\mathrm{snr} = 5\%$) show the case when we use a deterministic truncation $k_n$. Thus these tables illustrate the size properties of our small-uniform statistic. Firstly, we see there is little qualitative difference of the levels between the deterministic and data driven truncation cases, which suggests random variations in the eigenelements, and hence in the determination of $\hat{k}_n$, do not substantially affect the estimated levels. Thus for the remainder of this section, we will focus on the deterministic truncation $k_n$ case, as this focus allows us to further compare our $W_n$ statistic against \cite{cardot2003testing}'s $D_n$ and $T_n$ statistics.

Let's focus on Table~\ref{tab:psi1_oracle_kn_gaussian_nu_1900} with $\mathrm{snr} = 5\%$. We see the estimated size of our small-uniform statistic $W_n$ (for both the reciprocal and ridge regularization cases) matches the simulated levels of its asymptotic distribution $\widetilde{W}_n$ when the truncation $k_n$ is small. However, this matching deteriorates as the truncation increases, and perhaps paradoxically, also deteriorates with larger sample sizes. This can be explained from the log errors: as the truncation and sample size increases, the quality of the estimator $\hat{\rho}$ of the true coefficient $\rho_0 \equiv 0$ decreases. Indeed, when the true coefficient $\rho_0$ is zero, the ``optimal'' truncation should simple be $k_n = 0$. In all, our numerical results suggest the FPCA estimator (and especially the case of simple regularization) has significant difficulty in estimating a null coefficient. And since our small-uniform statistic $W_n$ is based on the FPCA estimator, it is thus no surprise the size performance of $W_n$ is also necessarily hampered. In contrast, the $D_n$ and $T_n$ statistics of \cite{cardot2003testing} do not depend on the FPCA estimator, and their nominal levels appear to be stable across truncations and sample sizes.  Table~\ref{tab:psi1_oracle_kn_gaussian_nu_900} is the results with a higher $\mathrm{snr} = 10\%$ and exhibit the same qualitative behavior of $W_n, D_n$ and $T_n$ as discussed above.

Let's now discuss the empirical power of our statistic $W_n$.  Tables~\ref{tab:psi2_oracle_kn_gaussian_nu_1900} ($\mathrm{snr} = 5\%$) and \ref{tab:psi2_oracle_kn_gaussian_nu_900} ($\mathrm{snr} = 10\%$) show the results for the power against $\rho_1$. By design, $\rho_1$ is a linear combination of the first three eigenvectors of $\Gamma$, and so the ``optimal'' truncation $k_n$ for $\rho_1$ is exactly $3$. Hence, we should expect the best performance for all the statistics $W_n, D_n$ and $T_n$ at $k_n = 3$ (i.e.\ $c = 4$). Even with a modest sample size of $n = 200$, it appears the empirical power of $W_n$ (for both the simple and ridge regularizations) yield qualitatively almost identical power to that of $D_n$ and $T_n$. For the other truncation cases (i.e.\ corresponding to $c = 3, 5, 7$ and $8$), it appears $W_n$, again for both the simple and ridge regularizations, yield higher power than $D_n$ and $T_n$. However this observation is not without reservations. On the one hand, higher truncations lead to higher log quadratic errors of $\hat{\rho}$. But on the other hand, it could very well be possible that the estimated coefficient $\hat{\rho}$ doesn't resemble the true coefficient $\rho_1$, but $\hat{\rho}$ nonetheless is still significantly different from the null vector, and that the optimizing nature of $W_n$ can advantage of this.  Thus this suggests our small-uniform statistic $W_n$ is robust at rejecting the null hypothesis $H_0:\ \rho = 0$ even if the underling FPCA estimator $\hat{\rho}$ has high estimation error, as is most evident when using the simple regularization, in finite samples.

Finally, Tables~\ref{tab:psi3_oracle_kn_gaussian_nu_1900} ($\mathrm{snr} = 5\%$) and \ref{tab:psi3_oracle_kn_gaussian_nu_900} ($\mathrm{snr} = 10\%$) show the results for the power against $\rho_2$. This coefficient vector $\rho_2$ is designed such that it is not a linear combination of the eigenvectors of $\Gamma$ and so higher truncations $k_n$ should yield better results. This coefficient vector $\rho_2$ example is particularly important because real world coefficient vectors of the FLM are highly unlikely to be just simple linear combinations of the eigenvectors of $\Gamma$.  Here, the power of our small-uniform statistic $W_n$ outperforms that of $D_n$ and $T_n$, especially at high truncations. Although it is not the purpose of this paper to empirically evaluate the performance of various regularization regimes, it does appear that the log quadratic error of the FPCA estimator under ridge regularization is substantially lower than when the FPCA estimator uses simple regularization. 

\begin{landscape}
\import{\tblpath}{tbl_psi1_oracle_kn_gaussian_nu_1900.tex}
\import{\tblpath}{tbl_psi1_data_based_kn_gaussian_nu_1900.tex}
\import{\tblpath}{tbl_psi1_oracle_kn_gaussian_nu_900.tex}
\import{\tblpath}{tbl_psi1_data_based_kn_gaussian_nu_900.tex}

\import{\tblpath}{tbl_psi2_oracle_kn_gaussian_nu_1900.tex}
\import{\tblpath}{tbl_psi2_oracle_kn_gaussian_nu_900.tex}

\import{\tblpath}{tbl_psi3_oracle_kn_gaussian_nu_1900.tex}
\import{\tblpath}{tbl_psi3_oracle_kn_gaussian_nu_900.tex}
\end{landscape}

%% file: tex_tables/tbl_psi1_oracle_kn_gaussian_nu_1900.tex
\begin{table}

\caption{\label{tab:psi1_oracle_kn_gaussian_nu_1900}The empirical power (in percentages) of our small-uniform $W_n$ statistic along with \cite{cardot2003testing}'s $D_n$ and $T_n$ statistics when $\rho(t) = \rho_0(t) \equiv 0$ and $\varepsilon_i$ has a $\mathcal{N}(0, \sigma_{\varepsilon}^2)$ distribution with $\sigma_{\varepsilon}^2 = \frac{1 - \mathrm{snr}}{\mathrm{snr}} \mathrm{Var}(\inner{X}{\rho})$ with $\mathrm{snr} = 5\%$. Here we assume the truncation parameter $k_n$ is known.  The $n$ here refers to sample size, and $c$ here refers to the exponent associated with the definition of $c_n$. The ``log error'' here refers to the average over all the simulations of the log of the error measure as given in \eqref{eq:ErrorMeasure}. Section~\ref{sec:SimTildeWn} describes our procedure to obtain the simulated levels of $\widetilde{W}_n$. The nominal levels of $D_n$ and $T_n$ are based on their respective asymptotic distributions as described in Section~\ref{sec:ComputeWn}.}
\centering
\resizebox{\linewidth}{!}{
\fontsize{10}{12}\selectfont
\begin{tabular}[t]{cccccccccccccccccccc}
\toprule
\multicolumn{2}{c}{ } & \multicolumn{5}{c}{\textbf{$\boldsymbol{\widetilde{W}_n}$ (simple regularization)}} & \multicolumn{5}{c}{\textbf{$\boldsymbol{\widetilde{W}_n}$ (ridge regularization)}} & \multicolumn{4}{c}{ } & \multicolumn{4}{c}{ } \\
\cmidrule(l{3pt}r{3pt}){3-7} \cmidrule(l{3pt}r{3pt}){8-12}
\multicolumn{3}{c}{ } & \multicolumn{4}{c}{\textbf{Simulated level}} & \multicolumn{1}{c}{ } & \multicolumn{4}{c}{\textbf{Simulated level}} & \multicolumn{4}{c}{\textbf{Nominal level of $\boldsymbol{D_n}$}} & \multicolumn{4}{c}{\textbf{Nominal level of $\boldsymbol{T_n}$}} \\
\cmidrule(l{3pt}r{3pt}){4-7} \cmidrule(l{3pt}r{3pt}){9-12} \cmidrule(l{3pt}r{3pt}){13-16} \cmidrule(l{3pt}r{3pt}){17-20}
$n$ & $k_n$ & log error & 1 & 5 & 10 & 20 & log error & 1 & 5 & 10 & 20 & 1 & 5 & 10 & 20 & 1 & 5 & 10 & 20\\
\midrule
\addlinespace[0.3em]
\multicolumn{20}{l}{\textbf{$c = 3$}}\\
\hspace{1em}50 & 2 & -327.44 & 2.24 & 9.60 & 16.80 & 31.04 & -376.72 & 2.60 & 10.56 & 18.00 & 31.60 & 0.92 & 5.36 & 10.64 & 20.04 & 2.76 & 5.68 & 8.04 & 10.76\\
\hspace{1em}200 & 2 & -321.85 & 0.76 & 3.40 & 7.32 & 15.04 & -349.58 & 0.44 & 3.44 & 7.16 & 14.36 & 0.84 & 4.64 & 9.96 & 20.00 & 2.48 & 4.84 & 7.12 & 10.40\\
\hspace{1em}1000 & 2 & -450.00 & 1.20 & 5.24 & 9.52 & 18.64 & -469.85 & 1.04 & 5.04 & 9.80 & 19.92 & 1.32 & 5.64 & 10.16 & 19.92 & 3.36 & 5.84 & 7.36 & 10.48\\
\addlinespace[0.3em]
\multicolumn{20}{l}{\textbf{$c = 4$}}\\
\hspace{1em}50 & 3 & -133.88 & 2.08 & 9.44 & 15.48 & 26.28 & -229.18 & 2.12 & 8.52 & 15.36 & 25.96 & 1.24 & 5.44 & 10.84 & 19.96 & 2.72 & 5.52 & 7.36 & 11.28\\
\hspace{1em}200 & 3 & -206.36 & 0.44 & 3.60 & 7.04 & 13.92 & -264.24 & 0.92 & 3.20 & 6.84 & 15.48 & 0.56 & 5.12 & 9.92 & 18.60 & 2.08 & 5.16 & 6.80 & 10.48\\
\hspace{1em}1000 & 3 & -303.70 & 0.68 & 4.40 & 9.16 & 18.04 & -339.78 & 0.80 & 4.16 & 8.60 & 16.76 & 0.92 & 4.92 & 10.24 & 19.96 & 2.60 & 4.96 & 7.32 & 10.72\\
\addlinespace[0.3em]
\multicolumn{20}{l}{\textbf{$c = 5$}}\\
\hspace{1em}50 & 4 & 58.28 & 3.32 & 9.24 & 15.56 & 25.72 & -139.09 & 1.48 & 6.52 & 11.92 & 21.16 & 1.12 & 5.08 & 9.56 & 19.60 & 2.24 & 5.08 & 6.44 & 12.12\\
\hspace{1em}200 & 4 & -42.88 & 1.84 & 7.20 & 12.76 & 24.64 & -164.06 & 2.24 & 7.68 & 13.48 & 24.72 & 1.00 & 4.84 & 9.36 & 19.40 & 2.48 & 4.80 & 6.60 & 11.16\\
\hspace{1em}1000 & 5 & -202.19 & 2.04 & 7.48 & 14.20 & 25.56 & -259.44 & 1.80 & 7.08 & 13.64 & 24.80 & 0.68 & 4.76 & 9.32 & 19.40 & 2.12 & 4.52 & 6.52 & 13.48\\
\addlinespace[0.3em]
\multicolumn{20}{l}{\textbf{$c = 7$}}\\
\hspace{1em}50 & 9 & 356.97 & 13.52 & 25.32 & 34.68 & 45.00 & -75.02 & 2.64 & 8.04 & 14.48 & 24.76 & 1.16 & 5.12 & 10.32 & 19.60 & 2.32 & 4.52 & 8.12 & 17.60\\
\hspace{1em}200 & 10 & 258.27 & 12.68 & 28.32 & 40.00 & 54.48 & -74.13 & 4.28 & 11.88 & 19.04 & 30.60 & 0.88 & 5.64 & 11.40 & 21.40 & 1.60 & 5.04 & 9.08 & 17.60\\
\hspace{1em}1000 & 11 & 121.34 & 12.60 & 27.12 & 38.76 & 53.80 & -99.45 & 6.28 & 18.80 & 28.72 & 42.92 & 0.96 & 5.40 & 10.92 & 21.00 & 1.64 & 4.44 & 8.84 & 18.68\\
\addlinespace[0.3em]
\multicolumn{20}{l}{\textbf{$c = 8$}}\\
\hspace{1em}50 & 14 & 502.37 & 18.72 & 32.68 & 41.52 & 51.80 & -60.25 & 2.40 & 8.00 & 13.48 & 22.72 & 1.60 & 5.96 & 11.52 & 21.16 & 2.40 & 5.48 & 9.76 & 19.76\\
\hspace{1em}200 & 16 & 400.44 & 19.52 & 35.36 & 44.52 & 56.88 & -51.52 & 4.28 & 13.16 & 21.04 & 33.24 & 1.40 & 6.08 & 11.72 & 21.72 & 2.00 & 5.56 & 10.60 & 20.76\\
\hspace{1em}1000 & 17 & 267.09 & 21.60 & 38.04 & 48.84 & 62.04 & -65.88 & 6.32 & 16.24 & 24.52 & 36.32 & 2.16 & 7.56 & 13.44 & 25.36 & 2.96 & 6.40 & 11.04 & 21.40\\
\bottomrule
\end{tabular}}
\end{table}

%% file: tex_tables/tbl_psi1_data_based_kn_gaussian_nu_1900.tex
\begin{table}

\caption{\label{tab:psi1_data_based_kn_gaussian_nu_1900}The empirical power (in percentages) of our small-uniform $W_n$ statistic when $\rho(t) = \rho_0(t) \equiv 0$ and $\varepsilon_i$ has a $\mathcal{N}(0, \sigma_{\varepsilon}^2)$ distribution with $\sigma_{\varepsilon}^2 = \frac{1 - \mathrm{snr}}{\mathrm{snr}} \mathrm{Var}(\inner{X}{\rho})$ with $\mathrm{snr} = 5\%$. Here we use a data driven truncation parameter $\hat{k}_n$. In particular, ``$\hat{k}_n$ avg'' is the average truncation value over $n_s$ number of simulations, and ``$\hat{k}_n$ std'' is the associated standard error. The $n$ here refers to sample size, and $c$ here refers to the exponent associated with the definition of $c_n$. The ``log error'' here refers to the average over all the simulations of the log of the error measure as given in \eqref{eq:ErrorMeasure}. Section~\ref{sec:SimTildeWn} describes our procedure to obtain the simulated levels of $\widetilde{W}_n$. }
\centering
\fontsize{10}{12}\selectfont
\begin{tabular}[t]{rrrrrrrrrrrrr}
\toprule
\multicolumn{3}{c}{ } & \multicolumn{5}{c}{\textbf{Simulated level of $\boldsymbol{\widetilde{W}_n}$ (simple)}} & \multicolumn{5}{c}{\textbf{Simulated level of $\boldsymbol{\widetilde{W}_n}$ (ridge)}} \\
\cmidrule(l{3pt}r{3pt}){4-8} \cmidrule(l{3pt}r{3pt}){9-13}
$n$ & $\hat{k}_n$ avg & $\hat{k}_n$ std & log error & 1 & 5 & 10 & 20 & log error & 1 & 5 & 10 & 20\\
\midrule
\addlinespace[0.3em]
\multicolumn{13}{l}{\textbf{$c = 3$}}\\
\hspace{1em}50 & 1.47 & 0.60 & -313.18 & 2.24 & 9.72 & 16.80 & 31.04 & -363.41 & 2.64 & 10.56 & 18.00 & 31.56\\
\hspace{1em}200 & 1.57 & 0.50 & -416.98 & 0.76 & 3.36 & 7.28 & 14.80 & -436.73 & 0.44 & 3.44 & 7.20 & 14.36\\
\hspace{1em}1000 & 1.93 & 0.26 & -473.03 & 1.20 & 5.20 & 9.52 & 18.68 & -488.62 & 1.04 & 5.08 & 9.88 & 19.92\\
\addlinespace[0.3em]
\multicolumn{13}{l}{\textbf{$c = 4$}}\\
\hspace{1em}50 & 2.43 & 1.18 & -137.21 & 2.12 & 9.56 & 15.52 & 26.32 & -244.48 & 2.12 & 8.52 & 15.32 & 25.96\\
\hspace{1em}200 & 2.45 & 0.59 & -234.49 & 0.64 & 4.08 & 7.80 & 15.60 & -289.46 & 1.00 & 3.36 & 7.32 & 15.76\\
\hspace{1em}1000 & 2.64 & 0.48 & -355.21 & 0.72 & 4.60 & 9.48 & 18.40 & -390.93 & 0.80 & 4.16 & 8.36 & 16.60\\
\addlinespace[0.3em]
\multicolumn{13}{l}{\textbf{$c = 5$}}\\
\hspace{1em}50 & 3.91 & 2.27 & 29.62 & 3.24 & 9.20 & 15.56 & 25.68 & -168.98 & 1.52 & 6.72 & 12.16 & 21.60\\
\hspace{1em}200 & 3.85 & 1.07 & -70.54 & 1.84 & 7.16 & 12.72 & 24.64 & -184.19 & 2.12 & 7.56 & 13.24 & 24.44\\
\hspace{1em}1000 & 4.02 & 0.53 & -207.77 & 1.92 & 7.28 & 13.96 & 24.84 & -262.03 & 1.80 & 7.36 & 14.40 & 26.08\\
\addlinespace[0.3em]
\multicolumn{13}{l}{\textbf{$c = 7$}}\\
\hspace{1em}50 & 10.39 & 7.65 & 332.29 & 13.52 & 25.36 & 34.76 & 45.12 & -91.57 & 2.44 & 7.72 & 13.64 & 23.92\\
\hspace{1em}200 & 9.59 & 3.53 & 227.29 & 12.68 & 28.12 & 39.56 & 54.00 & -85.21 & 4.40 & 12.16 & 19.80 & 31.52\\
\hspace{1em}1000 & 9.71 & 1.59 & 86.75 & 13.00 & 28.24 & 39.72 & 55.20 & -112.72 & 6.56 & 19.60 & 29.96 & 44.36\\
\addlinespace[0.3em]
\multicolumn{13}{l}{\textbf{$c = 8$}}\\
\hspace{1em}50 & 16.41 & 11.60 & 473.32 & 18.96 & 33.04 & 42.32 & 52.56 & -75.49 & 2.44 & 8.32 & 14.04 & 23.24\\
\hspace{1em}200 & 15.09 & 6.62 & 360.73 & 19.28 & 34.92 & 44.32 & 56.44 & -59.58 & 4.32 & 13.16 & 21.04 & 33.24\\
\hspace{1em}1000 & 15.12 & 2.83 & 228.74 & 22.56 & 39.12 & 50.28 & 63.60 & -74.90 & 6.40 & 16.36 & 24.76 & 36.76\\
\bottomrule
\end{tabular}
\end{table}

%% file: tex_tables/tbl_psi1_oracle_kn_gaussian_nu_900.tex
\begin{table}

\caption{\label{tab:psi1_oracle_kn_gaussian_nu_900}The empirical power (in percentages) of our small-uniform $W_n$ statistic along with \cite{cardot2003testing}'s $D_n$ and $T_n$ statistics when $\rho(t) = \rho_0(t) \equiv 0$ and $\varepsilon_i$ has a $\mathcal{N}(0, \sigma_{\varepsilon}^2)$ distribution with $\sigma_{\varepsilon}^2 = \frac{1 - \mathrm{snr}}{\mathrm{snr}} \mathrm{Var}(\inner{X}{\rho})$ with $\mathrm{snr} = 10\%$. Here we assume the truncation parameter $k_n$ is known.  The $n$ here refers to sample size, and $c$ here refers to the exponent associated with the definition of $c_n$. The ``log error'' here refers to the average over all the simulations of the log of the error measure as given in \eqref{eq:ErrorMeasure}. Section~\ref{sec:SimTildeWn} describes our procedure to obtain the simulated levels of $\widetilde{W}_n$. The nominal levels of $D_n$ and $T_n$ are based on their respective asymptotic distributions as described in Section~\ref{sec:ComputeWn}.}
\centering
\resizebox{\linewidth}{!}{
\fontsize{10}{12}\selectfont
\begin{tabular}[t]{cccccccccccccccccccc}
\toprule
\multicolumn{2}{c}{ } & \multicolumn{5}{c}{\textbf{$\boldsymbol{\widetilde{W}_n}$ (simple regularization)}} & \multicolumn{5}{c}{\textbf{$\boldsymbol{\widetilde{W}_n}$ (ridge regularization)}} & \multicolumn{4}{c}{ } & \multicolumn{4}{c}{ } \\
\cmidrule(l{3pt}r{3pt}){3-7} \cmidrule(l{3pt}r{3pt}){8-12}
\multicolumn{3}{c}{ } & \multicolumn{4}{c}{\textbf{Simulated level}} & \multicolumn{1}{c}{ } & \multicolumn{4}{c}{\textbf{Simulated level}} & \multicolumn{4}{c}{\textbf{Nominal level of $\boldsymbol{D_n}$}} & \multicolumn{4}{c}{\textbf{Nominal level of $\boldsymbol{T_n}$}} \\
\cmidrule(l{3pt}r{3pt}){4-7} \cmidrule(l{3pt}r{3pt}){9-12} \cmidrule(l{3pt}r{3pt}){13-16} \cmidrule(l{3pt}r{3pt}){17-20}
$n$ & $k_n$ & log error & 1 & 5 & 10 & 20 & log error & 1 & 5 & 10 & 20 & 1 & 5 & 10 & 20 & 1 & 5 & 10 & 20\\
\midrule
\addlinespace[0.3em]
\multicolumn{20}{l}{\textbf{$c = 3$}}\\
\hspace{1em}50 & 2 & -330.34 & 2.40 & 9.88 & 17.20 & 31.16 & -366.85 & 2.60 & 9.40 & 17.20 & 31.12 & 0.92 & 4.92 & 9.72 & 20.36 & 2.60 & 5.12 & 7.12 & 10.04\\
\hspace{1em}200 & 2 & -313.85 & 0.80 & 3.56 & 6.56 & 13.92 & -348.63 & 0.36 & 2.88 & 6.60 & 14.48 & 1.12 & 4.88 & 9.16 & 19.52 & 3.04 & 4.96 & 6.96 & 9.52\\
\hspace{1em}1000 & 2 & -452.11 & 0.88 & 5.08 & 9.08 & 19.56 & -472.10 & 0.48 & 3.92 & 8.88 & 18.16 & 1.04 & 5.32 & 9.16 & 20.08 & 3.12 & 5.44 & 6.92 & 9.44\\
\addlinespace[0.3em]
\multicolumn{20}{l}{\textbf{$c = 4$}}\\
\hspace{1em}50 & 3 & -124.35 & 2.88 & 9.36 & 15.96 & 27.44 & -224.59 & 2.64 & 8.88 & 15.64 & 27.88 & 1.08 & 5.24 & 10.36 & 19.56 & 2.76 & 5.40 & 7.88 & 10.96\\
\hspace{1em}200 & 3 & -201.98 & 0.60 & 3.32 & 7.28 & 15.00 & -266.06 & 0.68 & 3.72 & 7.04 & 13.68 & 0.88 & 4.72 & 9.24 & 19.16 & 1.88 & 4.76 & 6.40 & 9.56\\
\hspace{1em}1000 & 3 & -302.72 & 0.60 & 3.80 & 7.80 & 15.04 & -341.54 & 0.80 & 4.44 & 9.36 & 17.20 & 0.72 & 4.32 & 9.48 & 18.24 & 2.20 & 4.36 & 6.28 & 9.92\\
\addlinespace[0.3em]
\multicolumn{20}{l}{\textbf{$c = 5$}}\\
\hspace{1em}50 & 4 & 57.02 & 2.88 & 9.12 & 15.88 & 24.76 & -141.90 & 1.80 & 7.00 & 12.24 & 22.32 & 0.76 & 5.12 & 10.44 & 20.16 & 2.64 & 5.04 & 7.24 & 12.80\\
\hspace{1em}200 & 4 & -41.28 & 1.76 & 7.80 & 14.32 & 25.24 & -172.30 & 1.32 & 6.96 & 13.24 & 23.04 & 0.80 & 5.04 & 10.24 & 19.76 & 2.36 & 5.00 & 7.44 & 12.20\\
\hspace{1em}1000 & 5 & -199.48 & 2.04 & 6.84 & 13.52 & 24.56 & -261.46 & 2.00 & 7.60 & 14.20 & 25.00 & 1.24 & 5.12 & 9.80 & 19.12 & 2.36 & 5.00 & 6.84 & 13.64\\
\addlinespace[0.3em]
\multicolumn{20}{l}{\textbf{$c = 7$}}\\
\hspace{1em}50 & 9 & 358.46 & 13.96 & 26.52 & 34.44 & 45.64 & -74.80 & 3.00 & 8.84 & 14.80 & 23.92 & 1.28 & 5.40 & 10.64 & 20.24 & 2.04 & 4.52 & 8.20 & 17.68\\
\hspace{1em}200 & 10 & 256.43 & 11.08 & 24.84 & 36.08 & 51.00 & -73.54 & 4.32 & 13.88 & 21.80 & 33.72 & 1.40 & 5.60 & 11.16 & 21.24 & 2.28 & 5.04 & 8.32 & 19.04\\
\hspace{1em}1000 & 11 & 122.28 & 11.64 & 29.28 & 40.64 & 56.68 & -100.18 & 6.12 & 17.40 & 27.08 & 41.04 & 1.32 & 5.36 & 11.24 & 22.20 & 2.36 & 4.76 & 8.80 & 19.16\\
\addlinespace[0.3em]
\multicolumn{20}{l}{\textbf{$c = 8$}}\\
\hspace{1em}50 & 14 & 501.98 & 19.24 & 32.88 & 42.36 & 54.24 & -56.42 & 2.32 & 8.04 & 14.48 & 24.16 & 1.60 & 6.60 & 12.04 & 22.36 & 2.32 & 5.72 & 10.44 & 20.84\\
\hspace{1em}200 & 16 & 401.06 & 20.04 & 36.96 & 46.76 & 59.60 & -50.26 & 5.16 & 14.00 & 21.36 & 33.16 & 1.88 & 6.36 & 11.92 & 23.44 & 2.52 & 5.56 & 10.00 & 20.28\\
\hspace{1em}1000 & 17 & 264.41 & 22.44 & 40.72 & 51.48 & 64.32 & -65.30 & 6.88 & 16.72 & 25.92 & 39.08 & 1.88 & 7.00 & 13.20 & 23.24 & 2.48 & 6.08 & 11.68 & 21.12\\
\bottomrule
\end{tabular}}
\end{table}

%% file: tex_tables/tbl_psi1_data_based_kn_gaussian_nu_900.tex
\begin{table}

\caption{\label{tab:psi1_data_based_kn_gaussian_nu_900}The empirical power (in percentages) of our small-uniform $W_n$ statistic when $\rho(t) = \rho_0(t) \equiv 0$ and $\varepsilon_i$ has a $\mathcal{N}(0, \sigma_{\varepsilon}^2)$ distribution with $\sigma_{\varepsilon}^2 = \frac{1 - \mathrm{snr}}{\mathrm{snr}} \mathrm{Var}(\inner{X}{\rho})$ with $\mathrm{snr} = 10\%$. Here we use a data driven truncation parameter $\hat{k}_n$. In particular, ``$\hat{k}_n$ avg'' is the average truncation value over $n_s$ number of simulations, and ``$\hat{k}_n$ std'' is the associated standard error. The $n$ here refers to sample size, and $c$ here refers to the exponent associated with the definition of $c_n$. The ``log error'' here refers to the average over all the simulations of the log of the error measure as given in \eqref{eq:ErrorMeasure}. Section~\ref{sec:SimTildeWn} describes our procedure to obtain the simulated levels of $\widetilde{W}_n$. }
\centering
\fontsize{10}{12}\selectfont
\begin{tabular}[t]{rrrrrrrrrrrrr}
\toprule
\multicolumn{3}{c}{ } & \multicolumn{5}{c}{\textbf{Simulated level of $\boldsymbol{\widetilde{W}_n}$ (simple)}} & \multicolumn{5}{c}{\textbf{Simulated level of $\boldsymbol{\widetilde{W}_n}$ (ridge)}} \\
\cmidrule(l{3pt}r{3pt}){4-8} \cmidrule(l{3pt}r{3pt}){9-13}
$n$ & $\hat{k}_n$ avg & $\hat{k}_n$ std & log error & 1 & 5 & 10 & 20 & log error & 1 & 5 & 10 & 20\\
\midrule
\addlinespace[0.3em]
\multicolumn{13}{l}{\textbf{$c = 3$}}\\
\hspace{1em}50 & 1.48 & 0.61 & -311.28 & 2.28 & 9.88 & 17.20 & 31.16 & -362.82 & 2.64 & 9.40 & 17.20 & 31.08\\
\hspace{1em}200 & 1.57 & 0.50 & -417.98 & 0.88 & 3.60 & 6.72 & 13.92 & -439.45 & 0.36 & 2.88 & 6.60 & 14.48\\
\hspace{1em}1000 & 1.93 & 0.26 & -469.98 & 0.88 & 5.08 & 9.00 & 19.56 & -488.47 & 0.48 & 3.92 & 8.88 & 18.16\\
\addlinespace[0.3em]
\multicolumn{13}{l}{\textbf{$c = 4$}}\\
\hspace{1em}50 & 2.40 & 1.11 & -130.99 & 2.88 & 9.36 & 15.88 & 27.32 & -246.59 & 2.56 & 8.80 & 15.64 & 27.76\\
\hspace{1em}200 & 2.45 & 0.60 & -229.70 & 0.60 & 3.32 & 7.36 & 15.36 & -289.38 & 0.68 & 3.60 & 6.96 & 13.56\\
\hspace{1em}1000 & 2.63 & 0.48 & -357.13 & 0.60 & 3.84 & 7.84 & 15.04 & -386.91 & 0.76 & 4.28 & 9.08 & 16.88\\
\addlinespace[0.3em]
\multicolumn{13}{l}{\textbf{$c = 5$}}\\
\hspace{1em}50 & 3.89 & 2.25 & 29.26 & 2.76 & 8.92 & 15.48 & 24.32 & -167.34 & 1.72 & 6.64 & 11.84 & 21.60\\
\hspace{1em}200 & 3.83 & 1.04 & -70.54 & 1.72 & 7.80 & 14.40 & 25.44 & -189.46 & 1.28 & 6.80 & 12.68 & 22.04\\
\hspace{1em}1000 & 4.03 & 0.52 & -203.37 & 2.12 & 7.00 & 13.80 & 24.84 & -264.43 & 2.00 & 7.72 & 14.32 & 25.12\\
\addlinespace[0.3em]
\multicolumn{13}{l}{\textbf{$c = 7$}}\\
\hspace{1em}50 & 10.68 & 8.11 & 336.74 & 13.80 & 26.40 & 34.12 & 45.28 & -92.96 & 3.00 & 8.88 & 14.92 & 24.40\\
\hspace{1em}200 & 9.66 & 3.64 & 226.90 & 11.08 & 25.12 & 36.68 & 52.12 & -85.96 & 4.32 & 14.20 & 22.48 & 35.04\\
\hspace{1em}1000 & 9.76 & 1.60 & 87.98 & 11.68 & 29.28 & 40.52 & 56.60 & -112.98 & 6.36 & 17.72 & 27.32 & 41.40\\
\addlinespace[0.3em]
\multicolumn{13}{l}{\textbf{$c = 8$}}\\
\hspace{1em}50 & 16.05 & 11.21 & 472.39 & 19.12 & 32.80 & 42.00 & 53.84 & -69.30 & 2.20 & 7.84 & 14.28 & 23.76\\
\hspace{1em}200 & 15.39 & 6.86 & 367.03 & 20.48 & 37.32 & 47.24 & 59.76 & -59.13 & 5.16 & 13.88 & 21.20 & 32.72\\
\hspace{1em}1000 & 15.15 & 2.85 & 225.53 & 22.56 & 40.92 & 51.60 & 64.92 & -74.06 & 6.80 & 16.64 & 25.44 & 38.60\\
\bottomrule
\end{tabular}
\end{table}

%% file: tex_tables/tbl_psi2_oracle_kn_gaussian_nu_1900.tex
\begin{table}

\caption{\label{tab:psi2_oracle_kn_gaussian_nu_1900}The empirical power (in percentages) of our small-uniform $W_n$ statistic along with \cite{cardot2003testing}'s $D_n$ and $T_n$ statistics when $\rho(t) = \rho_1(t) = \sin(\pi t / 2) + \frac{1}{2} \sin(3 \pi t / 2) + \frac{1}{4} \sin(5 \pi  t / 2)$ and $\varepsilon_i$ has a $\mathcal{N}(0, \sigma_{\varepsilon}^2)$ distribution with $\sigma_{\varepsilon}^2 = \frac{1 - \mathrm{snr}}{\mathrm{snr}} \mathrm{Var}(\inner{X}{\rho})$ with $\mathrm{snr} = 5\%$. Here we assume the truncation parameter $k_n$ is known.  The $n$ here refers to sample size, and $c$ here refers to the exponent associated with the definition of $c_n$. The ``log error'' here refers to the average over all the simulations of the log of the error measure as given in \eqref{eq:ErrorMeasure}. Section~\ref{sec:SimTildeWn} describes our procedure to obtain the simulated levels of $\widetilde{W}_n$. The nominal levels of $D_n$ and $T_n$ are based on their respective asymptotic distributions as described in Section~\ref{sec:ComputeWn}.}
\centering
\resizebox{\linewidth}{!}{
\fontsize{10}{12}\selectfont
\begin{tabular}[t]{cccccccccccccccccccc}
\toprule
\multicolumn{2}{c}{ } & \multicolumn{5}{c}{\textbf{$\boldsymbol{\widetilde{W}_n}$ (simple regularization)}} & \multicolumn{5}{c}{\textbf{$\boldsymbol{\widetilde{W}_n}$ (ridge regularization)}} & \multicolumn{4}{c}{ } & \multicolumn{4}{c}{ } \\
\cmidrule(l{3pt}r{3pt}){3-7} \cmidrule(l{3pt}r{3pt}){8-12}
\multicolumn{3}{c}{ } & \multicolumn{4}{c}{\textbf{Simulated level}} & \multicolumn{1}{c}{ } & \multicolumn{4}{c}{\textbf{Simulated level}} & \multicolumn{4}{c}{\textbf{Nominal level of $\boldsymbol{D_n}$}} & \multicolumn{4}{c}{\textbf{Nominal level of $\boldsymbol{T_n}$}} \\
\cmidrule(l{3pt}r{3pt}){4-7} \cmidrule(l{3pt}r{3pt}){9-12} \cmidrule(l{3pt}r{3pt}){13-16} \cmidrule(l{3pt}r{3pt}){17-20}
$n$ & $k_n$ & log error & 1 & 5 & 10 & 20 & log error & 1 & 5 & 10 & 20 & 1 & 5 & 10 & 20 & 1 & 5 & 10 & 20\\
\midrule
\addlinespace[0.3em]
\multicolumn{20}{l}{\textbf{$c = 3$}}\\
\hspace{1em}50 & 2 & -43.90 & 20.16 & 41.92 & 55.28 & 69.20 & -70.48 & 20.36 & 40.76 & 54.28 & 69.00 & 11.60 & 28.44 & 40.72 & 55.84 & 21.00 & 28.80 & 33.92 & 41.24\\
\hspace{1em}200 & 2 & -88.55 & 62.76 & 83.32 & 89.88 & 94.92 & -117.76 & 63.96 & 83.04 & 89.88 & 94.84 & 65.68 & 85.64 & 91.44 & 95.72 & 79.12 & 85.84 & 88.68 & 91.56\\
\hspace{1em}1000 & 2 & -196.76 & 100.00 & 100.00 & 100.00 & 100.00 & -206.30 & 100.00 & 100.00 & 100.00 & 100.00 & 100.00 & 100.00 & 100.00 & 100.00 & 100.00 & 100.00 & 100.00 & 100.00\\
\addlinespace[0.3em]
\multicolumn{20}{l}{\textbf{$c = 4$}}\\
\hspace{1em}50 & 3 & 56.50 & 13.88 & 33.04 & 45.28 & 59.52 & -31.49 & 17.00 & 35.76 & 48.36 & 64.28 & 9.48 & 24.16 & 35.88 & 50.96 & 16.80 & 24.32 & 29.44 & 36.96\\
\hspace{1em}200 & 3 & -10.33 & 57.80 & 78.36 & 85.92 & 92.84 & -68.13 & 58.56 & 77.72 & 86.80 & 92.76 & 59.20 & 79.28 & 86.44 & 93.56 & 71.76 & 79.36 & 83.04 & 86.84\\
\hspace{1em}1000 & 3 & -117.00 & 100.00 & 100.00 & 100.00 & 100.00 & -148.56 & 100.00 & 100.00 & 100.00 & 100.00 & 100.00 & 100.00 & 100.00 & 100.00 & 100.00 & 100.00 & 100.00 & 100.00\\
\addlinespace[0.3em]
\multicolumn{20}{l}{\textbf{$c = 5$}}\\
\hspace{1em}50 & 4 & 238.38 & 11.04 & 28.80 & 39.52 & 53.56 & 43.23 & 9.92 & 25.80 & 37.20 & 51.00 & 7.72 & 22.72 & 33.68 & 48.92 & 14.16 & 22.36 & 27.84 & 35.36\\
\hspace{1em}200 & 4 & 137.54 & 58.84 & 78.12 & 85.60 & 91.80 & 17.25 & 59.20 & 79.60 & 86.84 & 92.96 & 52.40 & 73.40 & 82.08 & 89.56 & 63.76 & 73.04 & 77.64 & 82.80\\
\hspace{1em}1000 & 5 & -8.25 & 100.00 & 100.00 & 100.00 & 100.00 & -64.56 & 100.00 & 100.00 & 100.00 & 100.00 & 100.00 & 100.00 & 100.00 & 100.00 & 100.00 & 100.00 & 100.00 & 100.00\\
\addlinespace[0.3em]
\multicolumn{20}{l}{\textbf{$c = 7$}}\\
\hspace{1em}50 & 9 & 539.84 & 27.32 & 43.40 & 53.80 & 65.88 & 109.07 & 11.56 & 25.44 & 34.96 & 47.92 & 6.24 & 17.24 & 27.64 & 40.68 & 9.36 & 16.00 & 20.92 & 31.28\\
\hspace{1em}200 & 10 & 450.40 & 70.16 & 84.84 & 90.32 & 94.84 & 114.77 & 57.84 & 76.20 & 83.88 & 91.28 & 40.28 & 61.92 & 72.64 & 83.20 & 47.28 & 59.96 & 66.32 & 73.88\\
\hspace{1em}1000 & 11 & 376.49 & 100.00 & 100.00 & 100.00 & 100.00 & 133.72 & 100.00 & 100.00 & 100.00 & 100.00 & 99.96 & 100.00 & 100.00 & 100.00 & 100.00 & 100.00 & 100.00 & 100.00\\
\addlinespace[0.3em]
\multicolumn{20}{l}{\textbf{$c = 8$}}\\
\hspace{1em}50 & 14 & 692.65 & 38.80 & 53.52 & 61.60 & 72.08 & 125.39 & 10.32 & 22.64 & 32.56 & 44.20 & 8.36 & 20.72 & 29.04 & 43.20 & 11.08 & 18.68 & 23.92 & 33.48\\
\hspace{1em}200 & 16 & 617.61 & 81.88 & 91.04 & 94.24 & 96.76 & 141.20 & 60.28 & 77.24 & 85.64 & 91.20 & 45.84 & 66.40 & 75.80 & 84.52 & 51.08 & 63.80 & 69.84 & 76.88\\
\hspace{1em}1000 & 17 & 567.22 & 100.00 & 100.00 & 100.00 & 100.00 & 183.44 & 100.00 & 100.00 & 100.00 & 100.00 & 99.96 & 100.00 & 100.00 & 100.00 & 100.00 & 100.00 & 100.00 & 100.00\\
\bottomrule
\end{tabular}}
\end{table}

%% file: tex_tables/tbl_psi2_oracle_kn_gaussian_nu_900.tex
\begin{table}

\caption{\label{tab:psi2_oracle_kn_gaussian_nu_900}The empirical power (in percentages) of our small-uniform $W_n$ statistic along with \cite{cardot2003testing}'s $D_n$ and $T_n$ statistics when $\rho(t) = \rho_1(t) = \sin(\pi t / 2) + \frac{1}{2} \sin(3 \pi t / 2) + \frac{1}{4} \sin(5 \pi  t / 2)$ and $\varepsilon_i$ has a $\mathcal{N}(0, \sigma_{\varepsilon}^2)$ distribution with $\sigma_{\varepsilon}^2 = \frac{1 - \mathrm{snr}}{\mathrm{snr}} \mathrm{Var}(\inner{X}{\rho})$ with $\mathrm{snr} = 10\%$. Here we assume the truncation parameter $k_n$ is known.  The $n$ here refers to sample size, and $c$ here refers to the exponent associated with the definition of $c_n$. The ``log error'' here refers to the average over all the simulations of the log of the error measure as given in \eqref{eq:ErrorMeasure}. Section~\ref{sec:SimTildeWn} describes our procedure to obtain the simulated levels of $\widetilde{W}_n$. The nominal levels of $D_n$ and $T_n$ are based on their respective asymptotic distributions as described in Section~\ref{sec:ComputeWn}.}
\centering
\resizebox{\linewidth}{!}{
\fontsize{10}{12}\selectfont
\begin{tabular}[t]{cccccccccccccccccccc}
\toprule
\multicolumn{2}{c}{ } & \multicolumn{5}{c}{\textbf{$\boldsymbol{\widetilde{W}_n}$ (simple regularization)}} & \multicolumn{5}{c}{\textbf{$\boldsymbol{\widetilde{W}_n}$ (ridge regularization)}} & \multicolumn{4}{c}{ } & \multicolumn{4}{c}{ } \\
\cmidrule(l{3pt}r{3pt}){3-7} \cmidrule(l{3pt}r{3pt}){8-12}
\multicolumn{3}{c}{ } & \multicolumn{4}{c}{\textbf{Simulated level}} & \multicolumn{1}{c}{ } & \multicolumn{4}{c}{\textbf{Simulated level}} & \multicolumn{4}{c}{\textbf{Nominal level of $\boldsymbol{D_n}$}} & \multicolumn{4}{c}{\textbf{Nominal level of $\boldsymbol{T_n}$}} \\
\cmidrule(l{3pt}r{3pt}){4-7} \cmidrule(l{3pt}r{3pt}){9-12} \cmidrule(l{3pt}r{3pt}){13-16} \cmidrule(l{3pt}r{3pt}){17-20}
$n$ & $k_n$ & log error & 1 & 5 & 10 & 20 & log error & 1 & 5 & 10 & 20 & 1 & 5 & 10 & 20 & 1 & 5 & 10 & 20\\
\midrule
\addlinespace[0.3em]
\multicolumn{20}{l}{\textbf{$c = 3$}}\\
\hspace{1em}50 & 2 & -84.04 & 45.64 & 69.24 & 79.76 & 88.36 & -111.81 & 45.04 & 69.32 & 79.80 & 88.48 & 31.16 & 55.08 & 66.00 & 78.96 & 46.00 & 55.60 & 60.44 & 66.40\\
\hspace{1em}200 & 2 & -138.32 & 96.48 & 99.12 & 99.48 & 99.72 & -169.19 & 96.60 & 99.36 & 99.56 & 99.88 & 96.84 & 99.20 & 99.48 & 99.72 & 98.80 & 99.24 & 99.36 & 99.48\\
\hspace{1em}1000 & 2 & -234.36 & 100.00 & 100.00 & 100.00 & 100.00 & -245.04 & 100.00 & 100.00 & 100.00 & 100.00 & 100.00 & 100.00 & 100.00 & 100.00 & 100.00 & 100.00 & 100.00 & 100.00\\
\addlinespace[0.3em]
\multicolumn{20}{l}{\textbf{$c = 4$}}\\
\hspace{1em}50 & 3 & -8.32 & 35.12 & 59.00 & 70.60 & 82.40 & -89.08 & 34.04 & 58.68 & 71.08 & 82.24 & 26.80 & 49.28 & 61.80 & 75.32 & 38.76 & 49.40 & 55.52 & 62.56\\
\hspace{1em}200 & 3 & -70.53 & 94.84 & 98.84 & 99.40 & 99.84 & -134.69 & 94.68 & 98.76 & 99.32 & 99.80 & 94.92 & 98.56 & 99.52 & 99.84 & 97.44 & 98.56 & 99.08 & 99.52\\
\hspace{1em}1000 & 3 & -175.07 & 100.00 & 100.00 & 100.00 & 100.00 & -209.87 & 100.00 & 100.00 & 100.00 & 100.00 & 100.00 & 100.00 & 100.00 & 100.00 & 100.00 & 100.00 & 100.00 & 100.00\\
\addlinespace[0.3em]
\multicolumn{20}{l}{\textbf{$c = 5$}}\\
\hspace{1em}50 & 4 & 162.68 & 29.00 & 51.52 & 64.64 & 77.20 & -26.20 & 24.76 & 45.36 & 57.20 & 71.80 & 22.28 & 44.08 & 56.40 & 71.04 & 32.28 & 43.56 & 50.52 & 57.48\\
\hspace{1em}200 & 4 & 70.80 & 95.08 & 98.68 & 99.44 & 99.80 & -53.39 & 93.72 & 98.28 & 99.16 & 99.80 & 92.96 & 97.88 & 99.12 & 99.60 & 96.16 & 97.88 & 98.36 & 99.20\\
\hspace{1em}1000 & 5 & -65.50 & 100.00 & 100.00 & 100.00 & 100.00 & -121.91 & 100.00 & 100.00 & 100.00 & 100.00 & 100.00 & 100.00 & 100.00 & 100.00 & 100.00 & 100.00 & 100.00 & 100.00\\
\addlinespace[0.3em]
\multicolumn{20}{l}{\textbf{$c = 7$}}\\
\hspace{1em}50 & 9 & 469.73 & 43.40 & 62.52 & 72.40 & 81.20 & 40.43 & 23.88 & 44.60 & 55.88 & 69.36 & 17.16 & 35.60 & 46.48 & 59.88 & 22.16 & 33.76 & 40.24 & 48.32\\
\hspace{1em}200 & 10 & 394.15 & 96.32 & 99.12 & 99.60 & 99.76 & 53.64 & 93.68 & 97.68 & 98.72 & 99.56 & 84.20 & 94.48 & 96.68 & 98.80 & 88.20 & 93.44 & 95.64 & 96.92\\
\hspace{1em}1000 & 11 & 353.78 & 100.00 & 100.00 & 100.00 & 100.00 & 95.27 & 100.00 & 100.00 & 100.00 & 100.00 & 100.00 & 100.00 & 100.00 & 100.00 & 100.00 & 100.00 & 100.00 & 100.00\\
\addlinespace[0.3em]
\multicolumn{20}{l}{\textbf{$c = 8$}}\\
\hspace{1em}50 & 14 & 624.19 & 54.16 & 68.96 & 76.20 & 84.00 & 55.43 & 23.52 & 42.60 & 53.24 & 64.68 & 19.44 & 35.40 & 46.00 & 59.76 & 23.16 & 33.00 & 39.32 & 48.00\\
\hspace{1em}200 & 16 & 568.11 & 98.12 & 99.20 & 99.84 & 99.96 & 79.12 & 93.20 & 97.40 & 98.64 & 99.60 & 85.36 & 92.88 & 95.84 & 97.88 & 87.48 & 91.84 & 93.72 & 96.12\\
\hspace{1em}1000 & 17 & 553.00 & 100.00 & 100.00 & 100.00 & 100.00 & 154.03 & 100.00 & 100.00 & 100.00 & 100.00 & 100.00 & 100.00 & 100.00 & 100.00 & 100.00 & 100.00 & 100.00 & 100.00\\
\bottomrule
\end{tabular}}
\end{table}

%% file: tex_tables/tbl_psi3_oracle_kn_gaussian_nu_1900.tex
\begin{table}

\caption{\label{tab:psi3_oracle_kn_gaussian_nu_1900}The empirical power (in percentages) of our small-uniform $W_n$ statistic along with \cite{cardot2003testing}'s $D_n$ and $T_n$ statistics when $\rho(t) = \rho_2(t) = \sin(2 \pi t^3)^3$ and $\varepsilon_i$ has a $\mathcal{N}(0, \sigma_{\varepsilon}^2)$ distribution with $\sigma_{\varepsilon}^2 = \frac{1 - \mathrm{snr}}{\mathrm{snr}} \mathrm{Var}(\inner{X}{\rho})$ with $\mathrm{snr} = 5\%$. Here we assume the truncation parameter $k_n$ is known.  The $n$ here refers to sample size, and $c$ here refers to the exponent associated with the definition of $c_n$. The ``log error'' here refers to the average over all the simulations of the log of the error measure as given in \eqref{eq:ErrorMeasure}. Section~\ref{sec:SimTildeWn} describes our procedure to obtain the simulated levels of $\widetilde{W}_n$. The nominal levels of $D_n$ and $T_n$ are based on their respective asymptotic distributions as described in Section~\ref{sec:ComputeWn}.}
\centering
\resizebox{\linewidth}{!}{
\fontsize{10}{12}\selectfont
\begin{tabular}[t]{cccccccccccccccccccc}
\toprule
\multicolumn{2}{c}{ } & \multicolumn{5}{c}{\textbf{$\boldsymbol{\widetilde{W}_n}$ (simple regularization)}} & \multicolumn{5}{c}{\textbf{$\boldsymbol{\widetilde{W}_n}$ (ridge regularization)}} & \multicolumn{4}{c}{ } & \multicolumn{4}{c}{ } \\
\cmidrule(l{3pt}r{3pt}){3-7} \cmidrule(l{3pt}r{3pt}){8-12}
\multicolumn{3}{c}{ } & \multicolumn{4}{c}{\textbf{Simulated level}} & \multicolumn{1}{c}{ } & \multicolumn{4}{c}{\textbf{Simulated level}} & \multicolumn{4}{c}{\textbf{Nominal level of $\boldsymbol{D_n}$}} & \multicolumn{4}{c}{\textbf{Nominal level of $\boldsymbol{T_n}$}} \\
\cmidrule(l{3pt}r{3pt}){4-7} \cmidrule(l{3pt}r{3pt}){9-12} \cmidrule(l{3pt}r{3pt}){13-16} \cmidrule(l{3pt}r{3pt}){17-20}
$n$ & $k_n$ & log error & 1 & 5 & 10 & 20 & log error & 1 & 5 & 10 & 20 & 1 & 5 & 10 & 20 & 1 & 5 & 10 & 20\\
\midrule
\addlinespace[0.3em]
\multicolumn{20}{l}{\textbf{$c = 3$}}\\
\hspace{1em}50 & 2 & 0.08 & 12.84 & 31.64 & 44.00 & 59.76 & -4.85 & 13.36 & 31.40 & 43.64 & 58.56 & 9.08 & 22.64 & 32.12 & 47.84 & 15.92 & 23.20 & 27.56 & 32.24\\
\hspace{1em}200 & 2 & -13.51 & 41.60 & 62.60 & 73.68 & 82.80 & -16.30 & 39.52 & 62.20 & 73.20 & 82.64 & 50.84 & 71.72 & 80.92 & 88.72 & 64.40 & 72.32 & 76.72 & 81.08\\
\hspace{1em}1000 & 2 & -21.50 & 99.88 & 100.00 & 100.00 & 100.00 & -21.66 & 99.92 & 100.00 & 100.00 & 100.00 & 99.92 & 100.00 & 100.00 & 100.00 & 100.00 & 100.00 & 100.00 & 100.00\\
\addlinespace[0.3em]
\multicolumn{20}{l}{\textbf{$c = 4$}}\\
\hspace{1em}50 & 3 & 4.12 & 12.68 & 29.00 & 41.00 & 56.80 & -12.96 & 12.48 & 29.84 & 40.64 & 55.00 & 10.28 & 24.76 & 35.56 & 50.60 & 17.20 & 24.84 & 30.00 & 36.40\\
\hspace{1em}200 & 3 & -43.32 & 48.80 & 69.60 & 79.12 & 87.24 & -50.12 & 48.56 & 68.84 & 79.36 & 87.56 & 57.84 & 77.96 & 85.88 & 92.64 & 69.32 & 78.08 & 81.92 & 86.76\\
\hspace{1em}1000 & 3 & -100.84 & 100.00 & 100.00 & 100.00 & 100.00 & -102.20 & 99.92 & 100.00 & 100.00 & 100.00 & 100.00 & 100.00 & 100.00 & 100.00 & 100.00 & 100.00 & 100.00 & 100.00\\
\addlinespace[0.3em]
\multicolumn{20}{l}{\textbf{$c = 5$}}\\
\hspace{1em}50 & 4 & 53.53 & 11.64 & 27.48 & 37.24 & 51.76 & -32.05 & 9.92 & 23.56 & 33.96 & 47.52 & 9.00 & 23.52 & 34.08 & 49.08 & 15.24 & 23.16 & 28.32 & 35.48\\
\hspace{1em}200 & 4 & -27.28 & 63.00 & 81.44 & 88.04 & 93.20 & -72.30 & 58.76 & 76.92 & 85.16 & 91.80 & 55.52 & 76.36 & 84.48 & 91.16 & 67.28 & 76.04 & 80.72 & 85.00\\
\hspace{1em}1000 & 5 & -107.67 & 99.96 & 100.00 & 100.00 & 100.00 & -118.97 & 100.00 & 100.00 & 100.00 & 100.00 & 99.96 & 100.00 & 100.00 & 100.00 & 99.96 & 100.00 & 100.00 & 100.00\\
\addlinespace[0.3em]
\multicolumn{20}{l}{\textbf{$c = 7$}}\\
\hspace{1em}50 & 9 & 319.01 & 23.00 & 40.76 & 51.08 & 62.52 & -18.76 & 10.40 & 25.16 & 33.72 & 46.68 & 5.20 & 17.08 & 26.80 & 40.68 & 7.64 & 14.96 & 21.00 & 30.08\\
\hspace{1em}200 & 10 & 223.07 & 68.04 & 84.44 & 90.40 & 95.64 & -44.75 & 60.36 & 77.80 & 84.76 & 91.04 & 38.60 & 60.72 & 71.80 & 82.40 & 46.12 & 59.24 & 65.84 & 73.16\\
\hspace{1em}1000 & 11 & 111.44 & 100.00 & 100.00 & 100.00 & 100.00 & -66.55 & 100.00 & 100.00 & 100.00 & 100.00 & 99.92 & 100.00 & 100.00 & 100.00 & 99.96 & 100.00 & 100.00 & 100.00\\
\addlinespace[0.3em]
\multicolumn{20}{l}{\textbf{$c = 8$}}\\
\hspace{1em}50 & 14 & 466.78 & 31.60 & 48.80 & 58.12 & 68.44 & -11.74 & 10.20 & 24.76 & 34.32 & 46.60 & 5.12 & 16.04 & 26.00 & 39.32 & 7.24 & 14.52 & 20.44 & 30.12\\
\hspace{1em}200 & 16 & 374.11 & 73.68 & 87.40 & 92.04 & 96.24 & -28.19 & 56.96 & 76.12 & 84.08 & 89.92 & 34.28 & 56.00 & 66.80 & 78.60 & 39.40 & 52.28 & 59.56 & 68.28\\
\hspace{1em}1000 & 17 & 276.11 & 100.00 & 100.00 & 100.00 & 100.00 & -35.84 & 100.00 & 100.00 & 100.00 & 100.00 & 99.72 & 100.00 & 100.00 & 100.00 & 99.84 & 100.00 & 100.00 & 100.00\\
\bottomrule
\end{tabular}}
\end{table}

%% file: tex_tables/tbl_psi3_oracle_kn_gaussian_nu_900.tex
\begin{table}

\caption{\label{tab:psi3_oracle_kn_gaussian_nu_900}The empirical power (in percentages) of our small-uniform $W_n$ statistic along with \cite{cardot2003testing}'s $D_n$ and $T_n$ statistics when $\rho(t) = \rho_2(t) = \sin(2 \pi t^3)^3$ and $\varepsilon_i$ has a $\mathcal{N}(0, \sigma_{\varepsilon}^2)$ distribution with $\sigma_{\varepsilon}^2 = \frac{1 - \mathrm{snr}}{\mathrm{snr}} \mathrm{Var}(\inner{X}{\rho})$ with $\mathrm{snr} = 10\%$. Here we assume the truncation parameter $k_n$ is known.  The $n$ here refers to sample size, and $c$ here refers to the exponent associated with the definition of $c_n$. The ``log error'' here refers to the average over all the simulations of the log of the error measure as given in \eqref{eq:ErrorMeasure}. Section~\ref{sec:SimTildeWn} describes our procedure to obtain the simulated levels of $\widetilde{W}_n$. The nominal levels of $D_n$ and $T_n$ are based on their respective asymptotic distributions as described in Section~\ref{sec:ComputeWn}.}
\centering
\resizebox{\linewidth}{!}{
\fontsize{10}{12}\selectfont
\begin{tabular}[t]{cccccccccccccccccccc}
\toprule
\multicolumn{2}{c}{ } & \multicolumn{5}{c}{\textbf{$\boldsymbol{\widetilde{W}_n}$ (simple regularization)}} & \multicolumn{5}{c}{\textbf{$\boldsymbol{\widetilde{W}_n}$ (ridge regularization)}} & \multicolumn{4}{c}{ } & \multicolumn{4}{c}{ } \\
\cmidrule(l{3pt}r{3pt}){3-7} \cmidrule(l{3pt}r{3pt}){8-12}
\multicolumn{3}{c}{ } & \multicolumn{4}{c}{\textbf{Simulated level}} & \multicolumn{1}{c}{ } & \multicolumn{4}{c}{\textbf{Simulated level}} & \multicolumn{4}{c}{\textbf{Nominal level of $\boldsymbol{D_n}$}} & \multicolumn{4}{c}{\textbf{Nominal level of $\boldsymbol{T_n}$}} \\
\cmidrule(l{3pt}r{3pt}){4-7} \cmidrule(l{3pt}r{3pt}){9-12} \cmidrule(l{3pt}r{3pt}){13-16} \cmidrule(l{3pt}r{3pt}){17-20}
$n$ & $k_n$ & log error & 1 & 5 & 10 & 20 & log error & 1 & 5 & 10 & 20 & 1 & 5 & 10 & 20 & 1 & 5 & 10 & 20\\
\midrule
\addlinespace[0.3em]
\multicolumn{20}{l}{\textbf{$c = 3$}}\\
\hspace{1em}50 & 2 & -6.12 & 28.44 & 49.00 & 61.36 & 74.48 & -8.02 & 29.52 & 53.04 & 64.12 & 76.36 & 22.36 & 42.60 & 53.56 & 67.60 & 34.28 & 42.88 & 47.08 & 54.08\\
\hspace{1em}200 & 2 & -18.25 & 80.48 & 92.08 & 95.72 & 98.40 & -18.98 & 80.20 & 92.28 & 95.44 & 97.96 & 87.72 & 96.40 & 97.68 & 99.32 & 93.96 & 96.44 & 97.24 & 97.80\\
\hspace{1em}1000 & 2 & -22.50 & 100.00 & 100.00 & 100.00 & 100.00 & -22.64 & 100.00 & 100.00 & 100.00 & 100.00 & 100.00 & 100.00 & 100.00 & 100.00 & 100.00 & 100.00 & 100.00 & 100.00\\
\addlinespace[0.3em]
\multicolumn{20}{l}{\textbf{$c = 4$}}\\
\hspace{1em}50 & 3 & -12.89 & 28.44 & 49.48 & 61.04 & 74.44 & -20.69 & 29.00 & 49.08 & 61.24 & 73.96 & 25.64 & 46.32 & 59.08 & 71.52 & 36.76 & 46.44 & 52.40 & 60.04\\
\hspace{1em}200 & 3 & -59.21 & 89.80 & 96.36 & 98.08 & 99.16 & -58.12 & 88.16 & 95.92 & 97.72 & 99.04 & 94.16 & 98.36 & 99.16 & 99.72 & 97.08 & 98.36 & 98.80 & 99.20\\
\hspace{1em}1000 & 3 & -108.69 & 100.00 & 100.00 & 100.00 & 100.00 & -106.68 & 100.00 & 100.00 & 100.00 & 100.00 & 100.00 & 100.00 & 100.00 & 100.00 & 100.00 & 100.00 & 100.00 & 100.00\\
\addlinespace[0.3em]
\multicolumn{20}{l}{\textbf{$c = 5$}}\\
\hspace{1em}50 & 4 & 1.66 & 29.60 & 49.80 & 61.44 & 72.20 & -45.34 & 25.04 & 45.36 & 56.84 & 69.44 & 24.64 & 45.80 & 58.64 & 71.68 & 33.64 & 45.32 & 52.32 & 60.08\\
\hspace{1em}200 & 4 & -70.58 & 94.92 & 98.36 & 99.28 & 99.80 & -89.75 & 93.36 & 97.76 & 98.96 & 99.64 & 93.20 & 97.56 & 98.76 & 99.64 & 95.80 & 97.44 & 98.32 & 98.88\\
\hspace{1em}1000 & 5 & -128.01 & 100.00 & 100.00 & 100.00 & 100.00 & -130.83 & 100.00 & 100.00 & 100.00 & 100.00 & 100.00 & 100.00 & 100.00 & 100.00 & 100.00 & 100.00 & 100.00 & 100.00\\
\addlinespace[0.3em]
\multicolumn{20}{l}{\textbf{$c = 7$}}\\
\hspace{1em}50 & 9 & 249.77 & 40.64 & 61.48 & 71.20 & 81.08 & -40.74 & 26.52 & 46.76 & 57.16 & 70.00 & 16.52 & 34.60 & 46.44 & 61.48 & 21.84 & 32.84 & 39.40 & 48.48\\
\hspace{1em}200 & 10 & 157.24 & 94.96 & 98.36 & 99.32 & 99.68 & -74.06 & 93.76 & 98.08 & 99.00 & 99.56 & 83.00 & 94.04 & 96.64 & 98.44 & 87.64 & 93.56 & 95.12 & 96.88\\
\hspace{1em}1000 & 11 & 61.37 & 100.00 & 100.00 & 100.00 & 100.00 & -95.50 & 100.00 & 100.00 & 100.00 & 100.00 & 100.00 & 100.00 & 100.00 & 100.00 & 100.00 & 100.00 & 100.00 & 100.00\\
\addlinespace[0.3em]
\multicolumn{20}{l}{\textbf{$c = 8$}}\\
\hspace{1em}50 & 14 & 395.72 & 47.16 & 64.40 & 72.40 & 81.64 & -36.90 & 25.20 & 44.28 & 54.96 & 67.56 & 13.24 & 31.32 & 41.84 & 56.20 & 17.48 & 28.72 & 35.40 & 44.64\\
\hspace{1em}200 & 16 & 307.97 & 96.08 & 98.68 & 99.24 & 99.76 & -62.10 & 92.12 & 97.24 & 98.48 & 99.36 & 76.64 & 90.12 & 94.12 & 97.16 & 80.64 & 88.60 & 91.68 & 94.52\\
\hspace{1em}1000 & 17 & 236.78 & 100.00 & 100.00 & 100.00 & 100.00 & -65.24 & 100.00 & 100.00 & 100.00 & 100.00 & 100.00 & 100.00 & 100.00 & 100.00 & 100.00 & 100.00 & 100.00 & 100.00\\
\bottomrule
\end{tabular}}
\end{table}

%% file: conclusion.tex
\section{Concluding remarks}
\label{sec:Conclusion}
This paper introduces a small-uniform statistic $W_n$ that is constructed as a fractional programming problem out of the FPCA estimator $\hat{\rho}$ of the slope $\rho$ of the functional linear model. Our main result Theorem~\ref{thm:NormalizedLeungTamStatisticConvergence} shows $W_n$ converges in probability to the supremum of a Gaussian process $\widetilde{W}_n$. The key arguments to showing our main result are by taking advantage of identifying the regressors' underlying Hilbert space with its dual, and also recent advances by \cite{chernozhukov2014gaussian} in studying the suprema of empirical processes indexed by functionals. 

We see two interesting directions in extending the small-uniform statistic.  Firstly, while this paper focuses on the most commonly studied scalar-on-functional FLM, it seems feasible to extend our statistic to a functional-on-functional FLM. Secondly, the recent and growing literature on \emph{functional time series regressions}
\footnote{
	For example, see \cite{panaretos2013fourier} and \cite{hormann2015dynamic}.
}
represent a more exciting challenge of extending our small-uniform statistic. In particular, it is clear one needs a modification of Step I in our proof of Theorem~\ref{thm:NormalizedLeungTamStatisticConvergence} to a functional time series context. More importantly, we conjecture the required extension of our Step II will call for new results in studying the empirical processes constructed out of dependent random variables.

%% file: proofs.tex
\section{Proofs} 
Throughout the proofs, we will use the following asymptotic approximation notations. We will always use $C, c > 0$ to denote universal constants that may change between lines. For $x, y > 0$, we denote $x \asymp y$ to mean $c y \le x \le C y$. For a real sequence $\{ x_n \}$, we will write $x_n \lesssim \BigOh(a_n)$ to mean there exists some sequence $\{y_n\}$ such that $\abs{x_n} \le C \abs{y_n}$, and $\abs{y_n} \le c \abs{a_n}$. Likewise, if $\{ X_n \}$ is a sequence of random variables and $\{ a_n \}$ is a deterministic sequence, we will write $X_n \lesssim \BigOhPee(a_n)$ to mean there exists some sequence of random variables $\{ Y_n \}$ such that $\abs{X_n} \le C \abs{Y_n}$ with $Y_n = \BigOhPee(a_n)$; i.e.\ for all $\epsilon > 0$ there exists $C > 0$ such that $\Prob(\abs{Y_n / a_n} \le C ) \ge 1 - \epsilon$ for all $n$. We will also use $\norm{\cdot}$ to denote the operator norm; that is, for a bounded operator $A \in \mathscr{B}(\HSpace, \HSpace) =: \BddOp$, we denote $\norm{A} := \sup_{\norm{h} \le 1} \norm{A h}$. We will denote the space of compact operators on $\HSpace$ as $\CompactOp$. 

\begin{rem}[Our proof arguments vis-\'{a}-vis that of \cite{cardot2007clt}]
	Our proof arguments of Step I are heavily inspired by \cite{cardot2007clt}. Indeed, our proofs of Propositions~\ref{prop:TermYnGoesToZero} and \ref{prop:TermSnGoesToZero} are heavily based on the arguments of \cite[Propositions 2 and 3]{cardot2007clt}. But we also have two significant deviations. Firstly, a critical difference is that we neither define their set $\mathcal{E}_j(z)$ nor use their Lemma 4. In particular, we can't understand one their key arguments (last displayed equation on their pg 351) which seemingly requires the expression ``$\sup_{z \in \mathcal{B}_j} \mathcal{E}_j(z)$'', of which measureability concerns arise. Instead of pursing this argument direction, we simply recognize that the norm $\norm{(z - \Gamma_n)^{-1}}$ is bounded above by the reciprocal of the distance from $z \in \mathcal{B}_j \subseteq \rho(\Gamma_n)$ to its spectrum $\sigma(\Gamma_n)$. And thanks to the choice of the contours $\mathcal{C}_n$ and the event $\mathcal{A}_n$ from Lemma~\ref{lem:ApproxIntegrationEmpiricalCn}, this reciprocal can be approximated by the reciprocal of the radius of $\mathcal{B}_j$; see \eqref{pfeq:RoseResolventNormEventAn}. This argument allows us to estimate $\norm{(z - \Gamma_n)^{-1}}$ without the need to deal with potential measureability issues associated with the event $\mathcal{E}_j(z)$. 

	Secondly, we do not work with ``square-roots'' of the resolvent; that is, we do not write expressions like ``$(z - \Gamma)^{-1/2}$''. It is unclear to us whether this is necessarily a well-defined object. 
	\footnote{
		In general, if $A$ is self-adjoint, then it is clear that its resolvent $(z - A)^{-1}$ for $z \in \rho(A)$ is also self-adjoint. In particular, being self-adjoint implies it is normal. But conventional definitions of the square-root of an operator require the underlying operator to be normal and compact. Thus to define a square-root ``$(z - A)^{-1/2}$'', it necessarily requires $(z - A)^{-1}$ to be compact (i.e.\ so is in $\CompactOp$). But clearly $(z - A) \in \BddOp$. That implies $(z - A)(z - A)^{-1} = \ident_\HSpace$ is compact --- which is only possible if $\HSpace$ is finite-dimensional (a case which we explicitly do not consider throughout this paper). 
	}
	A lot of the work in our proofs goes to re-deriving the results of \cite{cardot2007clt} using only the resolvent but without invoking a square-root of the resolvent. This partly explains why our convergence rates differ from theirs. 
\end{rem}

Let's first setup some preliminary definitions and results. Let $\im := \sqrt{-1}$. Denote the orientated circle of the complex plane with center $\lambda_i$ and radius $\delta_i / 2$ as $\mathcal{B}_i := \{ \lambda_i + \frac{\delta_i}{2} e^{2\pi \im t} : t \in [0,1] \}$. We also denote the orientated circle $\widehat{\mathcal{B}}_i$ analogously with the center at $\hat{\lambda}_i$ and radius $\hat{\delta}_i / 2$. Define, 
\begin{equation*}
	\mathcal{C}_n := \bigcup_{i = 1}^{k_n} \mathcal{B}_i.
\end{equation*}
With some abuse of notations, for the approximate reciprocal $f_n$ that satisfies Assumption~\ref{as:ConditionF}, we will denote also $f_n$ as its analytic extension to the interior of $\mathcal{C}_n$. By Riesz functional calculus (see \cite{conway1994course} and \cite{kato1995perturbation}), we can define 
\begin{equation}
	\Gamma^\dagger
	:= f_n(\Gamma) 
	= \frac{1}{2 \pi \im} \int_{\mathcal{C}_n} (z - \Gamma)^{-1} f_n(z) \diff{z}.
	\label{eq:GammaDagger}
\end{equation}
Moreover, the projection of $\HSpace$ onto $\mathrm{span}\{ e_1, \ldots, e_{k_n} \}$ can be written as, 
\begin{equation}
	\Pi_{k_n}
	= \frac{1}{2 \pi \im} \int_{\mathcal{C}_n} (z - \Gamma)^{-1} \diff{z}.
	\label{eq:Projection}
\end{equation} 

Define the event, 
\begin{equation}
	\mathcal{A}_n := \bigcap_{j = 1}^{k_n} \left\{ \abs{ \hat{\lambda}_j - \lambda_j } < \frac{\delta_j}{4} \right\} 
	\label{eq:EventAn} 
\end{equation}
The following lemma show that, asymptotically, integrating over a collection of random circle traces centered at the empirical eigenvalues is equivalent to integrating over a collection of deterministic circle traces centered at the population eigenvalues. 
\begin{lem}
	\label{lem:ApproxIntegrationEmpiricalCn} 

	Let $f : \C \to \C$ be an analytic function. Define 
	\begin{align}
		\widehat{\mathcal{C}}_n &:= \bigcup_{j = 1}^{k_n} \widehat{\mathcal{B}}_j 
	\end{align}
	Then we have 
	\begin{equation}
		\frac{1}{2\pi\im} \int_{\widehat{\mathcal{C}}_n} f(z)(z - \Gamma_n)^{-1} \diff{z} 
		= \ind_{\mathcal{A}_n} \frac{1}{2\pi\im} \int_{\mathcal{C}_n} f(z) (z - \Gamma_n)^{-1} \diff{z} + r_n 
		\label{eq:ApproxIntegrationEmpiricalCn} 
	\end{equation}
	where $r_n$ is a random operator with 
	\begin{equation}
		\norm{ r_n } = \BigOhPee\left( \frac{k_n^2 \log k_n}{\sqrt{n}} \right) 
	\end{equation}
\end{lem}

\begin{proof}
	On the event $\mathcal{A}_n$, and by definition of the operator valued contour integral, it is immediate that 
	\begin{equation*}
		\ind_{\mathcal{A}_n} \frac{1}{2\pi\im} \int_{\widehat{\mathcal{C}}_n} f(z)(z - \Gamma_n)^{-1} \diff{z} 
		= \ind_{\mathcal{A}_n} \frac{1}{2\pi\im} \int_{\mathcal{C}_n} f(z) (z - \Gamma_n)^{-1} \diff{z} 
	\end{equation*}
	where in particular, the domain of integration simplifies from the random domain $\widehat{\mathcal{C}}_n$ to the deterministic domain $\mathcal{C}_n$.
	\footnote{
		Actually from \cite[Proposition VII.4.6]{conway1994course}, we clearly don't need the strong condition that $f$ is analytic over all of $\C$. But this strong condition is easier to state and suffices for our paper. 
	}
	\begin{noobs} 
		This is actually a somewhat delicate argument and \cite[Lemma 5]{cardot2007clt} brushes through the finer argument details. First we note the complex integration surrounding Riesz functional calculus does \emph{not} depend on the choice of curves. 

		\begin{thm*}[\cite{conway1994course}, Proposition VII.4.6] 
			Let $\mathscr{A}$ be a Banach algebra with identity, and let $a \in \mathscr{A}$, and let $G$ be an open subset of $\C$ such that $\sigma(a) \subseteq G$. If $\Gamma = \{ \gamma_1, \ldots, \gamma_m \}$ and $\Lambda = \{ \lambda_1, \ldots, \lambda_k \}$ are two positively orientated collection of curves in $G$ such that $\sigma(a) \subseteq \mathrm{ins}\Gamma \subseteq G$ and $\sigma(a) \subseteq \mathrm{ins}\Lambda \subseteq G$ and if $f : G \to \C$ is analytic, then 
			\begin{equation*}
				\int_\Gamma f(z)(z - a)^{-1} \diff{z} = \int_\Lambda f(z)(z - a)^{-1} \diff{z} 
			\end{equation*}
		\end{thm*}
		In other words, as long as the two sets of curves surround the same spectrum, it doesn't matter which collection of curves to integrate over. 

		Let's come back to our case. Naturally, we have that $\hat{\lambda}_j$ is on the trace $\widehat{\mathcal{B}}_j$. But on the event $\mathcal{A}_n$, we note that $\hat{\lambda}_j$ is actually inside of the circle described by the trace $\mathscr{B}_j$. And using the above result, it means on the event $\mathcal{A}_n$, integrating over $\widehat{\mathcal{B}}_j$ or $\mathcal{B}_j$ are equivalent. 
	\end{noobs} 
	Thus we can write, 
	\begin{align*}
		&\frac{1}{2\pi\im} \int_{\widehat{\mathcal{C}}_n} f(z)(z - \Gamma_n)^{-1} \diff{z} \\ 
		&= \ind_{\mathcal{A}_n} \frac{1}{2\pi\im} \int_{\widehat{\mathcal{C}}_n} f(z)(z - \Gamma_n)^{-1} \diff{z} + \ind_{\mathcal{A}_n^c} \frac{1}{2\pi\im} \int_{\widehat{\mathcal{C}}_n} f(z)(z - \Gamma_n)^{-1} \diff{z} \\ 
		&\equiv \ind_{\mathcal{A}_n} \frac{1}{2\pi\im} \int_{\mathcal{C}_n} f(z)(z - \Gamma_n)^{-1} \diff{z} + r_n 
	\end{align*}

	It remains to show the operator $r_n$ converges to zero in probability at some appropriate rate. Fix any $\epsilon \in (0,1)$. Then we have
	\begin{equation*}
		\Prob( \norm{r_n} > \epsilon ) \le \Prob( \ind_{\mathcal{A}_n^c} > \epsilon) = \Prob(\mathcal{A}_n^c) 
	\end{equation*}
	\begin{noobs}
		For convenience, define 
		\begin{equation*}
			b_n := \frac{1}{2\pi\im} \int_{\widehat{\mathcal{C}}_n} f(z)(z - \Gamma_n)^{-1} \diff{z} 		
		\end{equation*}
		so $r_n = \ind_{\mathcal{A}_n^c} b_n$. In particular, $\norm{r_n} = \ind_{\mathcal{A}_n} \norm{b_n}$. Note that since $\epsilon \in (0,1)$ and the indicator can only take on values $\{ 0, 1 \}$, the statement $\{ \ind_B < \epsilon \}$ being true for any event $B$ must be equivalent to the statement $\{ \ind_B = 0 \}$. 
		\begin{align*}
			\Prob( \norm{r_n} > \epsilon ) 
			&= \Prob( \norm{r_n} > \epsilon \,|\, \ind_{\mathcal{A}_n^c} > \epsilon) \Prob( \ind_{\mathcal{A}_n^c} > \epsilon) \\ 
			&\quad + \Prob(  \norm{r_n} > \epsilon \,|\, \ind_{\mathcal{A}_n^c} \le \epsilon) \Prob( \ind_{\mathcal{A}_n^c} \le \epsilon) \\ 
			&= \Prob( \sqrt{n} \ind_{\mathcal{A}_n^c} \norm{b_n} > \epsilon \,|\, \ind_{\mathcal{A}_n^c} > \epsilon) \Prob( \ind_{\mathcal{A}_n^c} > \epsilon) \\ 
			&\quad + \Prob( \underbrace{\ind_{\mathcal{A}_n^c}}_{= 0} \norm{b_n} > \epsilon \,|\, \ind_{\mathcal{A}_n^c} \le \epsilon) \Prob( \ind_{\mathcal{A}_n^c} \le \epsilon) \\ 
			&= \overbrace{ \Prob( \ind_{\mathcal{A}_n^c} \norm{b_n} > \epsilon \,|\, \ind_{\mathcal{A}_n^c} > \epsilon) }^{\le 1} \Prob( \ind_{\mathcal{A}_n^c} > \epsilon) \\ 
			&\quad + \underbrace{\Prob( 1 \cdot 0 > \epsilon \,|\, \ind_{\mathcal{A}_n^c} \le \epsilon)}_{= 0} \Prob( \ind_{\mathcal{A}_n^c} \le \epsilon) \\ 
			&\le \Prob(\ind_{\mathcal{A}_n^c} > \epsilon) \\ 
			&= \Prob( \overbrace{ \underbrace{\ind_{\mathcal{A}_n^c}}_{ = 1} > \epsilon }^{1 > \epsilon} \,|\, \mathcal{A}_n^c ) \Prob( \mathcal{A}_n^c)  
			+ \Prob( \overbrace{\underbrace{ \ind_{\mathcal{A}_n^c} }_{= 0} > \epsilon}^{ 0 > \epsilon }  \,|\, \mathcal{A}_n ) \Prob( \mathcal{A}_n) \\ 
			&= 1 \cdot \Prob(\mathcal{A}_n^c) + 0 \cdot \Prob(\mathcal{A}_n) 
		\end{align*}
	\end{noobs}
	At this point, the rest of the proof follows exactly as in \cite[Lemma 5]{cardot2007clt}, who show 
	\begin{equation*}
		\Prob(\mathcal{A}_n^c) 
		\le \frac{C}{\sqrt{n}} k_n^2 \log k_n. 
	\end{equation*}
	This completes the proof. 
\end{proof}

\begin{rem}
	\label{rem:EmpiricalCovResolvent} 
	For completeness, we should check that even on the event $\mathcal{A}_n$ and any $j = 1, \ldots, k_n$, the integral $\int_{\mathcal{C}_n} f(z) (z - \Gamma_n)^{-1} \diff{z}$ is well defined for all $z \in \mathcal{B}_j$. This is particularly since the resolvent $(\cdot - \Gamma_n)^{-1}$ has singularities exactly at the eigenvalues of $\Gamma_n$. Of course, if we are integrating over the random empirical contours $\widehat{\mathcal{B}}_j$ the resolvent $(\cdot - \Gamma_n)^{-1}$ is well defined by definition. The finite rank operator $\Gamma_n$ has spectrum $\sigma(\Gamma_n) = \{0, \hat{\lambda}_1, \ldots, \hat{\lambda}_n \}$ for which we had assumed $\hat{\lambda}_1 > \hat{\lambda}_2 > \cdots > \hat{\lambda}_n > 0$. So immediately by definition of the event $\mathcal{A}_n$, the point $z$ is not equal to any one of $\hat{\lambda}_1, \ldots, \hat{\lambda}_{k_n} $. However, we still need to check such $z$ is not equal to any one of $ \hat{\lambda}_{k_n + 1}, \ldots, \hat{\lambda}_{n} $. Because of the strictly decreasing ordering of the $\hat{\lambda}_j$'s, it suffices to check that $z$ does not equal to $\hat{\lambda}_{k_n + 1}$. 

	For contradiction, suppose there exists some $z \in \mathcal{B}_{k_n}$ with $z = \hat{\lambda}_{k_n + 1}$. Then $z = \hat{\lambda}_{k_n + 1} = \lambda_{k_n} + \delta_{k_n} / 2 = \lambda_{k_n} + (\lambda_{k_n} - \lambda_{k_n + 1}) / 2 = 3 \lambda_{k_n} / 2 - \lambda_{k_n + 1} / 2$. But on the event $\mathcal{A}_n$, we have $\abs{ \hat{\lambda}_{k_n} - \lambda_{k_n} } < \delta_{k_n} / 4$. This implies $\delta_{k_n} / 4 > \abs{ 3 \lambda_{k_n} / 2 - \lambda_{k_n + 1} / 2 - \lambda_{k_n} } = \delta_{k_n} / 2$, which is a contradiction. 

	In all, this implies on the event $\mathcal{A}_n$ the resolvent $(\cdot - \Gamma_n)^{-1}$ is well defined on $\mathcal{B}_j$ for all $j = 1, \ldots, k_n$. 
\end{rem}

The primary uses of Lemma~\ref{lem:ApproxIntegrationEmpiricalCn} are with the case $f \equiv 1$ and setting $f$ as $f_n$. With only a little more work via the Borel-Cantelli lemma, \cite[Proposition 13]{crambes2013asymptotics} shows $\Prob( \limsup \mathcal{A}_n^c ) = 0$ if $(k_n \log k_n)^2 / n \to 0$. But for our purposes we want to keep track of the various rates of convergences and thus we do not invoke this result. Note that a look into the proof shows the result holds regardless of whether $f$ is dependent on $n$. Indeed, Lemma~\ref{lem:ApproxIntegrationEmpiricalCn} motivates the definitions 
\begin{equation}
	\begin{gathered}
		\widehat{\Pi}_{k_n} 
		:= \ind_{\mathcal{A}_n} \frac{1}{2\pi\im} \int_{\widehat{\mathcal{C}}_n} (z - \Gamma_n)^{-1} \diff{z} 
		\equiv \ind_{\mathcal{A}_n} \frac{1}{2\pi\im} \int_{\mathcal{C}_n} (z - \Gamma_n)^{-1} \diff{z}  \\ 
		\Gamma_n^\dagger 
		:= \ind_{\mathcal{A}_n} \frac{1}{2\pi\im} \int_{\widehat{\mathcal{C}}_n} f_n(z) (z - \Gamma_n)^{-1} \diff{z} 
		\equiv \ind_{\mathcal{A}_n} \frac{1}{2\pi\im} \int_{\mathcal{C}_n} f_n(z) (z - \Gamma_n)^{-1} \diff{z} 
	\end{gathered}
	\label{eq:EmpiricalProjections}
\end{equation}

%

Let's observe a simple bound on the roughened standard deviation $t_n(h)$ that we will repeatedly use.  
\begin{lem}
	\label{lem:SupNormalizedVec} 
	\begin{enumerate}[(i)] 
		\item For any $h \in \OptDomain$, 
			\begin{equation*}
				t_n(h) \ge f_n(\lambda_1) \lambda_{k_n}^{1/2} \norm{h} + a_n 
			\end{equation*}

		\item Provided Assumption~\ref{as:ConditionR} holds, then for $n$ sufficiently large, 
			\begin{equation*}
				\sup_{h \in \OptDomain} \normBIG{ \frac{h}{t_n(h)} } \lesssim \BigOh( \sqrt{k_n \log k_n} ) 
			\end{equation*}
	\end{enumerate}
\end{lem}

\begin{proof}
	\underline{(i):} For any $h \in \OptDomain$, we can write $h = \sum_{j = 1}^{k_n} b_j e_j$ for some $b_j \in \R$ such that $\norm{h}^2 = \sum_{j = 1}^{k_n} b_j^2 \le 1$. 
\begin{align*}
	t_n(h)
	= \sqrt{ \norm{ \Gamma^{1/2} \Gamma^\dagger h }^2 } + a_n  
	&\ge \sqrt{ \sum_{ j = 1}^{k_n} b_j^2 f_n(\lambda_j)^2 \lambda_j } + a_n  \\ 
	&\ge \sqrt{ f_n(\lambda_1)^2 \lambda_{k_n} \sum_{j = 1}^{k_n} b_j^2 } + a_n \\ 
	&= f_n(\lambda_1) \lambda_{k_n}^{1/2} \norm{h} + a_n 
\end{align*}

\underline{(ii):} It is clear the supremum is not achieved at $h = 0$ for any $n$. Thus applying the calculations of part (i) for any $h \in \OptDomain \setminus \{0\}$,  
\begin{equation*}
	\normBIG{ \frac{h}{t_n(h)} }  
	\le \frac{\norm{h}}{f_n(\lambda_1) \lambda_{k_n}^{1/2} \norm{h} + a_n} 
	\le \frac{1}{f_n(\lambda_1) \lambda_{k_n}^{1/2}  + a_n} 
\end{equation*}
Since $f_n(\lambda_1) \lambda_{k_n}^{1/2} = \left( \frac{1}{\lambda_1} \BigOh\left( \frac{1}{\sqrt{n}} \right) + 1 \right) \BigOh\left( \sqrt{ \frac{1}{k_n \log k_n} } \right)  = \BigOh\left( \frac{1}{\sqrt{k_n \log k_n}} \right)$, we have $f_n(\lambda_1) \lambda_{k_n}^{1/2} + a_n = \BigOh\left( \max\left\{ \frac{1}{ \sqrt{k_n \log k_n} } \,,\, a_n \right\} \right) = \BigOh\left( \frac{1}{\sqrt{k_n \log k_n}} \right)$ where the last equality follows from Assumption~\ref{as:ConditionR}.

\end{proof}

As outlined in the proof outline of this paper's main result Theorem~\ref{thm:NormalizedLeungTamStatisticConvergence}, there are two distinct steps to proving the result. 

\subsection{Step I} 
The $\mathcal{T}_n$ term is directly handled by \cite[Lemma 6]{cardot2007clt}; we record the result here for completeness. 
\begin{prop}
	\label{prop:TermTnGoesToZero} 
	If \eqref{as:ConditionH} holds, 
	\begin{equation*}
		\norm{\mathcal{T}_n} = o_\Prob\left( \frac{1}{\sqrt{n}} \right) 
	\end{equation*} 
\end{prop}

The next result is critical in the proofs of Propositions~\ref{prop:TermSnGoesToZero} and \ref{prop:TermYnGoesToZero}. 
\begin{lem} 
	\label{lem:ResolventGammaEmpiricalGammaApprox} 
	For any sufficiently large $j$ and $n$. 
	\begin{equation*}
		\E\left[ \norm{ (z - \Gamma)^{-1} (\Gamma_n - \Gamma)}^2 \right] 
		\lesssim \frac{j^3 \log j}{n}, 
		\quad \text{for all $z \in \mathcal{B}_j$} 
	\end{equation*}
\end{lem} 

\begin{proof} 
	Since $\Gamma, \Gamma_n \in \CompactOp \subseteq \BddOp$ then it follows that the resolvent $R(z ; \Gamma) := (z - \Gamma)^{-1}$ is also in $\BddOp$, and thus $(\Gamma - \Gamma_n)(z - \Gamma)^{-1} \in \BddOp$. Hence we can bound the operator norm $\norm{\cdot}$ by the Hilbert-Schmidt norm $\norm{\cdot}_{\mathrm{HS}}$ 
	\footnote{
		See \cite[Exercise IX.2.19]{conway1994course} 
	}
	\begin{align}
		\norm{(\Gamma - \Gamma_n)(z - \Gamma)^{-1}}^2 
		&\le \norm{(\Gamma - \Gamma_n)(z - \Gamma)^{-1}}^2_{\mathrm{HS}} \nonumber \\ 
		&\equiv \sum_{l = 1}^\infty \norm{ (\Gamma - \Gamma_n)(z - \Gamma)^{-1}(e_l) }^2 \nonumber \\ 
		&= \sum_{l, k = 1}^\infty \frac{1}{(z - \lambda_l)^2} \absBIG{ \inner{(\Gamma - \Gamma_n)e_l}{e_k} }^2  
		\label{pfeq:BSEChicken} 
	\end{align}
	\begin{noobs}
		To see \eqref{pfeq:BSEChicken}: 
	\begin{align*}
		&\norm{(\Gamma - \Gamma_n)(z - \Gamma)^{-1}}^2 \\ 
		&\le \norm{(\Gamma - \Gamma_n)(z - \Gamma)^{-1}}^2_{\mathrm{HS}} \\ 
		&\equiv \sum_{l = 1}^\infty \norm{ (\Gamma - \Gamma_n)(z - \Gamma)^{-1}(e_j) }^2 && \text{defn of HS norm} \\ 
		&= \sum_{l = 1}^\infty \sum_{k = 1}^\infty  \absBIG{ \inner{(\Gamma - \Gamma_n)(z - \Gamma)^{-1}(e_l)}{e_k} }^2 && \text{Parceval's identity} \\ 
		&= \sum_{l = 1}^\infty \sum_{k = 1}^\infty  \absBIG{ \inner{(\Gamma - \Gamma_n) \frac{1}{z - \lambda_l}e_l}{e_k} }^2 && \text{eigenvalues of resolvent} 
	\end{align*}
	\begin{itemize}
		\item The first line is a result from \cite[Exercise IX.2.19]{conway1994course} that basically says the HS norm is more precise than the operator norm 

		\item Fourth line uses that the eigenvalues of the resolvent $R(z ; T)$ of an operator $T$ is exactly $(z - \lambda_j)$'s and with the same eigenvectors as $T$. See \cite{kato1995perturbation}. 
	\end{itemize}
	\end{noobs} 

	Observe that for $z \in \mathcal{B}_j$ and $l \neq j$, by the triangle inequality we have that 
	\begin{equation}
		\abs{z - \lambda_l} \ge \frac{\abs{\lambda_l - \lambda_j}}{2}. 
		\label{pfeq:BSEKanagroo} 
	\end{equation}
	\begin{noobs}
		\begin{align*}
			\abs{z - \lambda_l} 
			&= \abs{ (\lambda_l - \lambda_j) - (z - \lambda_j) } \\
			&\ge \abs{ \abs{\lambda_l - \lambda_j} - \abs{z - \lambda_j}  } && \text{reverse triangle inequality} \\ 
			&\ge \abs{\lambda_l - \lambda_j} - \abs{z - \lambda_j} \\ 
			&= \abs{\lambda_l - \lambda_j} - \frac{\delta_j}{2} \\ 
			&\ge \abs{\lambda_l - \lambda_j} / 2 
		\end{align*} 
		where we have used that any $z \in \mathcal{B}_j$ must be of the form $z = \lambda_j + \frac{\delta_j}{2} e^{\im\theta}$ for some $\theta \in [0, 2\pi)$. Note in the special case that $l = j$, we immediately have that $\abs{z - \lambda_j} = \frac{\delta_j}{2}$ (i.e.\ it's just the radius of the circle).
	\end{noobs}

	In addition, by the KL expansion, we have for all $l, k = 1, 2, \ldots$ 
	\begin{equation}
		\E[ \absBIG{\inner{(\Gamma_n - \Gamma) e_l }{e_k }}^2 ] 
		\le \frac{1}{n} \E\left[ \inner{X_1}{e_l}^2 \inner{X_1}{e_k}^2 \right] 
		\le \frac{M}{n} \lambda_l \lambda_k 
		\label{pfeq:BSEKoala} 
	\end{equation}
	\begin{noobs}
		\begin{align*}
		\E[ \absBIG{\inner{(\Gamma_n - \Gamma) e_l }{e_k }}^2 ] 
		&= \E\left[  \absBIG{  \inner{ \frac{1}{n} \sum_{i = 1}^n (X_i \otimes X_i - \Gamma) e_l }{e_k}  }^2 \right] \\ 
		&= \frac{1}{n^2} \sum_{i, j = 1}^n \E\left[ \inner{ (X_i \otimes X_i - \Gamma)e_l }{ e_k } \inner{ (X_j \otimes X_j - \Gamma)e_l }{ e_k } \right] \\ 
		&= \frac{1}{n^2} \sum_{i = 1}^n \E[ \inner{ (X_i \otimes X_i - \Gamma)e_l}{e_k}^2 ] \\ 
		&\quad + \frac{1}{n^2} \sum_{i \neq j}^n 
		\E\left[  \inner{ (X_i \otimes X_i - \Gamma)e_l }{e_k }  \inner{ (X_j \otimes X_j - \Gamma)e_l }{e_k } \right]  \\ 
		&= \frac{1}{n^2} \sum_{i = 1}^n \E\left[ \inner{ (X_i \otimes X_i - \Gamma)e_l }{e_k }  \right] \\
		&= \frac{1}{n^2} \cdot n \E\left[ \inner{ (X_1 \otimes X_1 - \Gamma)e_l }{e_k } \right] \\ 
		&= \frac{1}{n} \E\left[  \inner{X_1 \otimes X_1(e_l)}{e_k}^2 - 2 \inner{ X_1 \otimes X_1(e_l) }{ e_k } \inner{\Gamma e_l}{e_k} + \inner{\Gamma e_l}{e_k}^2 \right] \\ 
		&= \frac{1}{n} \E\left[ \inner{ \inner{X_1}{e_l} X_1 }{ e_k }^2 \right] - 2 \inner{\Gamma e_l}{e_k}^2 + \inner{\Gamma e_l}{e_k}^2 \\ 
		&= \frac{1}{n} \E\left[ \inner{  X_1 }{ e_l }^2 \inner{X_1}{e_k}^2 \right] - \inner{\Gamma e_l}{e_k}^2 \\ 
		&\le \frac{1}{n} \E\left[ \inner{X_1}{e_l}^2 \inner{X_1}{e_k}^2 \right] 
		\end{align*} 
		where we have used that $X_i$'s are iid, $X_i \otimes X_i - \Gamma$ is mean zero by the definition of $\Gamma$, and the definition of the tensor operator. 

		Now we use the KL expansion \eqref{eq:KLExpansion} to have, 
		\begin{align*} 
			\inner{X_1}{e_l} \inner{X_1}{e_k} 
			&= \inner{ \sum_{j = 1}^\infty \sqrt{\lambda_j} \xi_j e_j }{e_l} \inner{ \sum_{j' = 1}^\infty \sqrt{\lambda_{j'}} \xi_{j'} e_{j'} }{ e_k } \\ 
			&= \sum_{j, j' = 1}^\infty \sqrt{\lambda_j} \sqrt{\lambda_{j'}} \xi_j \xi_{j'} \inner{e_j}{e_l} \inner{e_{j'}}{e_k} \\ 
			&= \sqrt{\lambda_l \lambda_k} \xi_l \xi_k 
		\end{align*} 
		which then implies,
		\begin{equation*}
			\E[ \absBIG{\inner{(\Gamma_n - \Gamma) e_l }{e_k }}^2 ] 
			\le \frac{1}{n} (\sqrt{\lambda_l \lambda_k})^2 \E[ \xi_l^2 \xi_k^2 ] 
			\le \frac{M}{n} \lambda_l \lambda_k 
		\end{equation*}
		where we applied the Cauchy-Schwartz inequality, and recall Condition~\ref{as:ConditionA}. 
	\end{noobs} 

	Thus, applying \eqref{pfeq:BSEKoala} and \eqref{pfeq:BSEKanagroo} into \eqref{pfeq:BSEChicken} 
	\begin{equation}
		\E\left[ \norm{(\Gamma - \Gamma_n)(z - \Gamma)^{-1}}^2 \right] 
		= 4M \frac{\lambda_j}{\delta_j^2} \frac{1}{n} \sum_{k = 1}^\infty \lambda_k + 4M \frac{1}{n} \sum_{l \neq j}^\infty \frac{\lambda_l}{(\lambda_l - \lambda_j)^2} \sum_{k = 1}^\infty \lambda_k,
		\label{pfeq:BSEKiwi} 
	\end{equation}
	for all $z \in \mathcal{B}_j$. 
	\begin{noobs}
	\begin{align*}
		\E\left[ \sup_{z \in \mathcal{B}_j} \norm{(\Gamma - \Gamma_n)(z - \Gamma)^{-1}}^2 \right] 
		&\le \E\left[ \sup_{z \in \mathcal{B}_j} \sum_{l = 1}^\infty \sum_{k = 1}^\infty \frac{1}{(\lambda_l - z)^2} \inner{ (\Gamma_n - \Gamma)e_l }{ e_k }^2 \right] \\
		&= \E\left[ \sup_{z \in \mathcal{B}_j} \left( \sum_{l = j} + \sum_{l \neq j} \right) \sum_{k = 1}^\infty \frac{1}{(\lambda_l - z)^2} \inner{ (\Gamma_n - \Gamma)e_l }{ e_k }^2  \right] \\ 
		&= \E\Bigg[ \sup_{z \in \mathcal{B}_j}
			\Bigg(
				\sum_{k = 1}^\infty \overbrace{\frac{1}{(\lambda_j - z)^2}}^{= 1 / (\delta_j / 2)^2} \overbrace{\inner{ (\Gamma_n - \Gamma)e_j }{ e_k }^2}^{\le M \lambda_j \lambda_k / n }  \\ 
		&\quad +  \sum_{l \neq j}^\infty \sum_{k = 1}^\infty \underbrace{\frac{1}{(\lambda_l - z)^2}}_{ \le \left( 2 / \abs{\lambda_l - \lambda_j} \right)^2 }  \underbrace{\inner{ (\Gamma_n - \Gamma)e_l }{ e_k }^2}_{ \le M \lambda_l \lambda_k / n } 
		\Bigg) 
		\Bigg] \\ 
		&\le \sum_{k = 1}^\infty \frac{4}{\delta_j^2} \frac{M}{n} \lambda_j \lambda_k + \sum_{l \neq j}^\infty \sum_{k = 1}^\infty \frac{4}{( \lambda_l - \lambda_j )^2} \frac{M}{n} \lambda_l \lambda_k \\ 
		&= 4M \frac{\lambda_j}{\delta_j^2} \frac{1}{n} \sum_{k = 1}^\infty \lambda_k + 4M \frac{1}{n} \sum_{l \neq j}^\infty \frac{\lambda_l}{(\lambda_l - \lambda_j)^2} \sum_{k = 1}^\infty \lambda_k 
	\end{align*}
	\end{noobs}

	At this point we need to investigate the behavior of $\sum_{l \neq j} \frac{\lambda_l}{(\lambda_l - \lambda_j)^2}$.
	\footnote{
		Observe that this is a different term compared to that of \cite[Lemma 2]{cardot2007clt} 
	}
	Let's decompose,  
	\begin{equation}
		\sum_{l \neq j} \frac{\lambda_l}{(\lambda_l - \lambda_j)^2} 
		= \left(  \sum_{l = 1}^{j - 1} + \sum_{l = j + 1}^{2j} + \sum_{l = 2j + 1}^\infty \right) \frac{\lambda_l}{(\lambda_l - \lambda_j)^2} 
		=: T_1 + T_2 + T_3 
		\label{pfeq:BSESongbird} 
	\end{equation}

	By \cite[Lemma 1]{cardot2007clt} where we have $\lambda_l - \lambda_j \ge (1 - l / j ) \lambda_l$, and recalling that the eigenvalues are strictly decreasing, 
	\begin{equation}
		T_1 
		=  \sum_{l = 1}^{j - 1} \frac{\lambda_l}{(\lambda_l - \lambda_j)^2}  
		\le \frac{1}{\lambda_{j - 1}} \sum_{ l =1}^{j - 1} \frac{1}{(1 - l / j)^2}  
		= \frac{j^2}{\lambda_{j - 1}} \frac{1}{6} (\pi^2 - 6 \psi^{(1)}(j)) 
		\label{pfeq:BSESongbirdT1} 
	\end{equation}
	where $\psi^{(m)}$ is the polygamma function of order $m$. 
	\footnote{ 
		The \emph{polygamma function of order $m$} is defined to be the $(m + 1)$th derivative of the logarithm of the gamma function; or equivalently, it is the $m$th derivative of the digamma function. 
	} 
	\begin{noobs}
	\begin{align*}
		T_1 
		&=  \sum_{l = 1}^{j - 1} \frac{\lambda_l}{(\lambda_l - \lambda_j)^2} \\ 
		&\le \sum_{l = 1}^{j - 1} \frac{\lambda_l}{(1 - l / j)^2} \frac{1}{\lambda_l^2} \\ 
		&\le \frac{1}{\lambda_{j-1}} \sum_{ l =1}^{j - 1} \frac{1}{(1 - l / j)^2} \\ 
		&\le \frac{1}{\lambda_{j - 1}} \sum_{ l =1}^{j - 1} \frac{1}{(1 - l / j)^2} \\ 
		&= \frac{j^2}{\lambda_{j - 1}} \frac{1}{6} (\pi^2 - 6 \psi^{(1)}(j)) 
	\end{align*}
	The last equation follows from a WolframAlpha calculation... 
	\end{noobs} 
	Similarly, we have 
	\begin{equation}
		T_2 \le \frac{\lambda_{j + 1}}{\lambda_j^2} \frac{j^2}{6} (\pi^2 - 6 \psi^{(1)} (j + 1) ) 
		\label{pfeq:BSESongbirdT2} 
	\end{equation}
	\begin{noobs}
	\begin{align*}
		T_2 
		&= \sum_{l = j + 1}^{2j} \frac{\lambda_l}{(\lambda_l - \lambda_j)^2} \\ 
		&\le \sum_{l = j + 1}^{2j} \frac{\lambda_l}{(1 - j / l)^2} \frac{1}{\lambda_j^2} \\ 
		&= \frac{1}{\lambda_j^2} \sum_{l = j + 1}^{2j} \frac{\lambda_l}{( 1 - j / l)^2} \\ 
		&\le \frac{\lambda_{j + 1}}{\lambda_j^2} \sum_{l = j + 1}^{2j} \frac{1}{(1 - j / l)^2} \\ 
		&= \frac{\lambda_{j + 1}}{\lambda_j^2} \frac{j^2}{6} (\pi^2 - 6 \psi^{(1)} (j + 1) ) 
	\end{align*}
	\end{noobs}
	And since for $l \ge 2j + 1$ we have $\lambda_j - \lambda_l \ge \lambda_j - \lambda_{2j + 1} > \lambda_j - \lambda_{2j} \ge (1 - j / (2j) ) \lambda_j = 2 \lambda_j$, this implies, 
	\begin{align}
		T_3  
		< \frac{1}{4 \lambda_j^2} \sum_{l = 2j + 1}^\infty \lambda_l  
		\le \frac{1}{4 \lambda_j^2} ( (2j + 1) + 1 ) \lambda_{2j + 1} 
		< \frac{1}{4\lambda_j^2} 2(j + 1) \lambda_j 
		= \frac{j + 1}{2 \lambda_j} 
		\label{pfeq:BSESongbirdT3} 
	\end{align}
	where the second inequality follows from \cite[Lemma 1]{cardot2007clt}. 

	Now we use the following well-known bounds of the polygamma function: for $m \ge 1$ and $x > 0$, 
	\begin{equation}
		\frac{(m - 1)!}{x^m} + \frac{m!}{2x^{m + 1}} \le (-1)^{m + 1} \psi^{(m)}(x) \le \frac{(m - 1)!}{x^m} + \frac{m!}{x^{m+1}} 
		\label{pfeq:BSEPolygammaBounds} 
	\end{equation}
	\begin{noobs}
		In our situation, applying \eqref{pfeq:BSEPolygammaBounds} with $m = 1$, and to the case $x = j$
		\begin{equation*}
			\frac{1}{j} + \frac{1}{2j^2} \le \psi^{(1)}(j) \le \frac{1}{j} + \frac{1}{j^2} 
		\end{equation*}
		and to the case $x = j + 1$ 
		\begin{equation*}
			\frac{1}{j + 1} + \frac{1}{2(j + 1)^2} \le \psi^{(1)}(j + 1) \le \frac{1}{j + 1} + \frac{1}{(j + 1)^2} 
		\end{equation*}

		Thus from \eqref{pfeq:BSESongbirdT1}, we have 
		\begin{equation*}
			T_1 
			\le C_1 \frac{j^2}{\lambda_j} \left(1 + \frac{1}{j} + \frac{1}{j^2} \right) 
			= C_1 \frac{1}{\lambda_j} \left(j^2 + j + 1 \right) 
			= \mathcal{O}\left(  \frac{j^2}{\lambda_j} \right)
		\end{equation*}
		We have analogous calculations for $T_2$ and $T_3$. 
	\end{noobs}
	and applying \eqref{pfeq:BSEPolygammaBounds} to \eqref{pfeq:BSESongbirdT1}-\eqref{pfeq:BSESongbirdT3} we obtain the bounds, 
	\begin{align}
		T_1 &\le C_1 \frac{j^2}{\lambda_j}, 
		    & 
		T_2 &\le C_2 \frac{j^2}{\lambda_j}, 
		    &
		T_3 &\le C_3 \frac{j + 1}{\lambda_j} 
		\label{pfeq:BSESongbirdSummaryTBounds} 
	\end{align}
	Putting \eqref{pfeq:BSESongbirdSummaryTBounds} back into \eqref{pfeq:BSESongbird}, we arrive at, 
	\begin{equation}
		\sum_{l \neq j} \frac{\lambda_l}{(\lambda_l - \lambda_j)^2} \le C \frac{j^2}{\lambda_j} 
		\label{pfeq:BSESongEagle}
	\end{equation}

	Now putting \eqref{pfeq:BSESongEagle} into \eqref{pfeq:BSEKiwi}, we have 
	\begin{equation}
		\E\left[ \norm{(\Gamma - \Gamma_n)(z - \Gamma)^{-1}}^2 \right] 
		\le C \frac{1}{n} \max\left\{ \frac{\lambda_j}{\delta_j^2} , \frac{j^2}{\lambda_j} \right\}, 
		\label{pfeq:BSEMomentBound} 
	\end{equation}
	for all $z \in \mathcal{B}_j$.

	Applying Condition~\ref{as:ConditionA} shows that for sufficiently large $j$, 
	\begin{equation*}
		\max\left\{ \frac{\lambda_j}{\delta_j^2} , \frac{j^2}{\lambda_j} \right\} 
		\le C \max\left\{ j \log j,  j^3 \log j \right\} 
		= C j^3 \log j 
	\end{equation*}
	\begin{noobs} 
	\begin{equation*}
		\frac{\lambda_j}{\delta_j^2} 
		\le C \frac{1}{j \log j} (j \log j)^2 
		= C j \log j 
	\end{equation*}
	and likewise, 
	\begin{equation*}
		\frac{j^2}{\lambda_j} 
		\le C j^2 j \log j 
		= C j^3 \log j 
	\end{equation*}
	\end{noobs} 
	This completes the proof.

\end{proof} 

\begin{prop}
	\label{prop:TermYnGoesToZero} 
	For sufficiently large $n$, 
	\begin{equation*}
		\sup_{h \in \OptDomain} \absBIG{ \inner{\mathcal{Y}_n}{ \frac{h}{t_n(h)} } }  
		\lesssim \BigOhPee\left( \frac{k_n^{9/2} (\log k_n)^{3/2}}{\sqrt{n}}  \right) 
	\end{equation*}
\end{prop}

\begin{proof} 
	Firstly using Lemma~\ref{lem:ApproxIntegrationEmpiricalCn}, we have 
	\begin{noobs} 
		Given operators $A, B$ and $z \in \rho(A) \cap \rho(B)$ and where $\rho(A)$ is the resolvent set of $A$, we have the \emph{second resolvent identity}, 
		\begin{equation*}
			R(z ; A) - R(z ; B) = R(z ; A) (B - A) R(z ; B) 
		\end{equation*}
	\end{noobs} 
	\begin{align*}
		&\widehat{\Pi}_{k_n} - \Pi_{k_n} \nonumber \\
		&\equiv \frac{1}{2\pi\im}\sum_{j = 1}^{k_n} \left[ \int_{\widehat{\mathcal{B}}_j} (z - \Gamma_n)^{-1} \diff{z} - \int_{\mathcal{B}_j} (z - \Gamma)^{-1} \diff{z} \right] \nonumber \\ 
		&= \frac{1}{2\pi\im}\sum_{j = 1}^{k_n} \left[ \ind_{\mathcal{A}_n} \int_{\mathcal{B}_j} (z - \Gamma_n)^{-1} \diff{z} - ( \ind_{\mathcal{A}_n} + \ind_{\mathcal{A}_n^c} ) \int_{\mathcal{B}_j} (z - \Gamma)^{-1} \diff{z} \right] + r_n \nonumber \\ 
		&= \ind_{\mathcal{A}_n} \frac{1}{2\pi\im}\sum_{j = 1}^{k_n}  \int_{\mathcal{B}_j} \left[ (z - \Gamma_n)^{-1}  - (z - \Gamma)^{-1} \right] \diff{z}  \nonumber 
		- \ind_{\mathcal{A}_n^c} \frac{1}{2\pi\im} \sum_{j= 1}^{k_n} \int_{\mathcal{B}_j} (z - \Gamma)^{-1} \diff{z}  + r_n \nonumber \\ 
	\end{align*}
	By the resolvent identity, and this is feasible only because we are on the event $\mathcal{A}_n$,  
	\begin{align*}
		&\ind_{\mathcal{A}_n} \frac{1}{2\pi\im} \sum_{j = 1}^{k_n} \int_{\mathcal{B}_j} \left[ (z - \Gamma_n)^{-1} - (z - \Gamma)^{-1} \right] \diff{z} \\ 
		&= \ind_{\mathcal{A}_n} \frac{1}{2\pi\im} \sum_{j = 1}^{k_n} \int_{\mathcal{B}_j} (z - \Gamma_n)^{-1} (\Gamma_n - \Gamma) (z - \Gamma)^{-1} \diff{z} 
	\end{align*}
	Using the resolvent identity again, we can decompose 
	\begin{equation}
		\ind_{\mathcal{A}_n} \frac{1}{2\pi\im} \sum_{j = 1}^{k_n} \int_{\mathcal{B}_j} (z - \Gamma_n)^{-1} (\Gamma_n - \Gamma) (z - \Gamma)^{-1} \diff{z} 
		=: \mathsf{S}_n + \mathsf{R_n} 
		\label{pfeq:YellowDecomposition} 
	\end{equation}
	where we define, 
	\begin{subequations}
		\begin{align} 
			\mathsf{S}_n &:= \ind_{\mathcal{A}_n} \frac{1}{2\pi\im} \sum_{j = 1}^{k_n} \int_{\mathcal{B}_j} (z - \Gamma)^{-1} (\Gamma_n - \Gamma) (z - \Gamma)^{-1} \diff{z} \\ 
			\mathsf{R}_n &:= \ind_{\mathcal{A}_n} \frac{1}{2\pi\im} \sum_{j = 1}^{k_n} \int_{\mathcal{B}_j} (z - \Gamma)^{-1} (\Gamma_n - \Gamma) (z - \Gamma)^{-1} (\Gamma_n - \Gamma) (z - \Gamma_n)^{-1} \diff{z} 
		\end{align} 
		\label{pfeq:YellowDecompositionTermsSnAndRn} 
	\end{subequations}
	Thus, the above equation can be rewritten as, 
	\begin{equation}
		\widehat{\Pi}_{k_n} - \Pi_{k_n} 
		= \mathsf{S}_n + \mathsf{R}_n - \ind_{\mathcal{A}_n^c} \frac{1}{2\pi\im} \sum_{j= 1}^{k_n} \int_{\mathcal{B}_j} (z - \Gamma)^{-1} \diff{z}  + r_n
	\end{equation}
	By the triangle inequality and Cauchy-Schwartz inequality, 
	\begin{align}
		\sup_{h \in \OptDomain} \absBIG{ \inner{\mathcal{Y}_n}{ \frac{h}{t_n(h)} } } 
		&\le  \sup \absBIG{\inner{\mathsf{S}_n\rho}{\frac{h}{t_n(h)}}} 
			+ \sup \absBIG{ \inner{\mathsf{R}_n\rho}{ \frac{h}{t_n(h)}} } \nonumber \\ 
		&\quad  + \ind_{\mathcal{A}_n^c} \frac{1}{2\pi} \sum_{j = 1}^{k_n} \int_{\mathcal{B}_j} \sup \absBIG{ \inner{(z - \Gamma)^{-1}\rho}{ \frac{h}{t_n(h)} } } \diff{z} 
			+ \sup \absBIG{ \inner{r_n \rho}{ \frac{h}{t_n(h)} } }   
		\label{pfeq:YellowYnDecomposition} 
	\end{align}

	We will individually bound the four terms on the right hand side of \eqref{pfeq:YellowYnDecomposition}. Let's first discuss those last two remaining terms. By again Cauchy-Schwartz inequality and Lemma~\ref{lem:ApproxIntegrationEmpiricalCn}, 
	\begin{equation*}
		\sup_{h \in \OptDomain} \absBIG{ \inner{r_n \rho}{ \frac{h}{t_n(h)} }  } 
		\le \norm{\rho} \sup\normBIG{ \frac{h}{t_n(h)} } \norm{r_n} 
	\end{equation*}
	The $\norm{r_n}$ term is bounded by Lemma~\ref{lem:ApproxIntegrationEmpiricalCn}. 

	For the third integral expression, by Cauchy-Schwartz inequality again  
	\begin{align}
		&\ind_{\mathcal{A}_n^c} \frac{1}{2\pi} \sum_{j = 1}^{k_n} \int_{\mathcal{B}_j} \sup_{h \in \OptDomain} \absBIG{ \inner{(z - \Gamma)^{-1}\rho}{ \frac{h}{t_n(h)} } } \diff{z}  \nonumber  \\
		&\le \norm{\rho} \sup\normBIG{ \frac{h}{t_n(h)} }  \ind_{\mathcal{A}_n^c} \frac{1}{2\pi} k_n \max_{j = 1, \ldots, k_n} \sup_{z \in \mathcal{B}_j} \norm{(z - \Gamma)^{-1} } \mathrm{diam}(\mathcal{B}_j) \nonumber \\ 
		&< \norm{\rho} \sup\normBIG{ \frac{h}{t_n(h)} } \frac{1}{\pi} \ind_{\mathcal{A}_n^c} k_n \nonumber \\
		&\lesssim \sup\normBIG{ \frac{h}{t_n(h)} }  \BigOhPee\left( \frac{k_n^2 \log k_n}{\sqrt{n}}  \right) k_n 
		\label{pfeq:YellowRemainderIntegralTerm} 
	\end{align}
	In particular, we used that for any $z \in \mathcal{B}_j$, 
	\begin{equation}
		\norm{(z - \Gamma)^{-1}} 
		\le \frac{1}{\mathrm{dist}(z, \sigma(\Gamma))} 
		= \frac{1}{\delta_j / 2} 
		\label{pfeq:YellowResolventNorm} 
	\end{equation}
	where $\sigma(\Gamma)$ denotes the spectrum of $\Gamma$. By the choice of radii $\delta_j / 2$'s that define the circles $\mathcal{B}_j$'s, any point $z \in \mathcal{B}_j$ is not an eigenvalue of $\Gamma$, which by definition implies $z$ is in the resolvent set of $\Gamma$. Hence the first inequality of \eqref{pfeq:YellowResolventNorm} follows from standard results on the norm of a resolvent (e.g.\ \cite[Proposition VII.3.9]{conway1994course}). The equality $\mathrm{dist}(z, \sigma(\Gamma)) = \delta_j / 2$ in \eqref{pfeq:YellowResolventNorm} follows immediately by again the definition of $\mathcal{B}_j$. 

	In addition $\mathrm{diam}(\mathcal{B}_j) = \delta_j$, and that $\Prob(\mathcal{A}_n^c) \lesssim \frac{k_n^2 \log k_n}{\sqrt{n}}$ from the proof of Lemma~\ref{lem:ApproxIntegrationEmpiricalCn}. Thus by Lemma~\ref{lem:SupNormalizedVec}, the two remainder terms of \eqref{pfeq:YellowYnDecomposition} are of order, 
	\begin{align}
		\sup \normBIG{ \frac{h}{t_n(h)} } \BigOhPee\left( \frac{k_n^3 \log k_n}{\sqrt{n}} + \frac{k_n^2 \log k_n}{\sqrt{n}}  \right)   
		&= \BigOh( \sqrt{k_n \log k_n} ) \BigOhPee\left( \frac{k_n^3 \log k_n}{\sqrt{n}} + \frac{k_n^2 \log k_n}{\sqrt{n}}  \right) \nonumber \\
		&= \BigOhPee\left( \frac{k_n^{7/2} (\log k_n)^{3/2} }{\sqrt{n}}  \right) 
		\label{pfeq:YellowRemainderProbCvgRate} 
	\end{align}
	Now we turn to bounding the $\mathsf{S}_n$ term in \eqref{pfeq:YellowYnDecomposition}. 

	\underline{The $\mathsf{S}_n$ term:}
	\footnote{
		It is worth noting our discussions of this $\mathsf{S}_n$ term is substantially different than that of \cite[Proposition 2]{cardot2007clt}. In particular, while the integral of the $j$th summand of $\mathsf{S}_n$ has an explicit form due to \cite{dauxois1982asymptotic}, but for our purposes of obtaining moment bounds, knowing this explicit form is unnecessary. In contrast to our purposes, the desired computation of the predicted value $\E[ \inner{\mathsf{S}_n \rho}{X_{n + 1}}^2]$ in \cite[Proposition 2]{cardot2007clt} gives them an extra smoothing property which can take advantage of the explicit form of $\mathsf{S}_n$. Indeed, an earlier draft of this paper uses analogous proofs methods of this $\mathsf{S}_n$ term as the authors and we arrive at the same rate in \eqref{pfeq:YellowSnProbBound}. 
	}
	By triangle inequality, 
	\begin{align}
		\norm{\mathsf{S}_n} 
		&\le \frac{1}{2\pi} \sum_{j = 1}^{k_n} \int_{\mathcal{B}_j} \norm{ (z - \Gamma)^{-1} (\Gamma_n - \Gamma) } \,\norm{(z - \Gamma)^{-1}} \diff{z} 
		\label{pfeq:YellowSnNormTriangleIneq} 
	\end{align}
	Firstly, we have the bound $\sup_{z \in \mathcal{B}_j} \norm{(z - \Gamma)^{-1}} < 2 / \delta_j$ again by \eqref{pfeq:YellowResolventNorm}. Thus by Lemma~\ref{lem:ResolventGammaEmpiricalGammaApprox}, we have in all 
	\begin{align*}
		\E[ \norm{\mathcal{S}_n} ] 
		&\lesssim \sum_{j = 1}^{k_n} \sup_{z \in \mathcal{B}_j} \E[ \norm{(z - \Gamma)^{-1} (\Gamma_n - \Gamma)} ] \,\sup_{z \in \mathcal{B}_j} \norm{(z - \Gamma)^{-1}} \,\mathrm{diam}(\mathcal{B}_j) \\
		&\lesssim \sum_{j = 1}^{k_n} \sqrt{ \frac{j^3 \log j}{n} } \cdot \frac{2}{\delta_j} \cdot \delta_j \\ 
		&\lesssim \frac{ k_n^{5/2} (\log k_n)^{1/2} }{\sqrt{n} } 
	\end{align*}

	By Markov's and Jensen's inequality and Lemma~\ref{lem:SupNormalizedVec} 
	\begin{align}
		\sup_{h \in \OptDomain} \absBIG{ \inner{\mathsf{S}_n \rho}{x} } 
		\lesssim \sup_{h \in \OptDomain} \normBIG{ \frac{h}{t_n(h)} }  \BigOhPee \left(  \frac{k_n^{5/2} (\log k_n)^{1/2}}{\sqrt{n}} \right) 
		= \BigOhPee \left(  \frac{k_n^{3} \log k_n}{\sqrt{n}} \right) 
		\label{pfeq:YellowSnProbBound} 
	\end{align} 
	This completes the discussion of the $\mathsf{S}_n$ term. 

	\underline{The $\mathsf{R}_n$ term:} We first define for $j = 1, \ldots, k_n$, 
	\begin{equation}
		\mathsf{T}_{j,n} := \ind_{\mathcal{A}_n} \int_{\mathcal{B}_j} (z - \Gamma)^{-1} (\Gamma_n - \Gamma) (z - \Gamma)^{-1} (\Gamma_n - \Gamma) (z - \Gamma_n)^{-1} \diff{z} 
		\label{pfeq:RoseTjn} 
	\end{equation}
	so that we can write 
	\begin{equation*}
		\mathsf{R}_n = \frac{1}{2\pi\im} \sum_{j = 1}^{k_n} \mathsf{T}_{j,n} 
	\end{equation*}
	By again the Cauchy-Schwartz inequality, we have 
	\begin{align}
		&\sup_{h \in \OptDomain} \absBIG{ \inner{\mathsf{T}_{j,n}\rho}{ \frac{h}{t_n(h)} } } \nonumber \\ 
		&\le \norm{\rho} \sup\normBIG{ \frac{h}{t_n(h)} } \int_{\mathcal{B}_j} \norm{(z - \Gamma)^{-1}(\Gamma_n - \Gamma)(z - \Gamma)^{-1}(\Gamma_n - \Gamma)(z - \Gamma_n)^{-1} } \ind_{\mathcal{A}_n}  \diff{z} \nonumber \\ 
		&\le \norm{\rho} \sup\normBIG{ \frac{h}{t_n(h)} }  \int_{\mathcal{B}_j} \norm{(z - \Gamma)^{-1}(\Gamma_n - \Gamma)}^2 \,\norm{ (z - \Gamma_n)^{-1} } \ind_{\mathcal{A}_n}  \diff{z} 
		\label{pfeq:RoseTjnAbsVal} 
	\end{align}
	\begin{noobs}
		\begin{align*}
		&\abs{ \inner{ \mathsf{T}_{j,n}\rho}{x} } \\ 
		&\le \int_{\mathcal{B}_j} \absBIG{ \inner{ (z - \Gamma)^{-1}(\Gamma_n - \Gamma)(z - \Gamma)^{-1}(\Gamma_n - \Gamma)(z - \Gamma_n)^{-1}\rho }{x}  } \diff{z}  \\ 
		&\le \norm{x} \int_{\mathcal{B}_j} \norm{ (z - \Gamma)^{-1}(\Gamma_n - \Gamma)(z - \Gamma)^{-1}(\Gamma_n - \Gamma)(z - \Gamma_n)^{-1}\rho  } \diff{z}  
		\end{align*} 
	\end{noobs} 

	Recall from Remark~\ref{rem:EmpiricalCovResolvent} that $z \in \mathcal{B}_j$ is also in the resolvent set of $\Gamma_n$. So we have $\mathrm{dist}(z, \sigma(\Gamma_n)) = \abs{z - \hat{\lambda}_j} \ge \abs{z - \lambda_j} - \abs{\hat{\lambda}_j - \lambda_j} = \delta_j / 2 - \delta_j / 4 = \delta_j / 2$. By the same arguments for \eqref{pfeq:YellowResolventNorm}, we have 
	\begin{equation}
		\norm{ (z - \Gamma_n)^{-1} } \ind_{\mathcal{A}_n} 
		\le \frac{1}{ \mathrm{dist}( z, \sigma(\Gamma_n) ) }  \ind_{\mathcal{A}_n} 
		\le \frac{1}{\delta_j / 2} \ind_{\mathcal{A}_n} 
		\le \frac{2}{\delta_j} 
		\lesssim j \log j 
		\label{pfeq:RoseResolventNormEventAn} 
	\end{equation}

	\begin{noobs}
		Consider any nonnegative random variable $X$ and any event $A$. Fix $\eta \in (0, 1)$. Then we have 
		\begin{align*}
			&\Prob( X \ind_A > \eta ) \\
			&= \underbrace{\Prob( X \ind_A > \eta \,| \ind_A > \eta)}_{\le 1} \Prob( \ind_A > \eta) + \underbrace{ \Prob( X \ind_A > \eta \,| \ind_A \le \eta) }_{=0} \Prob(\ind_A \le \eta) \\ 
			&\le \Prob(\ind_A > \eta) \\ 
			&= \Prob(A) 
		\end{align*}
		where in the second equality we use that $\{ \ind_A \le \eta \} = \{ \ind_A = 0 \}$, which implies $\{ X \ind_A > \eta \} = \{ X \cdot 0 > \eta \} = \{ 0 > \eta \} = \varnothing$. 
	\end{noobs}

	Consider the expectation and use Lemma~\ref{lem:ResolventGammaEmpiricalGammaApprox}, 
	\begin{align*}
		\E\left[ \int_{\mathcal{B}_j} \norm{ (z - \Gamma)^{-1} (\Gamma_n - \Gamma) }^2 \diff{z}\right]  
		&= \int_{\mathcal{B}_j} \E\left[ \norm{ (z - \Gamma)^{-1} (\Gamma_n - \Gamma) }^2 \right] \diff{z}  \nonumber \\ 
		&\lesssim \frac{j^3 \log j}{n} \mathrm{diam}(\mathcal{B}_j) \nonumber \\ 
		&= \frac{j^3 \log j}{n} \delta_j \nonumber \\ 
		&\lesssim \frac{j^{3} \log j}{n} \frac{1}{j \log j} \nonumber \\ 
		&= \frac{j^2}{n} 
	\end{align*} 
	By Markov's inequality, we thus have that 
	\begin{equation}
		\int_{\mathcal{B}_j} \norm{ (z - \Gamma)^{-1} (\Gamma_n - \Gamma) }^2 \diff{z} 
		= \BigOhPee \left( \frac{j^2}{n} \right) 
		\label{pfeq:RoseResolventGammaEmpiricalGammaProbBound} 
	\end{equation}
	Putting \eqref{pfeq:RoseResolventGammaEmpiricalGammaProbBound} and \eqref{pfeq:RoseResolventNormEventAn} together we have
	\begin{equation*}
		\int_{\mathcal{B}_j} \norm{ (z - \Gamma)^{-1} (\Gamma_n - \Gamma) }^2 \,\norm{(z - \Gamma_n)^{-1}} \ind_{\mathcal{A}_n} \diff{z}   
		\lesssim \BigOhPee \left( (j \log j)  \frac{j^2}{n}  \right) 
		= \BigOhPee \left( \frac{j^3 \log j}{n}  \right) 
	\end{equation*}
	So by Lemma~\ref{lem:SupNormalizedVec}, 
	\begin{align}
		\sup_{h \in \OptDomain} \absBIG{ \inner{\mathsf{R}_n}{ \frac{h}{t_n(h)} } } \nonumber 
		&\lesssim \sup_{h \in \OptDomain} \normBIG{ \frac{h}{t_n(h)} }  \BigOhPee \left( \frac{1}{n}\sum_{j = 1}^{k_n} j^3 \log j  \right) \nonumber  \\ 
		&\lesssim \BigOh( \sqrt{k_n \log k_n} ) \BigOhPee \left( \frac{ k_n^4 \log k_n }{n} \right) \nonumber \\ 
		&= \BigOhPee \left( \frac{ k_n^{9/2} (\log k_n)^{3/2} }{n} \right)  
		\label{pfeq:RoseRnProbBound} 
	\end{align} 
	This completes the proof of the $\mathsf{R}_n$ term.

	\underline{Summary:} Now we can finally put everything together. Putting \eqref{pfeq:YellowSnProbBound}, \eqref{pfeq:RoseRnProbBound} and \eqref{pfeq:YellowRemainderProbCvgRate} back into \eqref{pfeq:YellowYnDecomposition},  
	\begin{align*}
		\sup_{h \in \OptDomain} \absBIG{ \inner{\mathcal{Y}_n}{\frac{h}{t_n(h)}} }  
		&\lesssim \BigOhPee\left(  \frac{k_n^{3} \log k_n}{\sqrt{n}} \right) 
		+ \BigOhPee\left( \frac{k_n^{9/2} (\log k_n)^{3/2}}{n} \right) 
		+ \BigOhPee\left( \frac{k_n^{7/2} (\log k_n)^{3/2}}{\sqrt{n}} \right) 
	\end{align*}
	This completes the proof.

\end{proof}

\begin{prop} 
	For sufficiently large $n$, 
	\label{prop:TermSnGoesToZero} 
	\begin{align*}
		\sup_{h \in \OptDomain} \absBIG{ \inner{\mathcal{S}_n}{ \frac{h}{t_n(h)} } } 
		&\lesssim \BigOhPee\left( \frac{k_n^{11/2} (\log k_n)^{3/2}}{\sqrt{n}} \right)
	\end{align*}
\end{prop} 

\begin{proof}
	The proof of this result closely follows the development of the proof of Proposition~\ref{prop:TermYnGoesToZero}. By an entirely analogous arguments leading up to \eqref{pfeq:YellowDecomposition}, we obtain the decomposition 
	\begin{align}
		&\sup_{h \in \OptDomain} \absBIG{ \inner{\mathcal{S}_n}{\frac{h}{t_n(h)}} } \nonumber \\ 
		&\le \sup_{h \in \OptDomain} \normBIG{ \frac{h}{t_n(h)}  } \Bigg( \ind_{\mathcal{A}_n} \frac{1}{2\pi\im}\sum_{j = 1}^{k_n} \int_{\mathcal{B}_j} \abs{f_n(z)} \,\norm{(z - \Gamma)^{-1}(\Gamma_n - \Gamma)} \,\norm{(z - \Gamma)^{-1}U_n} \diff{z} \nonumber \\ 
		&\quad\quad + \ind_{\mathcal{A}_n} \frac{1}{2\pi\im} \sum_{j = 1}^{k_n} \int_{\mathcal{B}_j} \abs{f_n(z)} \,\norm{(z - \Gamma)^{-1}(\Gamma_n - \Gamma)}^2 \,\norm{(z - \Gamma_n)^{-1}} \,\norm{U_n} \diff{z} \nonumber \\
		&\quad\quad + \ind_{\mathcal{A}_n^c} \frac{1}{2\pi\im} \sum_{j = 1}^{k_n} \int_{\mathcal{B}_j} \abs{f_n(z)} \,\norm{(z - \Gamma)^{-1}} \,\norm{U_n} \diff{z} + \norm{r_n}\,\norm{U_n} \Bigg) 
		\label{pfeq:SunshineSnDecomposition} 
	\end{align}
	\begin{noobs}
		To see \eqref{pfeq:SunshineSnDecomposition}, 
	\begin{align*}
		\Gamma_n^\dagger - \Gamma^\dagger \nonumber 
			&\equiv \frac{1}{2\pi\im} \int_{\widehat{\mathcal{C}}_n} f_n(z) (z - \Gamma_n)^{-1} \diff{z} - \frac{1}{2\pi\im} \int_{\mathcal{C}_n} f_n(z) (z - \Gamma)^{-1} \diff{z} \nonumber \\ 
		&= \ind_{\mathcal{A}_n} \frac{1}{2\pi\im} \sum_{j = 1}^{k_n} \int_{\mathcal{B}_j} f_n(z)\left[ (z - \Gamma_n)^{-1} - (z - \Gamma)^{-1}  \right] \diff{z} \nonumber \\
			&\quad - \ind_{\mathcal{A}_n^c} \sum_{j = 1}^{k_n} \int_{\mathcal{B}_j} f_n(z) (z - \Gamma)^{-1} \diff{z} + r_n  \nonumber \\ 
		&= \ind_{\mathcal{A}_n} \frac{1}{2\pi\im} \sum_{j = 1}^{k_n} \int_{\mathcal{B}_j} f_n(z) (z - \Gamma_n)^{-1} (\Gamma_n - \Gamma) (z - \Gamma)^{-1} \diff{z} \nonumber \\
			&\quad - \ind_{\mathcal{A}_n^c} \sum_{j = 1}^{k_n} \int_{\mathcal{B}_j} f_n(z) (z - \Gamma)^{-1} \diff{z} + r_n \nonumber \\ 
		&= \ind_{\mathcal{A}_n} \frac{1}{2\pi\im} \sum_{j = 1}^{k_n} \int_{\mathcal{B}_j} f_n(z) (z - \Gamma_n)^{-1} (z - \Gamma) (z - \Gamma)^{-1} (\Gamma_n - \Gamma) (z - \Gamma)^{-1} \diff{z} \nonumber \\
			&\quad - \ind_{\mathcal{A}_n^c} \sum_{j = 1}^{k_n} \int_{\mathcal{B}_j} f_n(z) (z - \Gamma)^{-1} \diff{z} + r_n \nonumber \\ 
			&= \ind_{\mathcal{A}_n} \frac{1}{2\pi\im} \sum_{j = 1}^{k_n} \int_{\mathcal{B}_j} f_n(z) (z - \Gamma)^{-1}(\Gamma_n - \Gamma)(z - \Gamma)^{-1} \diff{z} \nonumber \\
			&\quad + \ind_{\mathcal{A}_n} \frac{1}{2\pi\im} \int_{\mathcal{B}_j} f_n(z) (z - \Gamma)^{-1}(\Gamma_n - \Gamma)(z - \Gamma)^{-1}(\Gamma_n - \Gamma)(z - \Gamma_n)^{-1} \diff{z} \nonumber \\ 
			&\quad - \ind_{\mathcal{A}_n^c} \sum_{j = 1}^{k_n} \int_{\mathcal{B}_j} f_n(z) (z - \Gamma)^{-1} \diff{z} + r_n 
	\end{align*}
	Now multiply by $U_n$, apply the inner product with $x$, and apply Cauchy-Schwartz inequality.  
	\end{noobs}

	Firstly, let's see that $U_n = \BigOhPee(1)$. Since $\norm{U_n} \le \frac{1}{n} \sum_{i = 1}^n \norm{X_i} \,\abs{\varepsilon_i}$, taking expectations and applying Jensen's inequality, we have $\E[\norm{U_n}] \le C$. By Markov's inequality, this implies 
	\begin{equation}
		\norm{U_n} = \BigOhPee(1) 
		\label{pfeq:SunshineUnProbBound} 
	\end{equation}
	\begin{noobs}
		To see \eqref{pfeq:SunshineUnProbBound}, 
	\begin{align*}
		\E[\norm{U_n}] 
		&\le \frac{1}{n} \sum_{i = 1}^n \E\left[ \norm{X_i} \, \abs{\varepsilon_i} \right] \\ 
		&\le \frac{1}{n} \sum_{i = 1}^n \sqrt{ \E\left[ \norm{X_i}^2 \varepsilon_i^2  \right]  } && \text{Jensen's inequality} \\ 
		&= \frac{1}{n} \sum_{i = 1}^n \sqrt{ \E\left[ \E[ \norm{X_i}^2 \varepsilon_i^2 \,|\, X_i ]  \right]  } \\ 
		&= \frac{1}{n} \sum_{i = 1}^n \sqrt{ \E\left[ \norm{X_i}^2 \E[ \varepsilon_i^2 \,|\, X_i ]  \right]  } \\ 
		&= \frac{1}{n} \sum_{i = 1}^n \sqrt{ \E\left[ \norm{X_i}^2 \sigma_\varepsilon^2  \right]  } && \text{Homosckedastic assumption}  \\ 
		&= \frac{1}{n} n \sigma_\varepsilon  \sqrt{ \E\left[ \norm{X_1}^2  \right]  } && \text{$X_i$ is iid}  \\ 
		&\le \sigma_\varepsilon  \sqrt{ \left( \E\left[ \norm{X_1}^4  \right] \right)^{2/4}  } && \text{Jensen's inequality}  \\ 
		&=: C < \infty 
	\end{align*}

	To see that $\norm{U_n} = \BigOhPee(1)$, let's just push definitions. Fix any $\epsilon > 0$. By Markov's inequality, 
	\begin{equation*}
		\Prob( \norm{U_n} > \epsilon ) 
		\le \frac{ \E[ \norm{U_n} ]}{\epsilon} 
		\le \frac{ C }{\epsilon} 
	\end{equation*}
	This just suggest we change variables and let $\eta = C / \epsilon$ and $M_\eta := C / \eta$. This implies for all $\eta > 0$, 
	\begin{equation*}
		\Prob( \norm{U_n} \le M_\eta) \le \eta 
	\end{equation*}
	which is exactly the definition that $\norm{U_n} = \BigOhPee(1)$. 
	\end{noobs}
	Combining \eqref{pfeq:SunshineUnProbBound} along with the analogous arguments of the last two expressions of \eqref{pfeq:YellowRemainderProbCvgRate} from Proposition~\ref{prop:TermYnGoesToZero}, we have that 
	\begin{align}
		\text{\parbox{4cm}{Last two expressions in parentheses of \eqref{pfeq:SunshineSnDecomposition} }} 
		&\lesssim \BigOhPee\left( \frac{k_n^3 \log k_n}{\sqrt{n}} \right) \BigOhPee(1) + \BigOhPee\left( \frac{k_n^2 \log k_n}{\sqrt{n}} \right) \BigOhPee(1) \nonumber \\ 
		&= \BigOhPee\left( \frac{k_n^3 \log k_n}{\sqrt{n}} \right) 
		\label{pfeq:SunshineLastTwoTermsProbBound} 
	\end{align}
	
	It thus suffices to concentrate the discussion on the first two expressions of \eqref{pfeq:SunshineSnDecomposition}. Thanks to the arguments from Proposition~\ref{prop:TermYnGoesToZero}, we have already handled the terms $\norm{(z - \Gamma_n)^{-1}} \ind_{\mathcal{A}_n}$ and $\int_{\mathcal{B}_j} \norm{ (z - \Gamma)^{-1}(\Gamma_n - \Gamma) } \diff{z}$. Thus it remains to discuss the terms:  (i) $\abs{f_n(z)}$ and (ii) $\norm{ (z - \Gamma)^{-1} U_n }$. 

	\textit{Term (i)}: By Condition~\ref{as:ConditionH}, we have that 
	\begin{equation}
		\sup_{z \in \mathcal{B}_j} \abs{f_n(z)} 
		\le \frac{1}{\delta_j}\left(1 + \frac{C}{\sqrt{n}} \right)  
		\lesssim (j \log j) \left( 1 + \frac{1}{\sqrt{n}} \right) 
		\label{pfeq:SunshineTermi} 
	\end{equation}
	\begin{noobs}
		To see \eqref{pfeq:SunshineTermi}, 
		\begin{align*}
			\sup_{z \in \mathcal{B}_j} \abs{f_n(z)} 
			&= \sup_{z \in \mathcal{B}_j} \absBIG{ \frac{1}{z} z f_n(z) } \\ 
			&\le \sup_{z \in \mathcal{B}_j} \frac{1}{\abs{z}} \sup_{z \in \mathcal{B}_j} \abs{z f_n(z)}  \\ 
			&= \frac{1}{\delta_j / 2} \sup_{z \in \mathcal{B}_j} \abs{z f_n(z)}  \\ 
			&= \frac{1}{\delta_j / 2} \sup_{z \in \mathcal{B}_j} \abs{(z f_n(z) - 1) + 1}  \\ 
			&\le \frac{2}{\delta_j} \left( \sup_{z \in \mathcal{B}_j} \abs{z f_n(z) - 1}  + 1 \right) \\ 
			&\le \frac{2}{\delta_j} \left( \frac{C}{\sqrt{n}}  + 1 \right) && \text{by Condition~\ref{as:ConditionH}} \\ 
			&\lesssim (j \log j) \left( 1 + \frac{1}{\sqrt{n}} \right) 
		\end{align*}
	\end{noobs}

	\textit{Term (ii)}: Fix any $z \in \mathcal{B}_j$. 
	\begin{noobs}
		While the decomposition \eqref{eq:RhoHatDecomposition} is correct and the definition of the operators in \eqref{eq:RhoHatDecompositionOperators} is appropriate, there's some subtlety in understanding the operator $\mathcal{S}_n$. In particular, how does one correct identify the domain and ranges of something like $(z - \Gamma)^{-1} U_n$? 

		Let's first think about $\mathcal{S}_n$ in more detail. It is clear that $\Gamma_n^\dagger - \Gamma \in \CompactOp \subseteq \BddOp$. But by construction of the tensor product $X_i \otimes \varepsilon_i$ which takes an arbitrary $h \in \HSpace$ and returns $(X_i \otimes \varepsilon_i)(h) = \inner{X_i}{h} \varepsilon_i \in \R$. That is, $U_n : \HSpace \to \R$ and moreover, by continuity and linearity we indeed have $U_n \in \HSpace^*$. Then as stated the object `` $(\Gamma_n - \Gamma)U_n$'' has a rather confusing domain and range: due to $U_n$ it should first take an element from $\HSpace$ and return $\R$; but $\Gamma_n - \Gamma$ eats elements from $\HSpace$ and returns $\HSpace$ and does not eat $\R$. Thus as defined, `` $(\Gamma_n - \Gamma)U_n$'' doesn't make much sense. A same reasoning goes to see that as stated ``$(z - \Gamma)^{-1} U_n$'' is confusing because the resolvent is also an operator. In other words, we need to make sense of what does it mean by a composition of an operator with a linear functional. 

		Here we should invoke the definitions of \cite[Chapter VI, \S 1]{conway1994course}. Suppose $\mathscr{X}, \mathscr{Y}$ are vector spaces and $T : \mathscr{X} \to \mathscr{Y}$ is a linear transformation. Let $\mathscr{Y}' := \text{all of the linear functionals of $\mathscr{Y} \to \Field$}$.
		\footnote{
			Note this $\mathscr{Y}'$ is not necessarily equivalent to $\mathscr{Y}^*$ because we didn't impose that $\mathscr{X}, \mathscr{Y}$ are Banach / Hilbert spaces. 
		}
		If $y' \in \mathscr{Y}'$, then $y' \circ T : \mathscr{X} \to \Field$ is easily seen to be a linear functional on $\mathscr{X}$. That is, $y' \circ T \in \mathscr{X}'$. This defines a map, 
		\begin{equation*}
			T' : \mathscr{Y}' \to \mathscr{X}'  
		\end{equation*}
		by $T'(y') := y' \circ T$. 

		Let's specialize the above discussions to our context. In particular, it is important to note that $\Gamma_n - \Gamma$ and $(z - \Gamma)^{-1}$ are both self-adjoint. The definition of this map means we should really read ``$(\Gamma_n - \Gamma) U_n = (\Gamma_n - \Gamma)^* U_n$'' as $U_n (\Gamma_n - \Gamma)$, and likewise we should read ``$(z - \Gamma)^{-1} U_n = ( (z - \Gamma)^{-1} )^* U_n$'' as $U_n (z - \Gamma)^{-1}$. 
	\end{noobs} 
	By definition of the adjoint of a linear operator and using the Hilbert-Schmidt norm, 
	\begin{align*}
		\norm{ (z - \Gamma)^{-1} U_n }^2 
		&= \norm{ U_n (z - \Gamma)^{-1} }^2 \\ 
		&\le \norm{ U_n (z - \Gamma)^{-1} }^2_{\textrm{HS}} \\ 
		&= \sum_{l = 1}^\infty \frac{1}{(z - \lambda_l)^2} \left[ \frac{1}{n^2} \sum_{i = 1}^n \inner{X_i}{e_l}^2 \varepsilon_i^2 + \frac{1}{n^2} \sum_{i \neq j}^n \inner{X_i}{e_l} \inner{X_j}{e_l} \varepsilon_i \varepsilon_j   \right] 
	\end{align*}
	Taking expectations and using the KL expansion, 
	\begin{align*}
		\E[ \norm{(z - \Gamma)^{-1} U_n}^2 ] 
		&\le  \frac{\sigma_\varepsilon^2}{n} \sum_{l = 1}^\infty \frac{\lambda_l}{(z - \lambda_l)^2} \\ 
		&= \frac{\sigma_\varepsilon^2}{n} \left( \frac{\lambda_j}{(z - \lambda_j)^2} + \sum_{l \neq j}^\infty \frac{\lambda_l}{(z - \lambda_l)^2} \right) \\ 
		&\le \frac{\sigma_\varepsilon^2}{n} \left( \frac{\lambda_j}{(\delta_j / 2)^2} + C \frac{j^2}{\lambda_j} \right) \\ 
		&\lesssim \frac{1}{n} \left( \frac{1}{j \log j} (j \log j)^2 + j^2 (j \log j) \right) \\ 
		&= \frac{1}{n} \left( j \log j + j^3 \log j \right) \\ 
		&\lesssim \frac{j^3 \log j}{n} 
	\end{align*}
	where the third line follows from \eqref{pfeq:BSESongEagle} in the proof of Proposition~\ref{prop:TermYnGoesToZero}. 
	\begin{noobs}
		To see the above calculations in some detail, 
	\begin{align*}
		\norm{ (z - \Gamma)^{-1} U_n }^2 
		&= \norm{ U_n (z - \Gamma)^{-1} }^2 \\ 
		&\le \norm{ U_n (z - \Gamma)^{-1} }^2_{\textrm{HS}} \\ 
		&= \sum_{l = 1}^\infty \norm{ U_n (z - \Gamma)^{-1} (e_l) }^2 \\ 
		&= \sum_{l = 1}^\infty \normBIG{ U_n \frac{1}{z - \lambda_l} e_l }^2 \\ 
		&= \sum_{l = 1}^\infty \frac{1}{(z - \lambda_l)^2} \norm{ U_n e_l }^2 \\ 
		&= \sum_{l = 1}^\infty \frac{1}{(z - \lambda_l)^2} \inner{ \frac{1}{n} \sum_{i = 1}^n \inner{X_i}{e_l} \varepsilon_i }{ \frac{1}{n} \sum_{j = 1}^n \inner{X_j}{e_l} \varepsilon_j }   \\ 
		&= \sum_{l = 1}^\infty \frac{1}{(z - \lambda_l)^2} \frac{1}{n^2} \sum_{i, j = 1}^n \inner{X_i}{e_l} \inner{X_j}{e_l} \varepsilon_i \varepsilon_j   \\ 
		&= \sum_{l = 1}^\infty \frac{1}{(z - \lambda_l)^2} \left[ \frac{1}{n^2} \sum_{i = 1}^n \inner{X_i}{e_l}^2 \varepsilon_i^2 + \frac{1}{n^2} \sum_{i \neq j}^n \inner{X_i}{e_l} \inner{X_j}{e_l} \varepsilon_i \varepsilon_j   \right]  
	\end{align*}
	Note that it is \textit{critical} that we had the discussions and understanding of an adjoint of a linear operator. Without the correct understanding that we should read ``$(z - \Gamma)^{-1} U_n$'' as $U_n (z - \Gamma)^{-1}$, it would be immediate that applying the Hilbert-Schmidt norm in the third line would entail the term ``$\norm{(z - \Gamma)^{-1} U_n(e_l)}$''. Following down this path effectively goes nowhere (because the composition itself makes no sense). Thus the correct interpretation is not only mathematically important but also facilitates for the correct bound. 

	In addition, the analogous calculation on pg 354 of \cite{cardot2007clt} makes no sense. Firstly, they display an equality but clearly an inequality is needed. Secondly, they use an object ``$(z - \Gamma)^{-1/2}$'' --- this is questionable. While we can define the resolvent operator $(z - \Gamma)^{-1}$, one needs to at least ensure that it is positive-definite to define its square-root. This was never shown in the paper and it is not true in general. Thus throughout our work, we neither need nor write anything of the form ``$[(z - \Gamma)^{-1}]^{1/2}$''. 
	\end{noobs} 
	\begin{noobs} 
	\begin{align*}
		&\E[ \norm{(z - \Gamma_n)^{-1} U_n}^2 ] \\ 
		&\le \sum_{l = 1}^\infty \frac{1}{(z - \lambda_l)^2} \left[ \frac{1}{n^2} \sum_{i = 1}^n \E[ \inner{X_i}{e_l}^2 \varepsilon_i^2]  + \frac{1}{n^2} \sum_{i \neq j}^n \E[ \inner{X_i}{e_l} \inner{X_j}{e_l} \varepsilon_i \varepsilon_j ]   \right]  \\ 
		&= \sum_{l = 1}^\infty \frac{1}{(z - \lambda_l)^2} \left[ \frac{1}{n^2} \sum_{i = 1}^n \E\left[ \E[ \inner{X_i}{e_l}^2 \varepsilon_i^2 \,|\, X_i ] \right]  
		+ \frac{1}{n^2} \sum_{i \neq j}^n \E\left[ \E[ \inner{X_i}{e_l} \inner{X_j}{e_l} \varepsilon_i \varepsilon_j \,|\, X_i, X_j ] \right]   \right]  \\ 
		&= \sum_{l = 1}^\infty \frac{1}{(z - \lambda_l)^2} \left[ \frac{1}{n^2} \sum_{i = 1}^n \E\left[ \inner{X_i}{e_l}^2 \underbrace{\E[  \varepsilon_i^2 \,|\, X_i ]}_{= \sigma_\varepsilon^2} \right]  
		+ \frac{1}{n^2} \sum_{i \neq j}^n \E\left[  \inner{X_i}{e_l} \inner{X_j}{e_l} \underbrace{\E[ \varepsilon_i \,|\, X_i ]}_{= 0} \underbrace{\E[ \varepsilon_j \,|\, X_j ]}_{=0} \right]   \right]  \\ 
		&=  \sum_{l = 1}^\infty \frac{1}{(z - \lambda_l)^2} \frac{\sigma_\varepsilon^2}{n} \sum_{i = 1}^n \E[ \inner{X_i}{e_l}^2 ] \\ 
		&=  \sum_{l = 1}^\infty \frac{1}{(z - \lambda_l)^2} \frac{\sigma_\varepsilon^2}{n} n \E[ \inner{X_1}{e_l}^2 ] \quad\quad\text{iid} \\ 
		&=  \frac{\sigma_\varepsilon^2}{n} \sum_{l = 1}^\infty \frac{1}{(z - \lambda_l)^2} \E\left[  \inner{\sum_{k = 1}^\infty \sqrt{\lambda_k} \xi_k e_k}{e_l}^2 \right] \quad\quad\text{KL expansion} \\ 
		&=  \frac{\sigma_\varepsilon^2}{n} \sum_{l = 1}^\infty \frac{1}{(z - \lambda_l)^2} \E\left[ \sum_{k = 1}^\infty \innersmall{ \sqrt{\lambda_k} \xi_k e_k}{e_l} \sum_{k' = 1}^\infty \innersmall{ \sqrt{\lambda_{k'}} \xi_{k'} e_{k'}}{e_l} \right]  \\ 
		&=  \frac{\sigma_\varepsilon^2}{n} \sum_{l = 1}^\infty \frac{1}{(z - \lambda_l)^2} \E\left[ \sum_{k = 1}^\infty \sqrt{\lambda_k} \xi_k \innersmall{e_k}{e_l} \sum_{k' = 1}^\infty \sqrt{\lambda_{k'}} \xi_{k'} \innersmall{ e_{k'}}{e_l} \right]  \\ 
		&=  \frac{\sigma_\varepsilon^2}{n} \sum_{l = 1}^\infty \frac{1}{(z - \lambda_l)^2} \E\left[ \sqrt{\lambda_l} \xi_l \sqrt{\lambda_l} \xi_l \right]  \quad\quad\text{orthogonality of eigenvectors} \\ 
		&=  \frac{\sigma_\varepsilon^2}{n} \sum_{l = 1}^\infty \frac{1}{(z - \lambda_l)^2} \lambda_l \underbrace{\E\left[  \xi_l^2  \right]}_{= 1}  \quad\quad\text{by construction of KL expansion} \\ 
		&=  \frac{\sigma_\varepsilon^2}{n} \sum_{l = 1}^\infty \frac{\lambda_l}{(z - \lambda_l)^2}
	\end{align*}
	\end{noobs}
	Thus by Chebyshev's inequality, it follows we have for all $z \in \mathcal{B}_j$, 
	\begin{equation}
		\norm{(z - \Gamma)^{-1} U_n} = \BigOhPee\left( \frac{j^{3/2} (\log j)^{1/2}}{\sqrt{n}} \right) 
		\label{pfeq:SunshineTermii} 
	\end{equation}

	Putting \eqref{pfeq:SunshineTermi} and \eqref{pfeq:SunshineTermii} together along with the already discussed terms from Proposition~\ref{prop:TermYnGoesToZero}, it follows 
	\begin{align}
		\text{\parbox{4cm}{The $j$th summand of the 1st expression in parentheses of \eqref{pfeq:SunshineSnDecomposition}} }
		&\lesssim 
		(j \log j) \left( 1 + \frac{1}{\sqrt{n}} \right) \cdot 
		\BigOhPee\left( \sqrt{\frac{j^2}{n}} \right) \cdot 
		\BigOhPee\left( \frac{j^{3/2} (\log j)^{1/2} }{ \sqrt{n} }  \right) \nonumber \\ 
		&\lesssim 
		\BigOhPee\left( \frac{j^{7/2} (\log j)^{3/2} }{ n } \right) 
		\label{pfeq:SunshineFirstExpressionProbBound} 
	\end{align}

	Let's now discuss the second expression of \eqref{pfeq:SunshineSnDecomposition}. Using \eqref{pfeq:SunshineTermi}, \eqref{pfeq:RoseResolventGammaEmpiricalGammaProbBound}, \eqref{pfeq:RoseResolventNormEventAn} and \eqref{pfeq:SunshineUnProbBound}, we have 
	\begin{align}
		\text{\parbox{4cm}{The $j$th summand of the 2nd expression in parentheses of \eqref{pfeq:SunshineSnDecomposition}} }
		&\lesssim 
		(j \log j) \left( 1 + \frac{1}{\sqrt{n}} \right) \cdot 
		\BigOhPee\left( \frac{j^2}{n} \right) \cdot 
		\mathcal{O}_{\mathrm{a.s.}}( j \log j) \cdot 
		\BigOhPee(1) \nonumber \\ 
		&\lesssim \BigOhPee\left( \frac{j^4 (\log j)^2}{n} \right) 
		\label{pfeq:SunshineSecondExpressionProbBound} 
	\end{align}

	Finally, summing \eqref{pfeq:SunshineFirstExpressionProbBound} and \eqref{pfeq:SunshineSecondExpressionProbBound}, using \eqref{pfeq:SunshineLastTwoTermsProbBound} in \eqref{pfeq:SunshineSnDecomposition} and using Lemma~\ref{lem:SupNormalizedVec}, 
	\begin{align*}
		&\sup_{h \in \OptDomain} \absBIG{ \inner{\mathcal{S}_n}{ \frac{h}{t_n(h)} } } \\ 
		&\lesssim \BigOh(\sqrt{k_n \log k_n}) \left( \BigOhPee\left( \frac{k_n^{9/2} (\log k_n)^{3/2}}{n} \right) + \BigOhPee\left( \frac{k_n^5 (\log k_n)^2}{n} \right) + \BigOhPee\left( \frac{k_n^3 \log k_n}{\sqrt{n}} \right) \right) \\ 
		&= \BigOhPee\left( \frac{k_n^{11/2} (\log k_n)^{3/2}}{ \sqrt{n} } \right) 
	\end{align*}
	This completes the proof. 

\end{proof} 

The following result summarizes the discussions of Step~\ref{eq:ProofStepI}. 
\begin{prop}[Nuisance terms converge rate] 
	\label{prop:SupTYSConvergesToZero} 
	For sufficiently large $n$, 
	\begin{equation*}
		\sup_{h \in \OptDomain} \absBIG{ \frac{ \inner{(\mathcal{T}_n + \mathcal{Y}_n + \mathcal{S}_n) \rho}{h} }{t_n(h)}  } 
		\lesssim 
		\BigOhPee\left( \frac{\KeyCvgRateStepI}{\sqrt{n}} \right). 
	\end{equation*}
\end{prop}

\begin{proof}
	By Propositions~\ref{prop:TermTnGoesToZero}, \ref{prop:TermYnGoesToZero} and \ref{prop:TermSnGoesToZero}, the displayed equation on the left hand side is bounded above by 
	\begin{align*}
		&\BigOhPee(\sqrt{k_n \log k_n}) o_\Prob\left( \frac{1}{\sqrt{n}} \right) 
		+ \BigOhPee\left( \frac{k_n^{7/2} (\log k_n)^{3/2}}{\sqrt{n}} \right) 
		+ \BigOhPee\left( \frac{k_n^{11/2} (\log k_n)^{3/2}}{\sqrt{n}} \right) 
	\end{align*}
\end{proof}

\subsection{Step II} 
We now move onto Step~\ref{eq:ProofStepII}. The key to showing that Step~\ref{eq:ProofStepII} holds is to cast the $\mathcal{R}_n$ term into an empirical process theory framework and apply the approximation results of \cite{chernozhukov2014gaussian}.  Let's setup some standard notations. Let $N(\epsilon, T, \norm{\cdot})$ denote the covering number of radius $\epsilon > 0$ for the metric space $(T, \norm{\cdot})$. Let's also denote the uniform entropy integral (see \cite[Chapter 2.14]{van1996weak}) for the set $T$ equipped with measurable cover $F$, 
\begin{equation*}
	J(\delta, T) := \sup_Q \int_0^\delta \sqrt{ 1 + \log N(\epsilon\norm{F}_{Q,2}, T, L_2(Q)) } \diff{\epsilon}  
\end{equation*}
where the supremum is taken over all discrete probability measures $Q$ with $\norm{F}_{Q,2} > 0$. For an arbitrary set $T$, we will denote $l^\infty(T)$ as the space of all bounded functions $T \to \R$ with the uniform norm $\norm{f}_T := \sup_{t \in T} \abs{f(t)}$. 

By the Riesz representation theorem, $\OptDomain$ and its dual space $\OptDomain^*$ are isometrically isomorphic (this is especially since we're working with real valued Hilbert spaces). Thus for each $h \in \OptDomain$, we can identify $h \in \OptDomain$ with $h^* \in \OptDomain^*$ such that $h^*(\cdot) = \inner{h}{\cdot}$. With some abuse of notations, we will write 
\begin{equation*}
	\mathcal{R}_n(h) 
	\equiv \frac{1}{n}\sum_{i = 1}^n \innersmall{h}{\Gamma^\dagger X_i\varepsilon_i} 
	= \frac{1}{n}\sum_{i = 1}^n h^*(\Gamma^\dagger X_i \varepsilon_i) 
	= \frac{1}{n}\sum_{i = 1}^n h^*(V_i) 
	=: \mathcal{R}_n(h^*) 
\end{equation*}
for $V_{i,n} := \Gamma^\dagger X_i \varepsilon_i$. Note that $\Gamma^\dagger$ depends on $n$ but is otherwise entirely deterministic, and hence for each $n$, $\{V_{1,n}, \ldots, V_{n,n} \}$ is an iid sequence. Note that $\sqrt{ P|h^*|^2 } = t_n(h) - a_n$ and noting that $P h^* = \E[ h^*(V_1) ] = 0$, we normalize to write 
\begin{equation}
	\sqrt{n} \frac{ \mathcal{R}_n(h) }{\sigma_\varepsilon t_n(h)} 
	= \frac{1}{\sqrt{n}} \sum_{i = 1}^n \frac{ h^*(V_{i,n}) }{\sigma_\varepsilon (\sqrt{P|h^*|^2} + a_n) } =: \Gn g, \quad g \in \FnOptDomain 
	\label{eq:RnEmpiricalProcessForm} 
\end{equation}
where we define the class, 
\begin{equation}
	\FnOptDomain := \left\{ \frac{h^*}{\sigma_\varepsilon (\sqrt{P|h^*|^2} + a_n)} : h^* \in \OptDomain^* \right\} 
	\label{eq:RnFnSet} 
\end{equation}
In other words, we have the equivalence between the expressions $\sup_{h \in \OptDomain} \frac{\sqrt{n}}{\sigma_\varepsilon t_n(h)} \mathcal{R}_n(h)$ and $\sup_{f \in \FnOptDomain} \Gn g$. Most importantly this casts the handling of the $\mathcal{R}_n$ term into an empirical process framework. 

Let's first record some basic entropic properties about $\FnOptDomain$. These entropic properties are particularly simple to derive precisely due to the structure of $\OptDomain$. 
\begin{lem}[Entropic properties of $\FnOptDomain$]
	\label{lem:FnOptDomainProperties} 
	\begin{enumerate}[(i)] 
		\item The VC index of $\FnOptDomain$ is $V(\FnOptDomain) \le (k_n + 2)^2$. 

		\item A measurable cover for $\FnOptDomain$ is the (constant) function 
			\begin{equation*}
				F_n(g) \equiv \frac{1}{\sigma_\varepsilon \left(f_n(\lambda_1) \lambda_{k_n}^{1/2} + a_n \right)}, \quad \text{for all $g \in \FnOptDomain$}.  
			\end{equation*}

		\item The $\epsilon$-covering number for $\FnOptDomain$ satisfies for any discrete probability measure $Q$, 
			\begin{equation*}
				N(\epsilon\norm{F_n}_{Q,2}, \FnOptDomain, L_2(Q) ) 
				\le \left( \frac{A_n}{\epsilon} \right)^{\nu_n} 
			\end{equation*}
			where $A_n := \left( K V(\FnOptDomain) (16 e)^{V(\FnOptDomain)} \right)^{ \frac{1}{ 2( V(\FnOptDomain) - 1 )}}$ and $\nu_n := 2( V(\FnOptDomain) - 1)$, and $\epsilon \in (0,1)$. 

		\item Assume $\nu_n \ge 1$. Then the uniform entropy integral for $\FnOptDomain$ satisfies, 
			\begin{equation*}
				J(\delta, \FnOptDomain) \le \delta \sqrt{\nu_n} \left( 1  + \sqrt{1 + \log(A_n / \delta)} \right) 
			\end{equation*}

		\item If $\delta \in (0,1]$ is a constant that is independent of $n$, then for sufficiently large $n$, 
			\begin{equation*}
				J(\delta, \FnOptDomain) 
				\lesssim \delta \left( \sqrt{\BigOh(V(\FnOptDomain)} ) + \sqrt{ \BigOh(V(\FnOptDomain)) - \log\delta }  \right) \\ 
				\lesssim \delta \BigOh( k_n ) 
			\end{equation*}

	\end{enumerate}
	
\end{lem}

\begin{proof}
	\underline{(i)} Since $\OptDomain$ is isomorphic to $\R^{k_n}$, and thus $\OptDomain^*$ is isomorphic to $\R^{k_n}$. By \cite{van1996weak} Lemma 2.6.15 and Lemma 2.6.18(vii), the VC index of $\FnOptDomain$ satisfies $V(\FnOptDomain) \le (k_n + 2)^2$. 
	
	\underline{(ii)} Recall Lemma~\ref{lem:SupNormalizedVec}. Moreover, by Riesz representation theorem $\norm{h} = \norm{h^*}$ for any $h \in \OptDomain$ and where $h^*$ is its unique dual. For any non-zero $g \in \FnOptDomain$ there exists some non-zero $h^* \in \OptDomain^*$ such that, 
	\begin{equation*}
		\norm{g} 
		= \normBIG{ \frac{h^*}{\sigma_\varepsilon( \sqrt{P|h^*|^2} + a_n) } }  
		\le \frac{\norm{h^*}}{\sigma_\varepsilon \left(f_n(\lambda_1) \lambda_{k_n}^{1/2} \norm{h} + a_n \right) } 
		\le \frac{1}{\sigma_\varepsilon \left( f_n(\lambda_1) \lambda_{k_n}^{1/2} + a_n \right)} 
	\end{equation*}

	\underline{(iii)} By \cite[Theorem 2.6.7]{van1996weak}, 
	\begin{equation*}
		N(\epsilon \norm{F_n}_{Q, 2}, \FnOptDomain, L_2(Q)) 
		\le K V(\FnOptDomain) (16 e)^{V(\FnOptDomain)} \left( \frac{1}{\epsilon} \right)^{2 (V(\FnOptDomain) - 1)}   
	\end{equation*}
	for an universal constant $K$ and $\epsilon \in (0,1)$. The result follows by rearranging and defining the terms $A_n$ and $\nu_n$. 

	\underline{(iv)} By part (iii) and change of variables, 
	\begin{align*}
		J(\delta, \FnOptDomain) 
		\le \int_0^\delta \sqrt{1 + \nu_n \log(A_n / \epsilon)} \diff{\epsilon} 
		\le A_n \sqrt{\nu_n} \int_{A_n / \delta}^\infty \frac{ \sqrt{1 + \log \epsilon} }{\epsilon^2} \diff{\epsilon} 
	\end{align*}
	Observe we have the indefinite integral, 
	\begin{equation*}
		\int \frac{\sqrt{1 + \log x}}{x^2} \diff{x} 
		= -\frac{\sqrt{1 + \log x}}{x} - \frac{1}{2} e \Gamma\left( \frac{1}{2}, 1 + \log x \right) + \mathrm{const}  
	\end{equation*}
	where here $\Gamma(s,z) := \int_z^\infty t^{s - 1} e^{-t} \diff{t}$ is the upper incomplete gamma function. Since for a fixed $s$, $\lim_{z \to \infty} \Gamma(s, z) = 0$, it follows that $\lim_{x \to \infty} \Gamma\left(\frac{1}{2}, 1  + \log(x)  \right) = 0$. It is clear that $\lim_{x \to \infty} \frac{\sqrt{1 + \log x}}{x} = 0$. Thus it follows, 
	\begin{align*}
		\int_{A_n / \delta}^\infty \frac{ \sqrt{1 + \log \epsilon} }{\epsilon^2} \diff{\epsilon} 
		&= \frac{\sqrt{1 + \log (A_n / \delta)}}{A_n / \delta} + \frac{1}{2} e \Gamma\left( 1 + \frac{1}{2}, \log(A_n / \delta)  \right) \\ 
		&= \frac{\sqrt{1 + \log (A_n / \delta)}}{A_n / \delta} + \frac{e \sqrt{\pi}}{2} \Efrc\left( \sqrt{ 1  + \log(A_n / \delta)  } \right)   
	\end{align*}
	where $\Efrc$ is the complementary error function, $\Efrc(x) := 1 - \Efr(x) = 1 - \frac{2}{\sqrt{\pi}} \int_0^x e^{-t^2} \diff{t} = \frac{2}{\sqrt{\pi}} \int_x^\infty e^{-t^2} \diff{t}$. Using the bound $\Efrc(x) \le e^{-x^2}$, we obtain the bound as displayed. 

	\underline{(v)} By (iii), 
	\begin{align*}
		\nu_n \log A_n 
		= \log K + \log V(\FnOptDomain) + V(\FnOptDomain)\log(16e) 
		= \BigOh(V(\FnOptDomain)) 
	\end{align*}
	and moreover, $\nu_n = 2 (V(\FnOptDomain) - 1) = \BigOh( V(\FnOptDomain) )$. And thus by (iv), 
	\begin{align*}
		J(\delta, \FnOptDomain) 
		&\le \delta\left( \sqrt{\BigOh(V(\FnOptDomain))} + \sqrt{\BigOh(V(\FnOptDomain)) + \BigOh(V(\FnOptDomain)) - \log\delta  } \right) \\ 
		&= \delta \left( \sqrt{\BigOh(V(\FnOptDomain))}  + \sqrt{ \BigOh(V(\FnOptDomain)) - \log\delta }  \right) \\ 
		&\lesssim \delta \BigOh( \sqrt{V(\FnOptDomain)} ) 
	\end{align*}
	Apply (i) which implies $V(\FnOptDomain) = \BigOh(k_n^2)$ and we have the displayed result. 

\end{proof}

Next we state a slightly modified version of the key results of \cite{chernozhukov2014gaussian} that's applicable for our context. 
\begin{thm}[Gaussian approximation to suprema of empirical processes index by VC type classes; \cite{chernozhukov2014gaussian}]
	\label{thm:GaussianApproximationSuprema} 
	Fix $n \ge 1$. Let $(\FnOptDomain, \norm{\cdot})$ be a subset of a normed separable space of real functions $f : \mathcal{X} \to \R$ and is equipped with an envelope $F_n$. Suppose: 
	\begin{enumerate}[(i)] 
		\item $\FnOptDomain$ is pre-Gaussian. That is, there exists a tight Gaussian random variable $G_{P,n}$ in $l^\infty(\FnOptDomain)$ with mean zero and covariance function, 
			\begin{equation*}
				\E[ G_{P,n}(f) G_{P,n}(g) ] = P(fg) = \E[ f(Z_1) g(Z_1) ], \quad \text{for all $f, g \in \FnOptDomain$} 
			\end{equation*}

		\item The $\epsilon$-covering number of $\FnOptDomain$ satisfies $\sup_Q N(\epsilon \norm{F_n}_{Q,2}, \FnOptDomain, L_2(Q)) \le \left( \frac{A_n}{\epsilon} \right)^{\nu_n}$ for some $\nu_n \ge 1$ and $A_n > 0$, and where the supremum is taken over all discrete probability measures $Q$ such that $\norm{F_n}_{Q,2} > 0$; 

		\item For some $b_n \ge \sigma_n > 0$ and $q \in [4, \infty]$, we have $\sup_{f \in \FnOptDomain} P|f|^k \le \sigma_n^2 b_n^{k - 2}$ for $k = 2, 3$ and $\norm{F_n}_{P,q} \le b_n$. 
	\end{enumerate}

	Let $Z_n := \Gn f$. Then for every $\gamma \in (0,1)$, there exists a random variable $\tilde{Z}_n := \sup_{f \in \FnOptDomain} G_{P,n} f$ such that 
	\begin{equation*} 
		\Prob\left( |Z_n - \tilde{Z}_n| > \frac{b_n K_n}{\gamma^{1/2} n^{1/2 - 1/q} } + \frac{(b_n\sigma_n)^{1/2} K_n^{3/4}}{\gamma^{1/2}n^{1/4}} + \frac{(b_n \sigma_n^2 K_n^2)^{1/3}}{ \gamma^{1/3}n^{1/6} }\right) 
		\le C\left( \gamma + \frac{\log n}{n} \right) 
	\end{equation*}
	where $K_n := c \nu_n \max\left\{ \log n \,,\, \left( \left( 1 + \sqrt{ 1 +  \log \frac{A_n b_n}{\sigma_n} }  \right) \right)^2 \right\}$ and $c, C > 0$ are constants that only depend on $q$. 
\end{thm}

\begin{rem}
	This result is nothing more than Corollary 2.2 of \cite{chernozhukov2014gaussian}, which is based on their key result Theorem 2.1. We refer to their paper for the proof. But let's remark on what small proof modifications we need to adapt their result to our Theorem~\ref{thm:GaussianApproximationSuprema}. The major difference between our stated result and their Corollary 2.2 is the condition on the constant $A$ in the covering number bound. Their Corollary 2.2 requires $A \ge e$ but we do not impose this requirement here. Indeed, from Lemma~\ref{lem:FnOptDomainProperties}(i), it is unnatural to require that $A_n \ge e$ for all $n$, especially since we only have an upper bound for $V(\FnOptDomain)$ and not a lower bound. For their proofs, the authors only require the condition $A \ge e$ to arrive at the uniform entropy integral condition $J(\delta, \mathcal{F}) \lesssim \delta \sqrt{\nu \log(A / \delta)}$. From Lemma~\ref{lem:FnOptDomainProperties}(iv), we have instead a slightly larger bound of $J(\delta, \FnOptDomain) \le \delta \sqrt{\nu_n} ( 1 + \sqrt{1 + \log (A_n / \delta)} )$. Consequently and by inspecting the proofs of their Corollary 2.2, it suffices to replace their definition of $K_n = c \nu (\log n \vee \log \frac{A b}{\sigma} )$ with our slightly larger $K_n$, then the remainder of their proof goes through to our case. 
\end{rem}

This following result will conclude Step~\ref{eq:ProofStepII}. 
\begin{prop}
	\label{prop:SupZnGaussianApproximation}  
	Fix any $\gamma \in (0,1)$ and assume $k_n / n \to 0$. Then there exists a mean zero Gaussian process $G_{P,n}$ in $\ell^\infty(\OptDomain)$ with the displayed covariance function \eqref{eq:SupZnGaussianApproximationCovFun} such that the random variables $Z_n := \sup_{h \in \OptDomain} \frac{\sqrt{n}}{\sigma_\varepsilon t_n(h)} \mathcal{R}_n(h)$ and $\widetilde{Z}_n := \sup_{h \in \OptDomain} G_{P,n} h$ have, 
	\begin{align*}
		&\abs{Z_n - \widetilde{Z}_n} \\
		&\lesssim \BigOhPee\left( \KeyCvgRateStepII \right) 
	\end{align*}
\end{prop}

\begin{proof}
	We apply Theorem~\ref{thm:GaussianApproximationSuprema} with the choice of $q = \infty$ and recall the notations from that theorem statement. Fix any $n \ge 1$ and any $g \in \FnOptDomain$. Firstly the second moment is, 
	\begin{equation*}
		P|g|^2 
		= \frac{P|h^*|^2}{\sigma_\varepsilon^2( \sqrt{P|h^*|^2} + a_n )^2} 
		\le \frac{1}{\sigma_\varepsilon^2} 
	\end{equation*}
	For the third moment, 
	\begin{align*}
		P|g|^3 
		&= \frac{1}{(\sigma_\varepsilon \sqrt{P|h^*|^2} + a_n )^{3}}  P|h^*|^3 \\ 
		&\le \frac{1}{(\sigma_\varepsilon \sqrt{P|h^*|^2} + a_n )^{3}}  P|h^*|^3 \\ 
		&\le \frac{1}{ \sigma_\varepsilon^3 \left( f_n(\lambda_1) \lambda_{k_n}^{1/2} \norm{h^*} + a_n \right)^3 } \sqrt{ P|h^*|^6 } \\ 
		&\le \frac{1}{ \sigma_\varepsilon^3 \left( f_n(\lambda_1) \lambda_{k_n}^{1/2} \norm{h^*} + a_n \right)^3 } \sqrt{ \sup_j \E[\xi_j^6] \lambda_1^3 f_n(\lambda_{k_n})^6 \norm{h^*}^6 } \\ 
		&= C \frac{f_n(\lambda_{k_n})^3 \norm{h^*}^3}{ \sigma_\varepsilon^3 \left( f_n(\lambda_1) \lambda_{k_n}^{1/2} \norm{h^*} + a_n \right)^3 }  \\ 
		&\le C \frac{f_n(\lambda_{k_n})^3}{ \sigma_\varepsilon^3 \left( f_n(\lambda_1) \lambda_{k_n}^{1/2} + a_n \right)^3 }  
	\end{align*}

	Thus it suffices to set, 
	\begin{align*}
		\sigma_n &:= \frac{1}{\sigma_\varepsilon}, \\ 
		b_n 
						 &:= \max\left\{ \frac{1}{\sigma_\varepsilon \left( f_n(\lambda_1) \lambda_{k_n}^{1/2} + a_n \right)}  \,,\,  C \frac{ f_n(\lambda_{k_n})^3 }{\sigma_\varepsilon^3 \left(f_n(\lambda_1) \lambda_{k_n}^{1/2} + a_n \right)^3 }  \right\} \\ 
						 &= \frac{1}{\sigma_\varepsilon \left(f_n(\lambda_1) \lambda_{k_n}^{1/2} + a_n \right)} \max\left\{ 1 \,,\, C \frac{ f_n(\lambda_{k_n})^3 }{\sigma_\varepsilon^2 \left(f_n(\lambda_1) \lambda_{k_n}^{1/2} + a_n \right)^2 }  \right\} \\ 
						 &\lesssim (\sqrt{k_n \log k_n}) \max\{ 1 , (k_n \log k_n)^4 \} \\ 
						 &\lesssim k_n^{9/2} (\log k_n)^{9/2} 
	\end{align*}
	for which we obtain $\sup_{g \in \FnOptDomain} P|g|^2 \le \sigma_n^2$ and $\sup_{g \in \FnOptDomain} P|g|^3 \le \sigma_n^2 b_n$ and $F_n \le b_n$. Note that $\frac{b_n}{\sigma_n} \lesssim k_n^{9/2} (\log k_n)^{9/2}$. 

	Let's now obtain a bound for $K_n$. Observe that using Lemma~\ref{lem:FnOptDomainProperties}, 
	\begin{align*}
		\nu_n \log \frac{A_n b_n}{\sigma_n} 
		&= \nu_n \log\frac{b_n}{\sigma_n} + \log K + \log V(\FnOptDomain) + V(\FnOptDomain) \log(16e) \\ 
		&\lesssim \BigOh( k_n^2 ) \BigOh\left( \log k_n + \log\log k_n \right) + \BigOh(1) 
		+ \BigOh( \log k_n) + \BigOh(k_n^2) \\ 
		&= \BigOh( k_n^2 \log k_n ) 
	\end{align*}
	and so 
	\begin{align*}
		\nu_n \left( 1 + \sqrt{1 + \log \frac{A_n b_n}{\sigma_n} } \right)^2
		&= 2\nu_n + 2\sqrt{\nu_n}\sqrt{\nu_n + \nu_n \log\frac{A_n b_n}{\sigma_n} } + \nu_n\log\frac{A_n b_n}{\sigma_n} \\
		&= \BigOh(k_n^2) + \sqrt{\BigOh(k_n^2)} \sqrt{ \BigOh(k_n^2) + \BigOh(k_n^2 \log k_n)} + \BigOh(k_n^2 \log k_n) \\
		&= \BigOh(k_n^2 \log k_n)
	\end{align*}
	This implies, 
	\begin{align*}
		K_n 
		&:= c \nu_n \max\left\{ \log n \,,\, \left( 1 + \sqrt{1 + \log \frac{A_n b_n}{\sigma_n} }  \right)^2  \right\} \\ 
		&\lesssim \max\{ \BigOh(k_n)^2 \log n \,,\, \BigOh( k_n^2 \log k_n)  \} \\ 
		&= \BigOh( k_n^2 \log n) 
	\end{align*}
	where we used that $k_n / n \to 0$. 

Thus by Theorem~\ref{thm:GaussianApproximationSuprema}, for the random variables $Z_n$ there exists a random variable $\widetilde{W}_n$ with which we have a mean-zero Gaussian process $\{ G_{P,n}(g) \}_{g \in \FnOptDomain} $ with covariance function 
	\begin{equation*}
		\E[ G_{P,n} (g_1) G_{P,n} (g_2) ] 
		= \frac{ \inner{ \Gamma^{1/2} \Gamma^\dagger h_1}{ \Gamma^{1/2} \Gamma^\dagger h_2 }}{ (\norm{\Gamma^{1/2}\Gamma^\dagger h_1} + a_n) (\norm{\Gamma^{1/2}\Gamma^\dagger h_2} + a_n)     },
		\quad \text{for all $g_1, g_2 \in \FnOptDomain$} 
	\end{equation*}
	such that $g_i = \frac{h_i^*}{ \sqrt{P|h_i^*|^2}}$ with $h_i^* \in \OptDomain^*$, $i = 1, 2$. Indeed, thanks to again to the Riesz representation theorem, we can identify this Gaussian process $G_{P,n}$ indexed by $\FnOptDomain$ with covariance function on the left hand side with a Gaussian process indexed by $\OptDomain$ having the covariance function on the right hand side. So with some abuse of notations, we can write $\widetilde{Z}_n = \sup_{g \in \FnOptDomain} G_{P,n} g = \sup_{h \in \OptDomain} G_{P,n} h$. Moreover, $Z_n$ and $\widetilde{Z}_n$ satisfy, for all $\gamma \in (0,1)$ 
	\begin{align*}
		&\abs{Z_n - \widetilde{Z}_n} \\
		&= \BigOhPee\left( \KeyCvgRateStepII \right) 
	\end{align*}
\end{proof}

We can finally summarize everything and put Steps~\ref{eq:ProofStepI} and \ref{eq:ProofStepII} together. 
\begin{thm}
	\label{thm:UnnormalizedLeungTamStatisticConvergence} 
	Define $\widetilde{Z}_n := \sup_{h \in \OptDomain} G_{P,n} h$ where $\{ G_{P,n}(h) \}_{h \in \OptDomain}$ is a mean zero Gaussian process on $\ell^\infty(\OptDomain)$ with covariance function \eqref{eq:SupZnGaussianApproximationCovFun}. Then for sufficiently large $n$, 
	\begin{equation*}
		\absBIG{ \sup_{h \in \OptDomain} \inner{  \frac{\sqrt{n}}{{\sigma_\varepsilon t_n(h)}} (\hat{\rho} - \widehat{\Pi}_{k_n}\rho) }{h}  - \widetilde{Z}_n } 
	\lesssim \BigOhPee\left( \KeyCvgRateStepI +  \KeyCvgRateStepIIsimplified \right) 
	\end{equation*}
\end{thm}

\begin{proof} 
	By \eqref{eq:RhoHatDecomposition}, Propositions~\ref{prop:SupTYSConvergesToZero} and \ref{prop:SupZnGaussianApproximation} with an arbitrarily fixed $\gamma \in (0,1)$ and using the notations therein,  
	\begin{align*}
		&\absBIG{ \sup_{h \in \OptDomain} \inner{  \frac{\sqrt{n}}{{\sigma_\varepsilon t_n(h)}} (\hat{\rho} - \widehat{\Pi}_{k_n}\rho) }{h}  - \widetilde{Z}_n }  \\ 
		&\le \sup_{h \in \OptDomain} \absBIG{ \frac{\sqrt{n}}{\sigma_\varepsilon t_n(h)} \inner{\mathcal{T}_n + \mathcal{S}_n + \mathcal{Y}_n}{h}  } 
		+ \absBIG{ \sup_{h \in \OptDomain} \frac{\sqrt{n}}{\sigma_\varepsilon t_n(h)} \inner{\mathcal{R}_n }{x} - \widetilde{Z}_n } \\ 
		&\lesssim \sqrt{n} \BigOhPee\left(  \frac{\KeyCvgRateStepI}{n^{1/2}} \right)  \\ 
		&\quad + \BigOhPee\left( \KeyCvgRateStepII \right) 
	\end{align*}
	The result follows by taking the higher order terms. 
\end{proof}
The main result of Theorem~\ref{thm:NormalizedLeungTamStatisticConvergence} displayed in the main text is thus simply Theorem~\ref{thm:UnnormalizedLeungTamStatisticConvergence} normalized by the appropriate rate.

%% file: ms.bbl
\begin{thebibliography}{26}
\newcommand{\enquote}[1]{``#1''}
\expandafter\ifx\csname natexlab\endcsname\relax\def\natexlab#1{#1}\fi

\bibitem[\protect\citeauthoryear{Bosq}{Bosq}{2012}]{bosq2012linear}
\textsc{Bosq, D.} (2012): \emph{Linear Processes in Function Spaces: Theory and
  Applications}, vol. 149, Springer Science \& Business Media.

\bibitem[\protect\citeauthoryear{Cai, Hall et~al.}{Cai
  et~al.}{2006}]{cai2006prediction}
\textsc{Cai, T.~T., P.~Hall, et~al.} (2006): \enquote{Prediction in functional
  linear regression,} \emph{The Annals of Statistics}, 34, 2159--2179.

\bibitem[\protect\citeauthoryear{Cardot, Ferraty, Mas, and Sarda}{Cardot
  et~al.}{2003}]{cardot2003testing}
\textsc{Cardot, H., F.~Ferraty, A.~Mas, and P.~Sarda} (2003): \enquote{Testing
  hypotheses in the functional linear model,} \emph{Scandinavian Journal of
  Statistics}, 30, 241--255.

\bibitem[\protect\citeauthoryear{Cardot, Ferraty, and Sarda}{Cardot
  et~al.}{1999}]{cardot1999functional}
\textsc{Cardot, H., F.~Ferraty, and P.~Sarda} (1999): \enquote{Functional
  linear model,} \emph{Statistics \& Probability Letters}, 45, 11--22.

\bibitem[\protect\citeauthoryear{Cardot, Mas, and Sarda}{Cardot
  et~al.}{2007}]{cardot2007clt}
\textsc{Cardot, H., A.~Mas, and P.~Sarda} (2007): \enquote{CLT in functional
  linear regression models,} \emph{Probability Theory and Related Fields}, 138,
  325--361.

\bibitem[\protect\citeauthoryear{Cardot and Sarda}{Cardot and
  Sarda}{2011}]{cardot2011functional}
\textsc{Cardot, H. and P.~Sarda} (2011): \enquote{Functional linear
  regression,} in \emph{The Oxford Handbook of Functional Data Analysis}.

\bibitem[\protect\citeauthoryear{Chernozhukov, Chetverikov, Kato
  et~al.}{Chernozhukov et~al.}{2014}]{chernozhukov2014gaussian}
\textsc{Chernozhukov, V., D.~Chetverikov, K.~Kato, et~al.} (2014):
  \enquote{Gaussian approximation of suprema of empirical processes,} \emph{The
  Annals of Statistics}, 42, 1564--1597.

\bibitem[\protect\citeauthoryear{Conway}{Conway}{1994}]{conway1994course}
\textsc{Conway, J.~B.} (1994): \emph{A Course in Functional Analysis},
  Springer, 2nd ed.

\bibitem[\protect\citeauthoryear{Crambes and Mas}{Crambes and
  Mas}{2013}]{crambes2013asymptotics}
\textsc{Crambes, C. and A.~Mas} (2013): \enquote{Asymptotics of prediction in
  functional linear regression with functional outputs,} \emph{Bernoulli}, 19,
  2627--2651.

\bibitem[\protect\citeauthoryear{Cuesta-Albertos, Garc{\'\i}a-Portugu{\'e}s,
  Febrero-Bande, and Gonz{\'a}lez-Manteiga}{Cuesta-Albertos
  et~al.}{2019}]{cuesta2019goodness}
\textsc{Cuesta-Albertos, J.~A., E.~Garc{\'\i}a-Portugu{\'e}s, M.~Febrero-Bande,
  and W.~Gonz{\'a}lez-Manteiga} (2019): \enquote{Goodness-of-fit tests for the
  functional linear model based on randomly projected empirical processes,}
  \emph{The Annals of Statistics}, 47, 439--467.

\bibitem[\protect\citeauthoryear{Dauxois, Pousse, and Romain}{Dauxois
  et~al.}{1982}]{dauxois1982asymptotic}
\textsc{Dauxois, J., A.~Pousse, and Y.~Romain} (1982): \enquote{Asymptotic
  Theory for the Principal Component Analysis of a Vector Random Function: Some
  Applications to Statistical Inference,} \emph{Journal of Multivariate
  Analysis}, 12, 136--154.

\bibitem[\protect\citeauthoryear{Goia and Vieu}{Goia and
  Vieu}{2016}]{goia2016introduction}
\textsc{Goia, A. and P.~Vieu} (2016): \enquote{An introduction to recent
  advances in high/infinite dimensional statistics,} .

\bibitem[\protect\citeauthoryear{Hilgert, Mas, Verzelen et~al.}{Hilgert
  et~al.}{2013}]{hilgert2013minimax}
\textsc{Hilgert, N., A.~Mas, N.~Verzelen, et~al.} (2013): \enquote{Minimax
  adaptive tests for the functional linear model,} \emph{Annals of Statistics},
  41, 838--869.

\bibitem[\protect\citeauthoryear{H{\"o}rmann, Kidzi{\'n}ski, and
  Hallin}{H{\"o}rmann et~al.}{2015}]{hormann2015dynamic}
\textsc{H{\"o}rmann, S., {\L}.~Kidzi{\'n}ski, and M.~Hallin} (2015):
  \enquote{Dynamic functional principal components,} \emph{Journal of the Royal
  Statistical Society: Series B: Statistical Methodology}, 319--348.

\bibitem[\protect\citeauthoryear{Horv{\'a}th and Kokoszka}{Horv{\'a}th and
  Kokoszka}{2012}]{horvath2012inference}
\textsc{Horv{\'a}th, L. and P.~Kokoszka} (2012): \emph{Inference for functional
  data with applications}, vol. 200, Springer Science \& Business Media.

\bibitem[\protect\citeauthoryear{Hsing and Eubank}{Hsing and
  Eubank}{2015}]{hsing2015theoretical}
\textsc{Hsing, T. and R.~Eubank} (2015): \emph{Theoretical Foundations of
  Functional Data Analysis, with an Introduction to Linear Operators}, vol.
  997, John Wiley \& Sons.

\bibitem[\protect\citeauthoryear{Kato}{Kato}{1995}]{kato1995perturbation}
\textsc{Kato, T.} (1995): \emph{Perturbation Theory for Linear Operators},
  Springer Science \& Business Media, 2nd ed.

\bibitem[\protect\citeauthoryear{Leung and Tam}{Leung and
  Tam}{2021}]{leung2021suppsmalluniform}
\textsc{Leung, R. C.~W. and Y.-M. Tam} (2021): \enquote{Supplement to ``A
  Small-Uniform Statistic for the Inference of Functional Linear
  Regressions'',} .

\bibitem[\protect\citeauthoryear{Panaretos, Tavakoli et~al.}{Panaretos
  et~al.}{2013}]{panaretos2013fourier}
\textsc{Panaretos, V.~M., S.~Tavakoli, et~al.} (2013): \enquote{Fourier
  analysis of stationary time series in function space,} \emph{The Annals of
  Statistics}, 41, 568--603.

\bibitem[\protect\citeauthoryear{Powell}{Powell}{1994}]{powell1994direct}
\textsc{Powell, M.~J.} (1994): \enquote{A direct search optimization method
  that models the objective and constraint functions by linear interpolation,}
  in \emph{Advances in optimization and numerical analysis}, Springer, 51--67.

\bibitem[\protect\citeauthoryear{Ramsay and Silverman}{Ramsay and
  Silverman}{2005}]{ramsay2005functional}
\textsc{Ramsay, J. and B.~W. Silverman} (2005): \emph{Functional Data
  Analysis}, Springer-Verlag New York, 2nd ed.

\bibitem[\protect\citeauthoryear{Runarsson and Yao}{Runarsson and
  Yao}{2005}]{runarsson2005search}
\textsc{Runarsson, T.~P. and X.~Yao} (2005): \enquote{Search biases in
  constrained evolutionary optimization,} \emph{IEEE Transactions on Systems,
  Man, and Cybernetics, Part C (Applications and Reviews)}, 35, 233--243.

\bibitem[\protect\citeauthoryear{Stancu-Minasian}{Stancu-Minasian}{2012}]{stancu2012fractional}
\textsc{Stancu-Minasian, I.~M.} (2012): \emph{Fractional programming: theory,
  methods and applications}, vol. 409, Springer Science \& Business Media.

\bibitem[\protect\citeauthoryear{van~der Vaart and Wellner}{van~der Vaart and
  Wellner}{1996}]{van1996weak}
\textsc{van~der Vaart, A.~W. and J.~A. Wellner} (1996): \emph{Weak Convergence
  and Empirical Processes: With Applications to Statistics}, Springer.

\bibitem[\protect\citeauthoryear{Wang, Chiou, and M{\"u}ller}{Wang
  et~al.}{2016}]{wang2016functional}
\textsc{Wang, J.-L., J.-M. Chiou, and H.-G. M{\"u}ller} (2016):
  \enquote{Functional Data Analysis,} \emph{Annual Review of Statistics and Its
  Application}, 3, 257--295.

\bibitem[\protect\citeauthoryear{Yao, M{\"u}ller, and Wang}{Yao
  et~al.}{2005}]{yao2005functional}
\textsc{Yao, F., H.-G. M{\"u}ller, and J.-L. Wang} (2005): \enquote{Functional
  linear regression analysis for longitudinal data,} \emph{The Annals of
  Statistics}, 2873--2903.

\end{thebibliography}
